\documentclass[EJP]{ejpecp} 

\usepackage[notcite,notref,color]{showkeys}

\usepackage{tikz}

\usepackage{changebar}
\usepackage{pdfsync}
\usetikzlibrary{decorations.pathmorphing,patterns}
\usetikzlibrary{calc,arrows}

\usepackage{pgf}

\usetikzlibrary{automata}
\usetikzlibrary{positioning}

\tikzset{
    state/.style={
           rectangle,
           rounded corners,
           draw=black, very thick,
           minimum height=2em,
           inner sep=2pt,
           text centered,
           },
}

\usepackage{xcolor}

\usepackage{changebar}
\usepackage{pdfsync}

\usepackage[normalem]{ulem}

\makeindex
\newcommand{\mc}[1]{{\mathcal #1}}
\newcommand{\mf}[1]{{\mathfrak #1}}

\newcommand{\bb}[1]{{\mathbb #1}}

\newcommand\bbZ{{\mathbb Z}}
\newcommand\al{\alpha}

\newcommand\llangle{\langle\!\langle}
\newcommand\rrangle{\rangle\!\rangle}

\newcommand\E{{\mathbb E}}
\newcommand\bbE{{\mathbb E}}

\newcommand\R{{\mathbb R}}
\newcommand\RR{{\mathbb R}}

\newcommand\T{{\mathbb I}}

\newcommand\bbP{{\mathbb P}}

\newcommand\Z{{\mathbb Z}}

\newcommand\eps{\epsilon}

\newcommand \ga{\gamma}

\newcommand \la{\lambda}
\newcommand \si{\sigma}
\newcommand{\bbT}{\mathbb T}

\newcommand\Om{\Omega}

\renewcommand{\ge}{\geqslant}
\renewcommand{\le}{\leqslant}
\newcommand{\dd  }{\mathrm{d}}

\renewcommand{\hat}{\widehat}
\renewcommand{\tilde}{\widetilde}
\renewcommand{\bar}{\overline}
\numberwithin{equation}{section}

\SHORTTITLE{Hydrodynamic limit for a chain with forcing}

\TITLE{Hydrodynamic limit for a chain with thermal and  mechanical boundary forces
\thanks{This work was partially supported by the grant 346300 for IMPAN from the Simons Foundation and the matching 2015-2019 Polish MNiSW fund. }\
\thanks{T.K. acknowledges the support of  NCN grant UMO2016/23/B/ST1/00492. 
S.O. acknowledges the support of the  ANR-15-CE40-0020-01 grant LSD,
M.S. thanks Labex CEMPI (ANR-11-LABX-0007-01), and the ANR grant MICMOV (ANR-19-CE40-0012) of the French National Research Agency (ANR)}}

\AUTHORS{%
  Tomasz Komorowski\footnote{ Institute of Mathematics, Polish Academy
  Of Sciences, Warsaw, Poland.
    \EMAIL{tkomorowski@impan.pl}}
  \and 
  Stefano Olla \footnote{ Universit\'e Paris-Dauphine, PSL Research University,
  CNRS, CEREMADE,
  75016 Paris, France  \BEMAIL{olla@ceremade.dauphine.fr}}
   \and
  Marielle Simon\footnote{Inria, Univ. Lille, CNRS, UMR 8524 -- Laboratoire Paul Painlev\'e, F-59000 Lille.
    \EMAIL{marielle.simon@inria.fr}}}
  
\KEYWORDS{Hydrodynamic limit;  harmonic chain; boundary conditions; Fourier-Wigner functions} 

\AMSSUBJ{82C70; 60K35} 

\SUBMITTED{April 18, 2020} 
\ACCEPTED{January 10, 2021} 

\VOLUME{0}
\YEAR{2020}
\PAPERNUM{0}
\DOI{10.1214/YY-TN}

\ABSTRACT{We prove the hydrodynamic limit for a one dimensional harmonic
  chain with a
  random flip of the momentum sign.
  The system is open and subject to two thermostats at the boundaries
  and to an external tension at one of the endpoints.
  Under a diffusive scaling of space-time, we prove that the empirical profiles of the two
  locally conserved quantities, the volume stretch and the energy,
  converge to the solution of a non-linear diffusive system of conservative partial differential equations. }

\begin{document}

\section{{Introduction}\label{sec:intro}}

The mathematical  derivation of the macroscopic evolution of the conserved quantities
of a physical system, from its microscopic dynamics, through a rescaling of space and time
(so called hydrodynamic limit) has been the subject of much research in the last 40 years
(cf.~\cite{KL} and references within). Although heuristic assumptions like \emph{local equilibrium} and
\emph{linear response} permit to formally derive the macroscopic equations \cite{spohn},
mathematical proofs are very difficult and most of the techniques used are based on relative
entropy methods (cf.~\cite{KL} and references within). Unfortunately,
in the diffusive scaling
when energy is one of the conserved quantities, relative entropy methods cannot be used: {since one would need sharp bounds on higher energy moments, which cannot be controlled by the relative entropy}.
In some situations a different approach, based on Wigner distributions, is effective in controlling the
macroscopic evolution of energy. This is the case for a chain of harmonic springs with a random flip
of sign of the velocities, provided with periodic boundary conditions, for which the total energy and the
total length of the system are the two conserved quantities,
and where the hydrodynamic limit has been proven in \cite{kos1}.

The purpose of the present article is to deal with the case when
microscopic mechanical forces and thermal heat baths acting on the boundaries {  are present}, 
and to determine macroscopic  boundary conditions for the hydrodynamic diffusive equations.
In the scaling limit the presence of  boundary conditions is challenging, as the action
of the forces and thermostats  become singular.
In \cite{bo2} the authors prove the existence and  uniqueness of the non-equilibrium stationary
state, even in the anharmonic case. The existence of the Green-Kubo formula for the
thermal conductivity is also proven in \cite{bo2}. However, it turns out very difficult to control the limit
properties as the size of the system becomes infinite (\textit{i.e.}~macroscopic).
In particular the rigorous proof of the Fourier law,
which states that the average energy current is inversely proportional
to the size of the system and proportional to the negative temperature difference of the thermostats, is still an open problem.
One of the main difficulties of this open dynamics is to prove that energy inside
the system remains proportional to its size
(a trivial fact for the periodic case where energy is globally conserved).
In fact the random flip of the velocity signs does not move the energy in the system
and the energy transport is entirely due to the hamiltonian part of the dynamics, that is very hard to control.

This difficulty forced us to consider a different energy conserving random dynamics,
where kinetic energy is exchanged between nearest neighbor particles in a continuous random
mechanism, see \cite{kos2}.
In this case the stochastic dynamics is also
responsible for energy transport. The non-equilibrium stationary
state (NESS) for this dynamics was already considered in \cite{bo1}, 
where the Fourier law was proven without external force,
\textit{i.e.}~in mechanical (but not thermal) equilibrium.
In \cite{kos2} the validity of the Fourier law for the NESS is extended
  also to the situation when  an external tension force is present
  (then the system is in both  mechanical and thermal non-equilibrium).
  Furthermore, the existence of both stationary macroscopic profiles for
the temperature and volume
stretch,  at least in some situations, are  established
in  \cite{kos2}.
In particular, the temperature profile has the interesting feature that
the stationary temperatures in the
bulk can be higher than at the boundaries, a general behavior
conjectured in the NESS for many systems \cite{olla19}.
Furthermore, because of the presence of other conservation laws,
the stationary energy current can have the same sign as the gradient
of the temperature -- {\color{cyan} a} phenomenon called 
 \emph{uphill diffusion phenomenon} in the literature.

Concerning the hydrodynamic limit, in the appendix section of  \cite{kos2}  we have 
formulated a heuristic argument, based on entropy production
estimates, that have not been proved there, and that substantiated the
validity of the macroscopic equations governing the dynamics in the case
of a random momentum exchange microscopic model, see Section 2.2 of
\cite{kos2}. Besides the aforementioned  entropy production
estimates, in order to obtain the hydrodynamic limit, one needs 
to establish also  the property of 
\emph{equipartition} of the random fluctuations of the mechanical and thermal
components of the microscopic energy density, which was postulated in
\cite[identity (A.46)]{kos2}. As we have pointed out in
\cite{kos2} this property seems to
be out of reach of the relative entropy method and some other approach
to resolve the difficulty is needed.  
In the present work we employ the Wigner distribution method to give a
rigorous prove of the hydrodynamic limit, for an  open system with a
random flip of momenta,
see Section \ref{sec:adiab-micr-dynam} below for its precise
formulation.

A crucial observation is the identity \eqref{energy-evol3ff} that
holds for the $L^2$ norm of the covariances of  random fluctuations
of momenta and stretches, which we obtain by careful analysis of time
evolution of the Fourier-Wigner functions defined in Section
\ref{sec:wave-function-1}. The last two terms in the right hand side
of \eqref{energy-evol3ff} correspond to the dissipation, due to the
stochastic dynamics in the bulk. The remaining two
terms describe the interaction between the fluctuation of the thermal and
mechanical components of the kinetic
energy  
at the boundary points and in the bulk of the system, respectively. In
order to control these terms we need to control the rate of damping of
the mechanical energy, which is done in  
Lemma \ref{lm010911-19}. These controls allow us to prove that the
$L^2$ norm of the covariances of  random fluctuations of momenta and
stretches, at the given time, grows with the logarithm of the size of the
system: this is the content of Proposition \ref{thm020411-19}. This in turn enables us to show, using again the properties of
the  Fourier-Wigner function
dynamics, the already mentioned
equipartition property, which is stated in Proposition \ref{equipartition} and proved in Section \ref{sec11}.

The next ingredient that is important in the hydrodynamic limit
argument is the linear bound, in the system size, for the relative
entropy of the chain, with respect to both the thermal equilibrium and local
equilibrium probability measures. We establish this bound, together
with some of its consequences,  in Section \ref{sec6} (see Proposition \ref{cor021211-19}). A crucial
property that allows us to control the entropy production, coming from
the action of the external force, is the estimate of the damping rate
of the time average of the  momentum expectation  at the
respective endpoint of the system obtained in Proposition \ref{lm010911-19a}.

As we have already mentioned the model we consider in the present work,
with the random flip of the sign of momenta,
is  more difficult to handle than the random momentum exchange one
investigated in \cite{kos2}, due to the fact that the
energy is not transported by  the
stochastic part of the dynamics. We believe that the method used in
the  present
paper can be also applied to that model. In addition, the assumption that the
forcing acting at the boundary is constant in time, is only made here to
simplify the already complicated arguments for the entropy bound
of Section \ref{sec7entropy} and  the momentum
damping estimates formulated in  Proposition  \ref{lm010911-19a} and
Lemma \ref{lm010911-19}. At the expense of increasing the volume of
the calculations, with some additional effort, one could extend
the results of the present paper to the case when the tension is a $C^1$ smooth function of
time.

  A proof of the Fourier law in the stationary state  remains an open
  problem for the random flip model.  
We hope that in the future we will also be able to extend the results of the present paper to the more
challenging case of the chain of anharmonic springs.

Finally, concerning the organization of the paper. The description of
the model and basic notation is presented in Section
\ref{sec:adiab-micr-dynam}. The formulation of the main result,
together with the auxiliary facts needed to carry out the proof are done
in Section \ref{sec3}. For a reader convenience we sketch the
structure of the main
argument in Section \ref{sec:equipartition}.
The proof of the hydrodynamic limit is carried out in Section
\ref{sec:form-deriv-equat}. It is contingent on a number of auxiliary
results that are shown throughout the remainder of the paper. Namely,
the estimates of the momentum and stretch averages are done in Section
\ref{sec66}, the energy production bounds are obtained in  Section
\ref{sec6}, while Section \ref{sec:wigner} is devoted to showing the
equipartition property. Finally, in the appendix sections we give the
proofs of quite technical estimates used throughout Section \ref{sec66}.


\section{Preliminaries}
\label{sec:adiab-micr-dynam}

\subsection{Open chain of oscillators}

For $n\ge1$  an integer we let  $\T_n := \{0,1,\dots,n\}$ and
  $\T_n^{\circ} := \{1,\dots,n-1\}$. The points $0$ and $n$ are the extremities of the chain. Let $\T := [0,1]$ be the continuous counterpart.
We suppose that the position and momentum of a harmonic oscillator at
site $x\in\T_n$ are denoted by $(q_x,p_x)\in\R^2$.
The interaction
between two particles situated at $x-1,x\in\T_n^{\circ}$ is described by the quadratic
potential energy \[V(q_x-q_{x-1}):=\tfrac12(q_x-q_{x-1})^2.\]  
At the boundaries the system
is connected to two Langevin heat baths at temperatures $T_0:=T_-$ and
$T_n:=T_+$.
 We also assume that a force (tension) of constant value
$\bar\tau_+ \in \RR$ is acting on the utmost right point $x=n$. Since the system is \emph{unpinned}, the absolute positions $q_x$ 
do not have precise meaning, and the dynamics depends only on the
interparticle {\em stretch} 
\[
r_x := q_x - q_{x-1} \quad \text{ for } x=1,\ldots,n, 
\]and by convention throughout the paper we set $r_0:=0$.
The configurations are then described by
\begin{equation}
  \label{eq:1}
  (\mathbf r, \mathbf p) = (r_1, \dots, r_n, p_0, \dots, p_n) \in \Omega_n:=\R^{n}\times\R^{n+1}. 
\end{equation}
The total energy of the chain is defined by the Hamiltonian:
\begin{equation}
\label{Hn}
\mathcal{H}_n (\mathbf r, \mathbf p):=\sum_{x\in\T_n} {\cal E}_x (\mathbf r, \mathbf p),
\end{equation}
where the microscopic energy density is given by
\begin{equation}
\label{Ex}
{\cal E}_x (\mathbf r, \mathbf p):=  \frac{p_x^2}{2}+V(r_x)= \frac{p_x^2}2 + \frac{r_{x}^2}2 ,\quad x\in \T_n .
\end{equation}
Finally, we assume that for each $x \in \T_n$ the momentum $p_x$ can be flipped, at a random
exponential time with intensity $\ga n^2$, to $-p_x$, with
$\ga>0$. 

Therefore, the microscopic dynamics of
the process $\{(\mathbf r(t), \mathbf p(t))\}_{t\ge0}$
describing the total chain is  given in the bulk by 
\begin{equation} 
\label{eq:qdynamicsbulk}
\begin{aligned}
    \dd   r_x(t) &= n^2 \left(p_x(t)- p_{x-1}(t)\right) \; \dd   t,\qquad \qquad \qquad  \qquad \qquad x\in \{1, \dots, n\},\\
    \dd   p_x(t) &= n^2 \left(r_{x+1}(t)- r_x(t)\right) \; \dd   t -  2p_x(t^-)
    \; \dd   \mathcal N_x(\gamma n^2t), \; \quad x\in\T_n^{\circ},\end{aligned} \end{equation} and at the boundaries by
     \begin{align}\label{eq:qdynamicsbound} 
r_0(t)&\equiv0,\notag\\
     \dd   p_0(t) &= n^2  r_1(t)  \; \dd   t 
     - 2p_0(t^-)
        \; \dd   \mathcal N_0(\gamma n^2t) - \tilde\gamma n^2 p_0 \; \dd   t + n\sqrt{2\tilde\gamma T_-} \dd w_0(t),  \\
 \dd   p_n(t) &= -n^2  r_n(t) \; \dd   t  + n^2 \bar\tau_+ \; \dd t - 2p_n(t^-)
    \; \dd   \mathcal N_n(\gamma n^2t) - \tilde\gamma n^2 p_n \; \dd   t + n\sqrt{2 \tilde\gamma T_+} \dd w_n(t), \notag
  \end{align}
where $w_0(t)$ and $w_n(t)$ are independent
standard Wiener processes and $\mathcal N_x(t)$, $x\in\T_n$ are
independent of them i.i.d.~Poisson processes of intensity
$1$. Besides, $\tilde\gamma>0$ regulates the intensity of the Langevin thermostats. All processes are given
over some probability space $(\Sigma,{\cal F},\bbP)$. The factor
$n^2$ appearing in the temporal scaling comes from the fact that $t$,
used in the equations above,
is the \emph{macroscopic}  time, and the \emph{microscopic} time scale is the diffusive one. 

We assume that the initial data is random,
distributed according to the  probability distribution $\mu_n$ over $\Om_n$.
We denote by $\bbP_n:=\mu_n\otimes \bbP$ (resp.~$\E_n$) the product probability distribution over $\Om_n\times\Sigma$ (resp.~its expectation).


\bigskip

Equivalently, the generator of this dynamics is given by 
\begin{equation}
  \label{eq:generator}
  L := n^2 \left(A + \gamma S + \tilde\gamma \tilde S \right),
\end{equation}
where
\begin{equation}
  \label{eq:A}
  A: = \sum_{x=1}^n (p_x - p_{x-1}) \partial_{r_x} + \sum_{x=1}^{n-1} (r_{x+1} - r_x) \partial_{p_x}
  + r_1 \partial_{p_0} + \left(\bar\tau_+(t) - r_n\right) \partial_{p_n}
\end{equation}
and
\begin{equation}
  \label{eq:S}
  S F (\mathbf{r},\mathbf{p}): = \sum_{x=0}^n
  \left(F(\mathbf{r},\mathbf{p}^x) -
    F(\mathbf{r},\mathbf{p})\right)
\end{equation}
for any $C^2$--class smooth function $F$.
Here $\mathbf{p}^x$ is
the momentum configuration obtained from $\mathbf{p}$ with $p_x$
replaced by $-p_x$.
Finally, the generator of the Langevin heat bath at the boundary points equals:
\begin{equation}
  \label{eq:tS}
  \tilde S = 
  \sum_{x=0,n} \left( T_x \partial_{p_x}^2 - p_x \partial_{p_x} \right) , \qquad \text{with} \quad T_0:=T_-,
  \quad T_n := T_+. 
\end{equation}

\subsection{Notations}
\label{sec:some-notat-elem}

We collect here notations and conventions that we use throughout the paper.

$\bullet$ Given an integrable function $G:\T \to \mathbb C$, its Fourier transform is defined by
\begin{equation}
  \label{eq:F1}
  \mathcal F G (\eta) := \int_{\T} G(u) e^{-2i\pi u \eta} \; du, \qquad \eta\in \Z.
\end{equation}
If $G\in L^2(\T)$, then the inverse Fourier transform reads as
\begin{equation}
  \label{eq:F2}
  G(u) = \sum_{\eta\in \Z} e^{2i\pi u \eta} \mathcal F G (\eta) ,\qquad u\in\T,
\end{equation}
where the sum converges in the $L^2$ sense.

 $\bullet$ Given a sequence $\{f_x,\,x \in \T_n\}$, its Fourier transform is given by
\begin{equation}
  \label{eq:F3}
  \widehat f(k) = \sum_{x\in \T_n} f_x e^{-2i\pi x k} ,\qquad k\in \widehat\T_n := \big\{0, \tfrac 1{n+1}, \dots, \tfrac {n}{n+1}\big\}.
\end{equation}
Reciprocally, for any $\hat f: \hat\T_n \to \mathbb{C}$, the inverse Fourier transform reads 
\begin{equation}
  \label{eq:F4}
  f_x = \underset{{k\in \widehat{\T}_n}}{{\hat\sum}}\hat f(k)  e^{2i\pi x k},\qquad  x\in\T_n,
\end{equation}
where we use the following short notation 
\begin{equation}
\label{012001-20}
\underset{{k\in \widehat{\T}_n}}{{\hat\sum}}:=\frac 1{n+1} \sum_{k\in \widehat\T_n}
\end{equation}
for the averaged summation over frequencies $k\in\hat \T_n$. 
The Parseval identity can be then expressed  as follows
\begin{equation}
\label{parseval}
 \underset{{k\in \widehat{\T}_n}}{{\hat\sum}} \widehat f(k) \widehat g^\star(k) =
 \sum_{x\in \T_n}f_xg_x^\star ,\qquad f,g:\T_n\to\mathbb{C}.
\end{equation}
For a given function $f$ we adopt the convention
\begin{equation}
\label{iota}
f^+(k):=f(k)\quad\mbox{and}\quad f^-(k):=f^\star(-k),\qquad
k\in\hat\T_n.
\end{equation}
According to our notation, given a configuration 
\begin{equation*}
  (\mathbf r, \mathbf p) = (r_1, \dots, r_n, p_0, \dots, p_n) \in \Omega_n:=\R^{n}\times\R^{n+1}
\end{equation*}
we let, for any $k \in\hat\T_n$,
$$
\widehat r(k):=\sum_{x\in \T_n}r_x e^{-2\pi i x k},\qquad \widehat p(k):=\sum_{x\in \T_n}p_x e^{-2\pi i x k},
$$
recalling the convention $r_0:=0$.
Since the configuration components are real valued, the corresponding Fourier transforms have the property:
\begin{equation}
  \label{eq:w1}
  {\widehat p}^{\star}(k) =  {\widehat p}(-k), \qquad  {\widehat r}^{\star}(k) =  {\widehat r}(-k).
\end{equation}

 $\bullet$ For a  function  $G: \T \to \mathbb{C}$, we define three discrete  approximations: of the function itself, of its gradient and Laplacian, respectively by
 \begin{align}
\label{disc-approx}
G_x&:=G(\tfrac x n), \qquad \; \; \quad \qquad \qquad \qquad \qquad x\in\T_n,\notag\\
(\nabla_nG)_x&:=n\big(G(\tfrac{x+1}n)-G(\tfrac x n)\big), \qquad\qquad\qquad x\in\{0,\ldots,n-1\},\\
 (\Delta_nG)_x &:=n^2\big(G(\tfrac{x+1}n)+G(\tfrac{x-1}n)-2G(\tfrac x
   n)\big),\; \; \;\;\; x\in\T_n^{\circ}.\notag
\end{align}


$\bullet$ Having two families of functions $f_i,g_i:A\to\R$, $i\in I$, where $I$,
$A$ are some sets
we write $
f_i\lesssim g_i,$ $ i\in I $
if there exists $C>0$ such that
$$
f_i(a)\le C g_i(a),\qquad \text{for any } i\in I,\,a\in A.
$$
If both $f_i\lesssim g_i, i\in I$ and $g_i\lesssim f_i, i\in
I$, then we shall write
$
f_i\approx g_i, i\in I.
$

\section{Hydrodynamic limits: statements of the main results}

\label{sec3}

In this section we state our main results, given below in Theorem \ref{hydro1}, Theorem \ref{hydro2} and Theorem \ref{main-result}. Before that,  {we state} our first assumption on the initial probability distribution of the configurations.

Suppose that $T>0$. {We first introduce} $\nu_T (\dd{\bf r},\dd{\bf p})$ {which is defined as} the product
  Gaussian measure on $\Omega_n$
 of {zero} average and variance $T>0$ given by
\begin{equation} \label{eq:nuT}
 \nu_T (\dd{\bf r},\dd{\bf p}) : =\frac{e^{-\mc E_0/T}}{\sqrt{2\pi T}} \dd p_0 \;
  \prod_{x=1}^n \frac{e^{-\mc E_x/T}}{\sqrt{2\pi T}} \dd p_x \dd r_x.
\end{equation} Let  $\mu_n(t)$ be the probability law on $\Omega_n$ of the configurations $(\mathbf{r}(t),\mathbf{p}(t))$ and let $f_n(t,\mathbf{r},\mathbf{p})$ be the density of the measure $\mu_n(t)$ with respect to $\nu_T$.   

 We define  the linear interpolation between the inverse boundary temperatures $T_-^{-1}$
and $T_+^{-1}$ by
\begin{equation}
\label{beta}
\beta(u) := \left(T_+^{-1} - T_-^{-1}\right) u +  T_-^{-1},\qquad u\in\T.
\end{equation}
 Recall the definition of its discrete approximation:
  $\beta_x := \beta(x/n)$, $x\in\T_n$. 
Let $\tilde \nu$ be the corresponding inhomogeneous product measure with tension $\bar\tau_+$:
\begin{equation}
  \label{tilde-nu}
  \tilde \nu (\dd \mathbf{r},\dd \mathbf{p}) := \frac{e^{-\beta_0 p_0^2/2}}{\sqrt{2\pi \beta_0^{-1}}} \dd p_0 \; \prod_{x=1}^n
  \exp\big\{-\beta_x\left( \mathcal E_x - \bar\tau_+ r_x\right) -
    \mathcal G(\beta_x,\bar\tau_+)\big\} \dd r_x\dd p_x,
\end{equation}
where the Gibbs potential is
\begin{equation}
  \label{eq:16}
  \mathcal G(\beta,\tau):= \log\int_{\mathbb R^2} e^{-\frac{\beta}{2} (r^2+p^2)+ \beta \tau r} \dd p \dd r = \frac 12 \beta \tau^2
  + \frac12\log \left(2\pi \beta^{-1}\right),
\end{equation}
for $ \beta>0,\,\tau\in\mathbb R. $
Consider the 
density 
\begin{equation}
\tilde f_n(t) := f_n(t) \frac { \dd \nu_{T}}{ \dd \tilde{\nu}}. \label{eq:tildef}
\end{equation}
and define
the relative entropy
\begin{equation}
  \label{eq:7}
 \tilde{\mathbf{H}}_n(t) := \int_{\Om_n} \tilde f_n(t) \log\tilde f_n(t) \dd \tilde{\nu}.
\end{equation}
In the whole paper we assume \begin{equation}
\label{eq:ass0}
\tilde f_n(0)\in C^2(\Omega_n) \qquad \text{and} \qquad \tilde{\mathbf{H}}_n(0) \lesssim n, \qquad n \geqslant 1.
\end{equation}

\subsection{Empirical distributions of the averages}

We are interested in the evolution of the \emph{microscopic profiles}
of stretch, momentum, and energy, which we now define. For any
$n\ge 1$, $t\ge0$ and $x \in \T_n$, let
\begin{equation}
\label{bar-r}
\bar r_x^{(n)}(t):=\bbE_n \big[r_x(t)\big], \qquad 
\bar p_x^{(n)}(t):=\bbE_n
\big[p_x(t)\big], 
\qquad \bar{\cal E}_x^{(n)}(t):=\bbE_n \big[{\cal E}_x(t)\big].
\end{equation}
Moreover, we denote by $\widehat{\bar r}^{(n)}(t,k)$, $\widehat{\bar
  p}^{(n)}(t,k)$, with $k\in\hat\T_n$,  the Fourier
transforms of the first two fields defined in \eqref{bar-r}.
We shall make the following hypothesis: 


\begin{assumption}\label{ass1} We assume 
\begin{enumerate} \item an energy bound on the initial data:
\begin{equation}
\label{E0}
\sup_{n\ge1}\frac{1}{n+1}\sum_{x\in \T_n}\bar{\cal E}_x^{(n)}(0)<+\infty \; ; 
\end{equation}
\item a uniform bound satisfied by the spectrum of the initial averages:
  \begin{equation}
\label{spec-bound}
\sup_{n\ge1}\Big(\sup_{k\in \hat \T_n}|\hat{\bar
      r}^{(n)}(0,k)|
+\sup_{k\in \hat \T_n}|\hat
{\bar p}^{(n)}(0,k)|\Big)<+\infty.
\end{equation}
\end{enumerate}\end{assumption}

{ The assumption \eqref{spec-bound} prevents the
  average of the initial profiles to
  concentrate at any particular mode $k$, as $n\to\infty$. This is a natural regularity
  condition on the macroscopic profiles of volume stretch and momenta. 
  }

\subsection{Convergence of the average stretch and momentum}

\label{sec-stretch}

In order to state the convergence results for the profiles, we
extend the definition \eqref{bar-r} to profiles on
$\T$, as follows: for any $u \in \T$ and $x\in\T_n$ let 
\begin{equation}\label{En0} 
\begin{cases} \bar r^{(n)}(t,u)&=\bbE_n \big[r_x(t)\big], \\
 \bar p^{(n)}(t,u)&=\bbE_n \big[p_x(t)\big],\\
\bar{\cal E}^{(n)}(t,u)&=\bbE_n \big[{\cal E}_x(t)\big], \end{cases} \qquad \mbox{ if } u\in \big[ \tfrac x{n+1}, \tfrac{x+1}{n+1} \big). 
\end{equation} 
Let $r(t,u)$ be the solution of the following  partial differential equation
    \begin{align}
        \partial_t
        r(t,u)&=\frac{1}{2\gamma}\partial^2_{uu}r(t,u), \qquad  (t,u)\in \R_+\times \T,\label{eq:linear} 
      \end{align}
with the  boundary and initial conditions:
\begin{equation}
  \label{eq:bc0}
  \begin{split}
   & r(t,0) = 0
   , \qquad  r(t,1) = \bar\tau_+,\\
&r(0,u)=r_0(u)
  \end{split}
\end{equation} for any $(t,u)\in \R_+\times \T$.
To guarantee the {existence and uniqueness of a $C^2$--}regular solution of the above problem we
assume that 
\begin{equation}
\label{reg}
r_0\in C^2(\T)\quad\mbox{ and }\quad r_0(1)=\bar \tau_+.
\end{equation}
Let $p_0 \in C(\T)$ be an initial momentum profile. Our first result can be formulated as follows.
\begin{theorem}[Convergence of the stretch and momentum profiles]
\label{hydro1}
Assume that the initial distribution of the
  stretch and momentum weakly converges to $r_0(\cdot),p_0(\cdot)$ introduced above,
  i.e.~for any test function $G\in C^\infty(\T)$ we have
\begin{equation}
\label{weak-r0-p0} \begin{aligned}
&\lim_{n\to+\infty}\frac{1}{n+1}\sum_{x\in \T_n}\bar
r_x^{(n)}(0)G_x=\int_{\T}r_0(u)G(u)\dd u,\\
& \lim_{n\to+\infty}\frac{1}{n+1}\sum_{x\in \T_n}\bar
p_x^{(n)}(0)G_x=\int_{\T}p_0(u)G(u)\dd u.
\end{aligned} \end{equation}
Then, under Assumption \ref{ass1},
for any $t>0$ the following holds:
\begin{equation}
\label{hydro-conv}
\lim_{n\to+\infty}\bar r^{(n)}(t,\cdot)=r(t,\cdot)
\end{equation}
weakly in $L^2(\T)$, where $r(\cdot)$ is the solution of
\eqref{eq:linear}--\eqref{eq:bc0}. In addition,
we have
\begin{equation}
\label{hydro-conv1}
\lim_{n\to+\infty}\int_0^t\big\|\bar p^{(n)}(s,\cdot)\big\|^2_{L^2(\T)}\dd s=0.
\end{equation}
\end{theorem}

The proof of this theorem is given in Section \ref{sec11.2.1}.

It is not difficult to prove (see Section \ref{sec:equipartition}
below) that, under the same assumptions as in Theorem \ref{hydro1}, for
each $t>0$ the sequence of the squares of the mean stretches
$\{[\bar r^{(n)}]^2 (\cdot)\}_{n\ge1}$ -- the
\emph{mechanical energy} density -- is sequentially $\star-$weakly
compact in $\left(L^1([0,t];C(\T))\right)^\star$. 
However, in order to characterize its convergence
one needs substantial extra work, and this is why we state it as an additional important  result.


Let $C_{0}^2(\T)$ be  the class of
$C^2$ functions  on $\T$ such that $G(0) = G(1) =
0$.

\begin{theorem}[Convergence of the mechanical energy profile]
\label{hydro2}
Assume that  Assumption \ref{ass1} holds. 
Then, for any  test function $G\in
L^1([0,t];C_{0}^2(\T))$ we have
\begin{equation}
\label{060401-20x}
\lim_{n\to+\infty}\int_0^{t}\dd s\int_\T\big(\bar
r^{(n)} (s,u)\big)^2G(s,u)\dd u= \int_0^{t}\dd s\int_\T
r^2 (s,u)G(s,u)\dd u,
\end{equation}
where $r(s,u)$ is the solution of \eqref{eq:linear}--\eqref{eq:bc0}.
\end{theorem}
The proof of this theorem is contained in Section \ref{sec13.3.3}.

\subsection{Convergence of the  energy density  average}
\label{sec:macr-equat}
Our last result concerns the microscopic energy profile. 
 To obtain the  convergence of $\bar{\cal
  E}_x^{(n)}(t)$ for $t>0$, we add an assumption on the
\emph{fluctuating part} of the initial data distribution.
For any $x \in \T_n$, let
\begin{equation}
\label{fluct}
\begin{aligned} \tilde r_x^{(n)}(t) &:= r_x(t)-\bar r_x^{(n)}(t),\\  \tilde p_x^{(n)}(t) &:= p_x(t)-\bar p_x^{(n)}(t).\end{aligned}
\end{equation}
Similarly as before, let  $\widehat{\tilde
  r}^{(n)}(t,k)$, $\widehat{\tilde p}^{(n)}(t,k)$ be the Fourier
transforms of the fields defined in \eqref{fluct}.
We shall assume the following hypothesis on the covariance
of the stretch and momentum fluctuations.
\begin{assumption} \label{ass3}The following correlations sums are finite:
\begin{equation}
\label{2-mfluct}
\begin{split} & \sup_{n\ge 1}\frac 1{n+1} 
\sum_{x,x'\in\T_n}\left(\mathbb E_n\left[ \tilde p_x^{(n)}(0)
  \tilde p_{x'}^{(n)}(0)\right] \right)^2   <+\infty\\
  &
\sup_{n\ge 1}\frac 1{n+1} 
\sum_{x,x'\in\T_n}\left(\mathbb E_n\left[ \tilde r_x^{(n)}(0)
  \tilde r_{x'}^{(n)}(0)\right] \right)^2  <+\infty \\
  & \sup_{n\ge 1}\frac 1{n+1} 
\sum_{x,x'\in\T_n} \left(\mathbb E_n\left[ \tilde p_x^{(n)}(0)
  \tilde r_{x'}^{(n)}(0)\right] \right)^2 <+\infty.
\end{split} \end{equation}
\end{assumption}

{ Assumption \ref{ass3} constitutes a rather weak hypothesis about
  the spatial
  decorrelation for the initial distribution of the stretch and
  momentum fields. Under this assumption
  the convergence in probability for the local (weak)
  law of large numbers  for the initial profiles is guaranteed.
  For example under the local Gibbs measures 
\begin{equation*}
  \prod_{x=1}^n
  \exp\bigg\{-\frac{\beta_x}{2}\left( (p_x - \bar p_x^{(n)}(0))^2
    +(r_x - \bar r_x^{(n)}(0))^2\right) - \log(2\pi\beta_x) \bigg\} \dd r_x\dd p_x,
\end{equation*}
each of the fields $\big\{\tilde p_x^{(n)}(0)\big\}_{x\in\bbZ}$, $\big\{\tilde
r_x^{(n)}(0)\big\}_{x\in\bbZ}$  decorrelates for any $x\neq x'$ and the
first two conditions of
\eqref{2-mfluct} are trivially satisfied, with the bound provided by
$\inf_x\beta_x^{-1}$. The third one also holds, as then
$\bbE_x\big[\tilde p_x^{(n)}(0) \tilde r_{x'}^{(n)}(0)\big]=0$ for
all $x,x'$.
}

Let $e(t,u)$ be the solution of the initial-boundary value problem for
the inhomogeneus heat equation
    \begin{equation}
 \partial_t e (t,u) =  \frac 1{4\ga}  \partial^2_{uu}
   \left\{  e(t,u) +
       \frac 12 r^2(t,u)\right\},\qquad
        (t,u)\in \R_+\times \T ,\label{eq:linear2}\end{equation}
        with the boundary and initial conditions \begin{equation}  \label{eq:bc2}\begin{split}
&e (t,0) = T_-, \qquad \ e(t,1) = T_+ + \frac12\bar\tau_+^2,\\
 & e(0,u)= e_0(u), 
\end{split} \end{equation}  for any $(t,u)\in \R_+\times \T$.
Here $r(t,u)$ is the solution of  \eqref{eq:linear}--\eqref{eq:bc0}, and $e_0$ is non-negative.
Our principal result concerning the convergence of the energy
functional is contained in the following {theorem}:
\begin{theorem}[Convergence of the total energy profile]
\label{main-result}
Assume  that the initial distribution of the energy converges weakly to some $e_0 \in C(\T)$, \emph{i.e.}~for any $G\in C^\infty(\T)$ we have:
\begin{equation}
\label{weak-E0}
\lim_{n\to+\infty}\frac{1}{n+1}\sum_{x\in \T_n}\bar{\cal
  E}_x^{(n)}(0)G_x=\int_{\T}e_0(u)G(u)\dd u.
\end{equation}
Then, under Assumptions \ref{ass1}   and  \ref{ass3},
for any $t>0$ and  any test function $G\in
L^1([0,t]; C_{0}^2(\T))$ we have
\begin{equation}
\label{060401-20y}
\lim_{n\to+\infty}\int_0^{t}\dd s\int_\T\bar{\cal 
E}^{(n)} (s,u)G(s,u)\dd u= \int_0^{t}\dd s\int_\T
e (s,u)G(s,u)\dd u,\quad 
\end{equation}
where $e(\cdot)$ is the solution of \eqref{eq:linear2}--\eqref{eq:bc2}.
\end{theorem}
The proof of this theorem is presented in Section \ref{sec13.2.2}.

\section{Sketches of proof and Equipartition of energy}

In this section we present some essential intermediate results which
will be used to prove the convergence theorems, and which are
consequences of the various assumptions made. We have decided to expose them in an independent section in order to emphasize the main steps of the proofs, and to highlight the role of our hypotheses.

\label{sec:equipartition}

\subsection{Consequences of Assumption \ref{ass1}}

\subsubsection{The boundary terms}
An important feature of our model is the presence of $\bar \tau_+ \neq 0$. A significant part of the work consists in estimating boundary terms. We first state in this section the crucial bounds that we are able to get, under Assumption \ref{ass1}, and which concern the extremity points $x=0$ and $x=n$. One of the most important result is the following:
\begin{proposition}
\label{prop012812-19}
Under Assumption \ref{ass1}, for any $t>0$ we have
\begin{align}
\label{022812-19a}
\int_0^t\big|\bar p_{0}(s)-\bar p_{n}(s)\big|^2\dd s \lesssim \frac{1}{n^2},\qquad n\ge1,
\end{align}
and
\begin{align}
\label{022812-19b}
\int_0^t\big|\bar p_{0}(s)+\bar p_{n}(s)\big|^2\dd s \lesssim  \frac{\log
  (n+1)}{n^2},\qquad n\ge1.
\end{align}
\end{proposition}

This  result is proved in Section  \ref{sec7}.  
Another consequence of Assumption \ref{ass1} is the following one-point estimate, which uses the previous result \eqref{022812-19a}, but allows us to get a sharper bound:

\begin{proposition}
\label{lm010911-19a} 
Under Assumption \ref{ass1}, 
for any $t\ge0$ we have
\begin{equation}
\label{052412-19a}
\left|\int_0^t\bar p_{0}^{(n)}(s)\dd s\right|\lesssim \frac{1}{n} \qquad \text{and} \qquad \left|\int_0^t\bar p_{n}^{(n)}(s)\dd s\right|\lesssim \frac{1}{n},\qquad
 n\ge1.
\end{equation}
\end{proposition}

This proposition is proved in Section  
\ref{sec8}.

\begin{remark}
In fact, in the whole paper, only the second estimate in \eqref{052412-19a} will be used. However, in its proof, the first estimate comes freely.
\end{remark}

\subsubsection{Estimates in the bulk} 
Provided with a good control on the boundaries, one can then obtain several estimates in the bulk of the chain. Two of them are used several times in the argument, and can be proved independently of each other. The first one is

\begin{proposition}[$L^2$ bound on average momenta and stretches]
\label{cor013112-19z}
Under Assumption \ref{ass1}, for any $t>0$
\begin{equation}
\label{053112-19z}
\frac{1}{n+1}\sup_{s\in[0,t]}\sum_{x\in\T_n}\Big\{\big(\bar r_x^{(n)} (s)\big)^2+\big(\bar
p_x^{(n)}(s)\big)^2\Big\}\lesssim 1,\qquad n\ge1.
\end{equation}
In addition,
for any $t>0$ we have
\begin{equation}
\label{010301-20}
n\sum_{x\in\T_n}\int_0^t \big(\bar
p_x^{(n)}(s)\big)^2\dd s\lesssim 1,\qquad n\ge1.
\end{equation}
\end{proposition} 
The proof of Proposition \ref{cor013112-19z} {is given} in Section
\ref{sec6.1} below, and makes use of Proposition \ref{lm010911-19a}.  Here we formulate some of its immediate consequences: 
\begin{itemize}
\item thanks to \eqref{053112-19z} we conclude  
 that   for each $t>0$ the sequence of the averages $\{\bar
  r^{(n)}(t)\}_{n\ge1}$ is bounded in $L^2(\T)$, thus it is weakly
compact. Therefore, to prove Theorem \ref{hydro1} one needs to identify the limit in \eqref{hydro-conv}, which is carried out in Section
\ref{sec11.2.1},
\item the second equality \eqref{hydro-conv1} of Theorem \ref{hydro1} simply follows from  \eqref{010301-20},
\item finally, the estimate  \eqref{053112-19z}  implies in particular that
\begin{equation}
\label{010203y}
\sup_{n\ge1}\sup_{s\in[0,t]}\big\|[\bar
r^{(n)}]^2 (s,\cdot)\big\|_{L^1(\T)}\lesssim 1.
\end{equation}
Therefore, we conclude that, for each $t>0$ the sequence $\{[\bar r^{(n)}]^2 (\cdot)\}_{n\ge1}$ is sequentially $\star-$weakly
compact in $\left(L^1([0,t];C(\T))\right)^\star$, as claimed. This is the first step to prove Theorem \ref{hydro2}.
\end{itemize}

The second important estimate focuses on the microscopic energy averages and is formulated as follows:
\begin{proposition}[Energy bound] Under Assumption \ref{ass1}, 
\label{prop022111-19}
for any $t\ge0$ we have
\begin{equation}
\label{022111-19}
\sup_{s\in[0,t],\,n\ge1}\bigg\{\frac{1}{n+1}\sum_{x\in\T_n}\bar{\cal
    E}_x^{(n)}(s)\bigg\}<+\infty.
\end{equation}
\end{proposition}

This estimate is proved in Section \ref{sec7entropy}, using a bound on the \emph{entropy production}, given in Proposition \ref{thm-entropy-production} below. 
Thanks to Proposition   \ref{prop022111-19}   the sequence $\{ \bar{\cal  E}^{(n)}(\cdot)\}_{n\ge1}$ 
is sequentially $\star-$weakly
compact in $(L^1([0,t];C_{0}^2(\T)))^\star$ for each $t>0$.  Therefore, to prove Theorem \ref{main-result}, one needs to identify the limit.  This identification requires the extra Assumption \ref{ass3}. 

\subsection{Consequence of Assumption \ref{ass3}}
The proof of Theorem \ref{main-result} is based on a 
\emph{mechanical and thermal energy equipartition} result given as follows:
\begin{proposition}[Equipartition of energy]
\label{equipartition}
Under Assumptions \ref{ass1} and \ref{ass3},  for any complex valued test function $
G\in C_0^\infty([0,+\infty)\times \T\times\bbT)$ we have
 \begin{equation}
\label{042701-20}
\lim_{n\to+\infty}\int_0^{t} \frac 1{n}  \sum_{x\in\T_n}
                                                 G_x(s)
   {\mathbb E}_n\Big[\big(\tilde r_x^{(n)}(s) \big)^2-\big(\tilde p_x^{(n)}(s) \big)^2\Big]
                                                 \dd s=0.
\end{equation}
\end{proposition}
The proof of this result is presented in Section \ref{sec:wigner} (cf.~conclusion in Section \ref{sec11}), and uses some of the results above, namely Proposition \ref{prop012812-19} and Proposition \ref{prop022111-19}.



\section{Proofs of the hydrodynamic limit theorems}

\label{sec:form-deriv-equat}

In the present section we show  Theorems \ref{hydro1}, \ref{hydro2}
and \ref{main-result} announced in Section \ref{sec3}. 
The proof of the latter is contingent on several intermediate  results: 
\begin{itemize}
\item first of all, to prove the three results we need specific boundary estimates which will be all stated in Section \ref{ssec:boundary} (see Lemma \ref{lem:bound2}), and which are byproducts of Proposition \ref{lm010911-19a}, Proposition \ref{cor013112-19z} and Proposition \ref{prop022111-19};
\item the proof of Theorem \ref{hydro2} requires moreover Lemma \ref{lm010301-20}, which is based on a detailed analysis of
the average dynamics $(\bar r_x^{(n)},\bar p_x^{(n)})_{x\in\T_n}$ (that
will be carried out in Section \ref{sec66});
\item finally, to show Theorem \ref{main-result} we need: first, a uniform $L^2$ bound on the averages of momentum, see Lemma \ref{lm010911-19} below. The latter will be proved in Section \ref{sec7bis}, as a consequence of Proposition \ref{prop012812-19}; second, the equipartition  result for the fluctuation of the potential and kinetic energy
of the chain, which has already been stated in Section \ref{sec:equipartition}, see Proposition \ref{equipartition}.
\end{itemize}


\subsection{Treatment of boundary terms} \label{ssec:boundary}
First of all, the conservation of the energy gives the following microscopic identity: 
\begin{equation}
  \label{eq:en-evol}
  n^{-2} L \mathcal E_x(t) =  j_{x-1,x} (t)-
  j_{x,x+1} (t), \qquad x\in\T_n^{\rm o},
\end{equation}
where \begin{equation}
\label{jx}
 j_{x,x+1} (t):= j_{x,x+1} ({\bf r}(t), {\bf p}(t)), \qquad \text{with }  j_{x,x+1}({\bf r}, {\bf p}) :=-p_xr_{x+1}, 
\end{equation}
are the \emph{microscopic currents}.
At the boundaries we have
\begin{align}
  \label{eq:en-evol-}
 n^{-2} L \mathcal E_0(t) &= - j_{0,1} (t)+ \tilde\gamma  \left(T_- - p_0^2(t) \right),\\
  \label{eq:en-evol+}
   n^{-2} L \mathcal E_n(t) &= j_{n-1,n}(t) + \bar\tau_+ p_n(t) + \tilde\gamma \left(T_+ - p_n^2(t) \right).
\end{align}
One can see that boundaries play an important role. Before proving the hydrodynamic limit results, one needs to understand very precisely how boundary variables behave. This is why we start with collecting here all the estimates that  are essential in the following argument. Their proofs require quite some work, and for the sake of clarity this will be postponed to Section \ref{sec:conseque}. 


\begin{lemma}[Boundary estimates]\label{lem:bound2}

  The following  holds: for any $t\ge0$
  \begin{itemize}
  \item[(i)] \emph{(Momentum correlations)}
  \begin{equation}
  \lim_{n\to\infty} \mathbb{E}_n\left[ \int_0^t  p_{0}(s) p_{1}(s)\dd s\right] = 0, \qquad   \lim_{n\to\infty} \mathbb{E}_n\left[ \int_0^t  p_{n-1}(s) p_{n}(s)\dd s\right] = 0. \label{ex-p0p1}
\end{equation}

\item[(ii)] \emph{(Boundary correlations)}
  \begin{equation}
 \bigg| \mathbb{E}_n\left[ \int_0^t p_{0}(s) r_1(s) \dd s\right] \bigg|\lesssim
    \frac{1}{\sqrt{n}}, 
 \qquad \bigg|\mathbb{E}_n\left[ \int_0^t p_{n}(s) r_n(s) \dd s\right]\bigg|\lesssim
    \frac{1}{\sqrt{n}} , \qquad n\ge1. \label{ex-2-r}
\end{equation}
\item[(iii)] \emph{(Boundary stretches)}
  \begin{equation}
  \left|\mathbb{E}_n\left[ \int_0^t r_1(s)\dd s\right]\right|\lesssim
    \frac{1}{\sqrt{n}},   \qquad \left|\mathbb{E}_n\left[ \int_0^t \big(r_n(s) -\bar\tau_+\big)\dd s\right]\right|\lesssim
    \frac{1}{\sqrt{n}},  \qquad n\ge 1.  \label{ex-2-l1}
\end{equation}
\item[(iv)] \emph{(Boundary temperatures, part I)} for any $ n\ge1$
  \begin{equation}
    \label{eq:ex-1}
 \left| \mathbb{E}_n\left[   \int_0^t  \left(T_--p_{0}^2(s) \right)\dd
  s\right] \right|\lesssim \frac{1}{\sqrt n},\qquad    \left|\mathbb{E}_n\left[   \int_0^t  \left(T_+-p_{n}^2(s) \right)\dd
  s\right]\right| \lesssim \frac{1}{\sqrt n}.  
 \end{equation}
\item[(v)] \emph{(Mechanical energy at the boundaries, part I)}
  \begin{equation}
    \label{eq:13}
    \mathbb{E}_n\left[ \int_0^t \left(r_1^2(s) + r_n^2(s)\right)
      \dd s\right] \lesssim 1, \qquad n\ge1.
  \end{equation}
\item[(vi)] \emph{(Boundary currents)}
  \begin{equation}
    \label{eq:11}
      \lim_{n\to\infty}    \mathbb{E}\left[   \int_0^t  j_{0,1}(s)
      \dd s\right]  = 0, \qquad \lim_{n\to\infty}    \mathbb{E}\left[   \int_0^t  j_{n-1,n}(s)
      \dd s\right]  = 0.
  \end{equation}
\item[(vii)] \emph{(Mechanical energy at the boundaries, part II) :}  at the left boundary point
  \begin{equation}
  \left|\mathbb{E}_n\left[ \int_0^t \Big(r_1^2(s)-T_-\Big) \dd s\right]\right|\lesssim
    \frac{1}{\sqrt{n}}, \qquad n\ge 1  \label{ex-2-l11t}
\end{equation}
and at the right boundary point
\begin{equation}   
 \left|\mathbb{E}_n\left[ \int_0^t \Big(r_n^2(s)-\bar\tau_+^2-T_+\Big)\dd s\right] \right|\lesssim
    \frac{1}{\sqrt{n}}, \qquad n\ge 1. \label{ex-2-r11t}
\end{equation}
\item[(viii)] \emph{(Boundary temperatures, part II)}
\begin{equation}
    \label{eq:ex-1bis}
  \sum_{x=0,n}T_x\mathbb{E}_n\left[   \int_0^t  \left(T_x-p_{x}^2(s) \right)\dd
  s\right] \lesssim \frac{1}{n},\qquad
\,n\ge1.
 \end{equation}
 \end{itemize}
\end{lemma}




Provided with all the previous results which have been stated (but not proved yet), we are ready to prove Theorem \ref{hydro1} and \ref{main-result}. Before that, in order to make the presentation unequivocal, {we present} in Figure \ref{fig} a diagram with the previous statements, and the sections where they will be proved into parentheses.  

\begin{center}
\begin{figure}[h!]

\begin{tikzpicture}[->,>=stealth']

 \node[state] (Prop41) 
 { \begin{tabular}{c} \textbf{Proposition} \ref{prop012812-19}  \\ (Section \ref{sec7}) \end{tabular} };
  
 \node[state,    	
  yshift=2cm, 		
  right of=Prop41, 	
  node distance=3cm, 	
  anchor=center] (Cor42) 	
 {%
 \begin{tabular}{c} 	
  \textbf{Proposition} \ref{lm010911-19a}\\
  (Section \ref{sec8})
 \end{tabular}
 };
 
 \node[state,
  right of=Prop41, 	
  node distance=7cm, 	
  yshift=-3cm,
  anchor=center] (Prop46) 
 {%
 \begin{tabular}{c}
  \textbf{Proposition} \ref{equipartition} \\
(Section \ref{sec:wigner})
 \end{tabular}
 };

 \node[state,
  right of=Cor42,
  yshift=2cm,
  node distance=4cm,
  anchor=center] (Prop44) 
 {%
 \begin{tabular}{c}
  \textbf{Proposition} \ref{cor013112-19z}\\
 (Section \ref{sec6.1})
 \end{tabular}
 };
 
  \node[state,
  right of=Cor42,
  yshift=-2cm,
  node distance=4cm,
  anchor=center] (Prop45) 
 {%
 \begin{tabular}{c}
  \textbf{Proposition} \ref{prop022111-19}\\
 (Section \ref{sec7entropy})
 \end{tabular}
 };
 
   \node[state,
  right of=Cor42,
  node distance=8cm,
  anchor=center] (Lem51) 
 {%
 \begin{tabular}{c}
  \textbf{Lemma} \ref{lem:bound2}\\
 (Section \ref{sec:conseque})
 \end{tabular}
 };

 \path[dashed] (Prop41) 	edge[bend left=20]   (Cor42)
 (Prop41)     	edge[bend right=20]  (Prop46)
(Cor42)  edge[bend left=20]  (Prop44)
(Cor42) edge[bend right=20]  (Prop45)
(Prop44) edge[bend left=20]  (Lem51)
(Prop45) edge[bend right=20] (Lem51)
(Cor42) edge (Lem51)
(Prop45) edge (Prop46)
(Lem51) edge[bend left=20] (Prop46) 
 ;

\end{tikzpicture}
\caption{An arrow from A to B means that A is used to prove B, but is not necessarily a direct implication.}
\label{fig}
\end{figure}
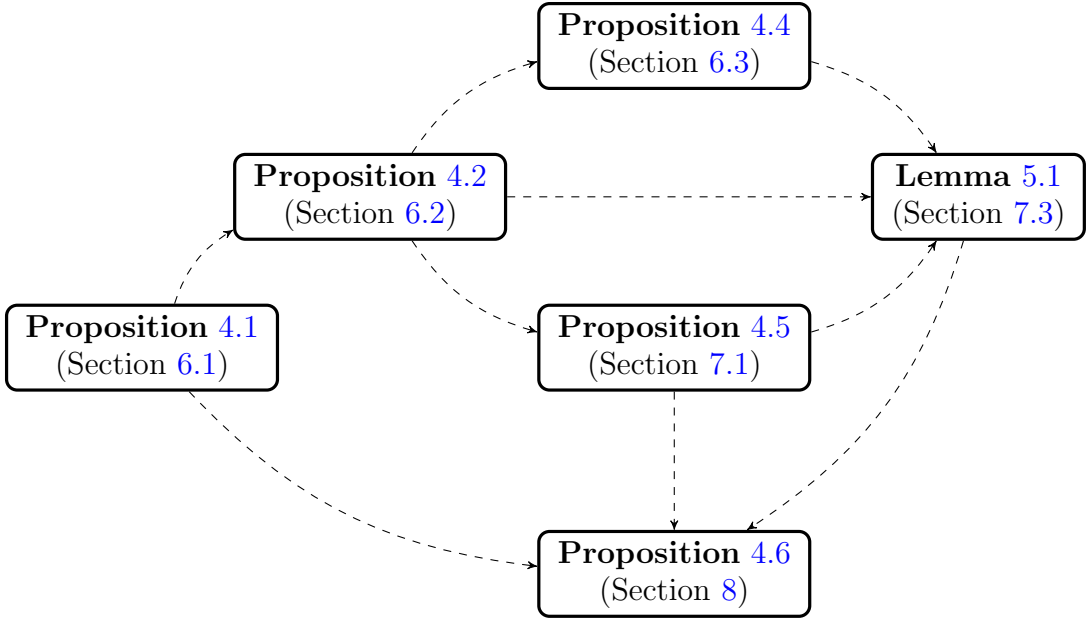
\end{center}

\subsection{Proof of Theorem \ref{hydro1}}

\label{sec11.2.1}

Recall the diffusive equation \eqref{eq:linear}, which can be formulated in a weak form as:
\begin{multline}
  \label{eq:wlinear}
  \int_0^1 \; G(u) \big(r(t,u) - r(0,u)\big) \dd u \\ 
= \frac {1}{{2\gamma}} \int_0^t \dd s
  \int_0^1    \; G''(u) r(s,u) \dd u -  \frac {1}{{2\gamma}}  G'(1)
  \bar\tau_+ t
  ,\qquad t\ge 0,
\end{multline}
for any  test function $G\in C^2_{0}(\T)$. Existence and uniqueness of such weak solutions in an appropriate
space of integrable functions are standard, {see for instance \cite[Section 2.3]{strauss}}.

By the microscopic evolution equations, see 
\eqref{eq:qdynamicsbulk}--\eqref{eq:qdynamicsbound}, we have  (cf.~\eqref{disc-approx})
\begin{multline}
 \bbE_n\left[  \frac{1}{n+1}\sum_{x\in\T_n}
    G_x\left(r_x(t) -  r_x(0)\right) \right] 
    = \frac{n^2}{n+1}\bbE_n\left[ \int_0^t \dd s  \sum_{x=1}^n 
   G_x \left(p_x(s) -  p_{x-1}(s)\right) \right] \\
 = \bbE_n\left[  \int_0^t \dd s \left\{ -\sum_{x=1}^{n-1} 
    (\nabla_n G)_x\;p_x(s) - (n+1)
    G_1 p_0(s)\right\}\right]+o_n(1).
  \label{eq:timevol1} \end{multline}
As usual, the symbol $o_n(1)$ denotes an expression that {tends to zero} with $n\to+\infty$.  The dynamics of the averages $(\bar{\mathbf{r}}(t),\bar{\mathbf{p}}(t))$ is easy to deduce from the evolution equations (see also 
\eqref{eq:pdyn-bar} where it is detailed). We can therefore rewrite the right hand side of \eqref{eq:timevol1} as
\begin{align}
  &
 \bbE_n\left[ -\int_0^t  \dd s  \left\{ \sum_{x=1}^{n-1} \frac{1}{2\gamma} (\nabla_n
      G)_x\left(r_{x+1}(s) - r_x(s)\right) +\frac {1}{2\gamma+ \tilde\gamma} (\nabla_n
      G)_0 r_1(s) \right\}\right] \label{eq:timevol+b} 
\\
&+\bbE_n\left[ \frac{1}{2\gamma n^2} \sum_{x=1}^{n-1} (\nabla_n
      G)_x\left(p_x(t) - p_x(0)\right) + \frac
      {1}{(2\gamma +\tilde \gamma)n^2}   
     (\nabla_n
      G)_0\left(p_0(t) - p_0(0)\right) \right]  +o_n(1).  \notag
\end{align}
Since $G$ is smooth we have $
\lim_{n\to+\infty}\sup_{x\in \T_n}\big|(\nabla_nG)_x- G'(x)\big|=0.
$
Using this and Proposition \ref{cor013112-19z} one shows that the
second expression in
\eqref{eq:timevol+b} converges to $0$,
leaving as the only possible significant the first term. 
Summing by parts and recalling  that $G(0)=0$, it can be rewritten as
\begin{multline}\label{eq:timevol3}
\bbE_n\left[ \int_0^t  \frac{1}{2\gamma}\left\{\frac{1}{n+1} \sum_{x=2}^{n-1}
      (\Delta_n G)_x\; r_x(s) - (\nabla_n G)_{n-1}\; r_n(s) \right\} \dd s\right]\\
      - \left(\frac {1}{2\gamma+ \tilde\gamma}  
      (\nabla_n G)_0 -\frac{1}{2\gamma}  (\nabla_n G)_1 \right)
   \bbE_n\left[\int_0^t r_1(s) \dd s\right].    
\end{multline}
Therefore, we need to understand the macroscopic behavior of the boundary strech variables, which is done thanks to Lemma \ref{lem:bound2}: from \eqref{ex-2-l1} we conclude that the second term {tends to zero}, as $n\to+\infty$.
Using again \eqref{ex-2-l1} but for the right boundary we infer that
\eqref{eq:timevol3} can be written as
\begin{equation}\label{eq:timevol31}
  \begin{split}
\int_0^t \dd s\left\{ \frac{1}{2\gamma}\int_0^1G''(u) \bar r^{(n)}(s,u)\dd u- G'(1) \bar\tau_+
\right\}
+o_n(t),
  \end{split}
\end{equation}
 where $\lim_{n\to+\infty} \sup_{s\in[0,t]}o_n(s)=0$.
Thanks to Proposition \ref{cor013112-19z} we know that for a given $t_*>0$
\begin{equation}
\label{010203}
\sup_{n\ge1}\sup_{s\in[0,t_*]}\big\|\bar
r^{(n)}(s,\cdot)\big\|_{L^2(\T)}\lesssim 1.
\end{equation}
The above means, in particular  that   the sequence $\{\bar
  r^{(n)}(\cdot)\}_{n\ge1}$ is bounded in the space $
L^\infty([0,t_*];L^2(\T))$. As this space is dual to the separable
Banach space $L^1([0,t_*];L^2(\T))$, the sequence $\{\bar
  r^{(n)}(\cdot)\}_{n\ge1}$ is $\star$-weakly
sequentially compact. Suppose that  $
r\in L^\infty([0,t_*];L^2(\T))$ is its $\star$-weakly limiting point. Any limiting point of the
sequence satisfies  
\eqref{eq:wlinear}, which shows that has to be unique and as a result $\{\bar
  r^{(n)}(\cdot)\}_{n\ge1}$ is $\star$-weakly
convergent to $
r\in L^\infty([0,t_*];L^2(\T))$, the solution to \eqref{eq:linear}--\eqref{eq:bc0}.

\subsection{Proof of Theorem \ref{hydro2}}

\label{sec13.3.3}

The following estimate shall be crucial in our
subsequent argument.
\begin{lemma}
\label{lm010301-20}
For any $t>0$ we have
\begin{equation}
\label{060301-20}
n\sum_{x=0}^{n-1}\int_0^t\big(\bar r_{x+1}^{(n)}(s)-\bar
r_x^{(n)}(s)\big)^2\;\dd s\lesssim 1,\qquad n\ge1.
\end{equation}
\end{lemma}
The proof of the lemma uses Proposition \ref{lm010911-19a}, and  is postponed to Section
\ref{sec6.5}.

 Define    $\bar r^{(n)}_{\rm int}:[0,+\infty)\times\T\to\R$ as the function obtained by
 the piecewise linear interpolation between the nodal points $(x/(n+1),\bar r_x)$,
 $x=0,\ldots,n+1$.  
 Here we let $\bar r_{n+1}:=\bar r_n$. As a consequence of Lemma \ref{lm010301-20} above we obtain 
 
 \begin{lemma}
\label{lm020401-20}
For any $t \ge 0$ we have
\begin{equation}
\label{030401-20}
\sup_{n\ge1}\int_0^{t}\big\| \bar r^{(n)}_{\rm int}(s,\cdot)\big\|_{H^1(\T)}^2\dd
s=\mathfrak{h}(t)<+\infty,
\end{equation}
where $H^1(\T)$ is the $H^1$ Sobolev norm: 
$
\|F\|_{H^1(\T)}^2:=\|F\|_{L^2(\T)}^2+ \|F'\|_{L^2(\T)}^2.
$ Moreover,
\begin{equation}
\label{090401-20z}
\lim_{n\to+\infty}\sup_{u\in\T}\left|\int_0^t \bar r^{(n)}_{\rm int}(s,u)\dd s-\int_0^t  r(s,u)\dd s\right|=0.
\end{equation}
\end{lemma}
 
 \begin{proof} 
 It is easy
 to see that
\begin{equation}
\label{tbrn}
\big\|\bar r^{(n)}_{\rm int}(t,\cdot)-\bar r^{(n)}(t,\cdot)\big\|_{L^2(\T)}^2=\frac{1}{3(n+1)}\sum_{x=0}^{n-1}\big(\bar r_{x+1}(t)-\bar r_x(t)\big)^2,
 \qquad n\ge1.
\end{equation}
Thanks to \eqref{060301-20} 
we obtain  \eqref{030401-20}. 
Using  \eqref{tbrn} we also get
\begin{equation}
\label{020401-20xx}
\lim_{n\to+\infty}\int_0^t\big\|\bar r^{(n)}_{\rm int}(s,\cdot)-\bar r^{(n)}(s,\cdot)\big\|_{L^2(\T)}^2 \;\dd s=0,\qquad t>0.
\end{equation}
From the proof of Theorem \ref{hydro1} given in Section \ref{sec11.2.1} we know
that the sequence
$\int_0^t \bar r^{(n)}_{\rm int}(s,u)\dd s$ weakly converges in
$L^2(\T)$ to $\int_0^t  r(s,u)\dd s$. From \eqref{030401-20}
and the compactness of Sobolev embedding into $C(\T)$ in dimension
$1$ we conclude \eqref{090401-20z}.
\end{proof}

Thanks to \eqref{030401-20} we know that
for any $t_*>0$ we have
\begin{equation}
\label{010203x}
\sup_{s\in[0,t_*]}\big\|[\bar
r^{(n)}_{\rm int}]^2 (s,\cdot)\big\|_{L^1(\T)}\lesssim 1,\qquad n\ge 1.
\end{equation}
The above implies that  the sequence $\{[\bar
  r^{(n)}_{\rm int}]^2 (\cdot)\}_{n\ge1}$ is sequentially $\star-$weakly
compact in $\left(L^1([0,t_*];C(\T))\right)^\star$. One can choose a
subsequence, that for convenience sake we denote by the same symbol, which is $\star-$weakly
convergent in any $\left(L^1([0,t_*];C(\T))\right)^\star$, $t_*>0$.
We prove now that for any  $G\in
L^1([0,t_*];C^2_{0}(\T))$ we have
\begin{equation}
\label{060401-20}
\lim_{n\to+\infty}\int_0^{t_*}\dd t\int_\T\big(\bar
r^{(n)}_{\rm int} (t,u)\big)^2G(t,u)\dd u= \int_0^{t_*}\dd t\int_\T
r^2 (t,u)G(t,u)\dd u,
\end{equation}
where $r(\cdot)$ is the solution of \eqref{eq:linear}--\eqref{eq:bc0}. By a density
argument is suffices only to consider functions of the form
$G(t,u)=\mathbf{1}_{[0,t_*)}(t)G(u)$, where $G\in C^2_{0}(\T)$, $t_*>0$.
To prove \eqref{060401-20} it suffices to show that
\begin{equation}
\label{060401-20bis}
\lim_{n\to+\infty}\frac{1}{n+1}\int_0^{t_*}\dd t\sum_{x\in\T_n}\left\{[\bar
r^{(n)}_x (t)]^2-r^2\left(t,\tfrac{x}{n+1}\right)\right\}G_x=0.
\end{equation}
Let $M\ge 1$ be an integer, that shall be specified later on, and 
$t_\ell:=\ell t_*/M$, for $\ell=0,\ldots,M$. 
The expression under the limit in \eqref{060401-20bis}
can be rewritten as $B_n^1(M)+B_n^2(M)+o_n(1)$, where $o_n(1)\to0$, as $n\to+\infty$, and
\begin{align}
B_n^1(M)& :=\frac{1}{n+1}\sum_{\ell=0}^{M-1}\sum_{x\in\T_n}G_x\bigg\{  \bar
r^{(n)}_x (t_\ell)\int_{t_\ell}^{t_{\ell+1}} \bar
r^{(n)}_x (t) \dd t  -r\left(t_\ell,\tfrac{x}{n+1}\right)
  \int_{t_\ell}^{t_{\ell+1}}r\left(t,\tfrac{x}{n+1}\right) \dd t\bigg\}\notag\\
B_n^2(M)&:=\frac{1}{n+1}\sum_{\ell=0}^{M-1}\sum_{x\in\T_n}G_x\int_{t_\ell}^{t_{\ell+1}}  \bar
r^{(n)}_x (t) \Big\{\int_{t_\ell}^t  \frac{\dd}{\dd s}\bar
r^{(n)}_x (s)\dd s \Big\}  \dd t \notag\\
&
=\frac{n^2}{n+1}\sum_{\ell=0}^{M-1}\sum_{x\in\T_n}G_x \int_{t_\ell}^{t_{\ell+1}}\bar
r^{(n)}_x (t) \int_{t_\ell}^t\Big(\bar
p^{(n)}_x (s)-\bar  p^{(n)}_{x-1} (s)\Big) \dd s  \dd t  .
\label{B-n}
\end{align}
The last equality follows from \eqref{eq:qdynamicsbulk}--\eqref{eq:qdynamicsbound}. 
In what follows we prove that
\begin{equation}
\label{010502-20}
\lim_{M\to+\infty}\limsup_{n\to+\infty}|B_n^j(M)|=0,\qquad j=1,2.
\end{equation}
Summing by parts in the utmost right hand side of \eqref{B-n} we
obtain 
$$
B_n^2(M)=\sum_{j=1}^3B_{n,j}^2(M),
$$
where
\begin{align*}
B_{n,1}^2&:=\frac{n^2}{n+1}\sum_{\ell=0}^{M-1}\sum_{x=1}^{n-1} \int_{t_\ell}^{t_{\ell+1}} \big(G_x\bar
r^{(n)}_x (t) -G_{x+1}\bar
r^{(n)}_{x+1} (t)\big)\int_{t_\ell}^t  \bar
p^{(n)}_x (s)\dd s  \dd t,  \\
B_{n,2}^2&:=\frac{n^2}{n+1}\sum_{\ell=0}^{M-1} \int_{t_\ell}^{t_{\ell+1}} G_n\; \bar
r^{(n)}_n (t) \left\{\int_{t_\ell}^t  \bar
p^{(n)}_n (s) \dd s  \right\} \dd t, \\
B_{n,3}^2&:=-\frac{n^2}{n+1}\sum_{\ell=0}^{M-1} \int_{t_\ell}^{t_{\ell+1}} G_1\; \bar
r^{(n)}_1 (t)\left\{ \int_{t_\ell}^t  \bar
p^{(n)}_0 (s) \dd s \right\} \dd t.
\end{align*}
We have $B_{n,1}^2=B_{n,1,1}^2+B_{n,1,2}^2$, where
\begin{align*}
B_{n,1,1}^2&:=-\frac{n}{n+1}\sum_{\ell=0}^{M-1}\sum_{x=1}^{n-1} \int_{t_\ell}^{t_{\ell+1}} \bar
r^{(n)}_x (t) (\nabla_nG)_x \; \left\{\int_{t_\ell}^t\bar
p^{(n)}_x (s)  \dd s\right\} \dd t ,\\
 B_{n,1,2}^2&:=\frac{n^2}{n+1}\sum_{\ell=0}^{M-1}\sum_{x=1}^{n-1} \int_{t_\ell}^{t_{\ell+1}} G_{x+1}\big(\bar
 r^{(n)}_x (t)-\bar
 r^{(n)}_{x+1} (t)\big)\; \left\{\int_{t_\ell}^t\bar
 p^{(n)}_x (s)\dd s  \right\}\dd t.
\end{align*}
By the Cauchy-Schwarz inequality we  bound $|B_{n,1,2}^2|$ from above by
\begin{align*}
& \frac{n^2\|G\|_{\infty}}{n+1}\left\{ \sum_{x=1}^{n-1}\int_{0}^{t_*}\big(\bar
r^{(n)}_{x+1} (t)-\bar
r^{(n)}_x (t)\big)^2 \dd t\right\}^{1/2}\left\{ \sum_{\ell=0}^{M-1}\sum_{x=1}^{n-1}\int_{t_\ell}^{t_{\ell+1}}\left\{\int_{t_\ell}^t\bar
p^{(n)}_x (s)\dd s \right\}^2 \dd t \right\}^{1/2}\\
&
\lesssim \left\{ (n+1)\sum_{x=1}^{n-1}\int_{0}^{t_*}\big(\bar
r^{(n)}_{x+1} (t)-\bar
r^{(n)}_x (t)\big)^2 \dd t\right\}^{1/2}\left\{ n\sum_{\ell=0}^{M-1}\sum_{x=1}^{n-1}\int_{t_\ell}^{t_{\ell+1}}\left\{\int_{t_\ell}^t\bar
p^{(n)}_x (s)\dd s \right\}^2 \dd t \right\}^{1/2}\\
&
\lesssim \left\{ n\sum_{\ell=0}^{M-1}\sum_{x=1}^{n-1}\int_{t_\ell}^{t_{\ell+1}}(t-t_\ell) \int_{t_\ell}^t
\big(\bar p^{(n)}_x (s)\big)^2\dd s  \dd t \right\}^{1/2} \lesssim \left\{\frac{n}{M^2} \sum_{x=1}^{n-1} \int_0^t \big(\bar p^{(n)}_x (s)\big)^2\dd s \right\}^{1/2} 
\end{align*}
by virtue of Lemma \ref{lm010301-20}. Using Proposition \ref{cor013112-19z}  (estimate
\eqref{010301-20}), we conclude  $|B_{n,1,2}^2| \lesssim 1/M$ and $\lim_{M\to+\infty}\limsup_{n\to+\infty}|B_{n,1,2}^2|=0$.
The argument for $|B_{n,1,1}^2|$ is
analogous. As a result we conclude $\lim_{M\to+\infty}\limsup_{n\to+\infty}|B_{n,1}^2|=0$.

Concerning $B_{n,2}^2$ we write
\begin{align*}
B_{n,2}^2=-\frac{n}{n+1}\sum_{\ell=0}^{M-1} \int_{t_\ell}^{t_{\ell+1}}(\nabla_nG)_n\; \bar
r^{(n)}_n (t) \left\{\int_{t_\ell}^t\bar
p^{(n)}_n (s)\dd s \right\} \dd t.
\end{align*}
Therefore,  we have
\begin{align*}
|B_{n,2}^2| & \lesssim\sum_{\ell=0}^{M-1} \int_{t_\ell}^{t_{\ell+1}}\big|\bar
r^{(n)}_n (t)\big| \; \left\{\int_{t_\ell}^t\big|\bar
p^{(n)}_n (s)\big|\dd s \right\}  \dd t
\\ & \lesssim  \left\{\int_{0}^{t_*}|\bar
r^{(n)}_n (t)| \dd t \right\}\; \left\{ t^\star \int_{0}^{t_*}\big(\bar
p^{(n)}_n (s)\big)^2 \dd s\right\}^{1/2} \lesssim \frac{1}{\sqrt n},
\end{align*}
by virtue of Lemma \ref{lem:bound2}--\eqref{eq:13} and Proposition \ref{cor013112-19z} (estimate
\eqref{010301-20}).
We conclude therefore that $\lim_{M\to+\infty}\limsup_{n\to+\infty}|B_{n,2}^2|=0$.
An analogous argument shows that also
$\lim_{M\to+\infty}\limsup_{n\to+\infty}|B_{n,3}^2|=0$. Thus,
\eqref{010502-20} holds for $j=2$.


We have $B^1_{n}(M)=
B_{n,1}^1(M)+ B_{n,2}^1(M)$, where
\begin{align*}
&
B_{n,1}^1(M):=\frac{1}{n+1}\sum_{\ell=0}^{M-1}\sum_{x\in\T_n}G_x\bar
r^{(n)}_x (t_\ell) 
\int_{t_\ell}^{t_{\ell+1}}\left[r \left(t_{\ell},\tfrac{x}{n+1}\right) -r\left(t,\tfrac{x}{n+1}\right) \right]\dd t,\\
&
B_{n,2}^1(M) :=\frac{1}{n+1}\sum_{\ell=0}^{M-1}\sum_{x\in\T_n}G_x\bar
r^{(n)}_x (t_\ell) \int_{t_\ell}^{t_{\ell+1}}\Big[\bar
r^{(n)}_x (t)- r \left(t_{\ell},\tfrac{x}{n+1}\right)\Big]\dd t.
\end{align*}
where $r$ is the solution of \eqref{eq:linear}--\eqref{eq:bc0}. 
By the regularity of the
 $r(t,u)$, Lemma \ref{lm020401-20} and estimate \eqref{010203x} we can
 easily conclude that
$\lim_{M\to+\infty}\limsup_{n\to+\infty}|B_{n,j}^1|=0$, $j=1,2$. Thus,
\eqref{010502-20} holds also for $j=1$, which ends the proof of
Theorem \ref{hydro2}.

\subsection{Proof of Theorem  \ref{main-result}}

\label{sec13.2.2}
 Concerning equation \eqref{eq:linear2}--\eqref{eq:bc2}, its weak
  formulation is as follows:
for any   test function $G\in L^1([0,+\infty);C_{0}^2(\T))$  which is compactly supported,
we have
\begin{align}
\label{022308-19}
 0= & \;  
 \int_\T  G(0,u) e_0(u)\dd u   +
 \int_0^{+\infty}  \int_\T   \Big(\partial_sG(s,u) + \frac1{4\ga}
    \partial_u^2G(s,u) \Big)\; e(s,u)\dd s\dd u  \notag \\
&
+ \frac {1}{8\ga}\int_0^{+\infty}  \int_\T \partial_u^2 G(s,u)
   r^2(s,u) \dd s \dd u  \notag \\
   & - \frac{1}{4\ga}\int_0^{+\infty}  \left(\partial_uG(s,1) \left( T_+
       + \bar\tau_+^2\right)
     -T_-\;\partial_uG(s,0) \right)\dd s.
  \end{align}
Given a non-negative initial data $e_0\in L^1(\T)$ and
the macroscopic stretch $ r(\cdot,\cdot)$ (determined via \eqref{eq:wlinear}) one can easily
    show that the respective weak formulation of the boundary value
    problem for a linear heat equation, resulting from \eqref{022308-19},
    admits a unique measure valued solution.
 
Recall that the  averaged energy density function $\bar{ \mathcal E}^{(n)}(t,u)$
has been defined in \eqref{En0}. 
It is easy to see, thanks to Proposition \ref{prop022111-19}, that for any $t_*>0$ 
we have
$$
\mathbf{E}(t_*):=\sup_{n\ge1}\sup_{t\in[0,t_*]}\big\|\bar{ \mathcal E}^{(n)}(t)\big\|_{L^1(\T)}<+\infty.
$$
Thus
the sequence
\begin{equation}
\label{seq}
E_n(t):=\int_0^t\bar{ \mathcal E}^{(n)}(s)\dd s,\qquad n\ge1,\,t\in[0,t_*]
\end{equation}
lies in the space $C\big([0,t_*],{\cal M}(\mathbf{E}(t_*))$, where ${\cal M}(\mathbf{E}(t_*))$ is the space of all Borel measures on $\T$ with mass
less than, or equal to, $\mathbf{E}(t_*)$, equipped with the
topology of weak convergence of measures. Since $\T$ is compact, the
space ${\cal M}(\mathbf{E}(t_*))$ is compact and metrizable. The  sequence \eqref{seq} 
is equicontinuous in the space $C\big([0,t_*];{\cal M}(\mathbf{E}(t_*))\big)$, therefore it is sequentially compact by virtue of the Ascoli-Arzel\`a Theorem, see e.g.~\cite[p.~234]{kelley}.  

Suppose that   $E(\cdot)  \in C\big([0,t_*],{\cal M}(\mathbf{E}(t_*))\big)$ is the limiting point of $\{E_n\}_{n\geq 1}$, as
$n\to+\infty$.
We shall show that for any
$G$ as in \eqref{022308-19} we have
\begin{align}
\label{022308-19x} 
 \int_0^t &\dd s \int_{\T}G(s,u)E(s,\dd u) \notag  =\int_0^{t}\dd s\int_{\T}\partial_s G(s,u)E(s,\dd u ) +\int_{\T} G(0,u)e_0(u)\dd u\notag \\
&
+
\frac1{4\ga} \int_0^{t} \dd s  \int_{\T} 
    \partial_u^2G(s, u) E(s,\dd u)+ \frac {1}{8\ga}\int_0^{t} \dd s \int_{\T}  \dd u\; \partial_u^2 G(s,u)
\Big( \int_0^s  r^2(\sigma,u)\dd \sigma\Big) \notag \\
   & - \frac{1}{4\ga} \int_0^ts\Big(\partial_uG(s,1) \left( T_+
       + \bar\tau_+^2\right)
     -T_-\partial_uG(s,0) \Big)\dd s,\qquad t\in[0,t_*].
  \end{align}
This identifies the limit $E$ of $\{E_n\}$ as a function
$E:[0,+\infty)\times\T\to\R$ that is
the unique solution of
the problem
\begin{align}
   \partial_t E (t,u)& =  \frac 1{4\ga}  \partial^2_{uu}
   \left\{  E(t,u) +
       \frac 12\; \int_0^tr^2(s,u)\dd s\right\}+e_0(u),\quad
        (t,u)\in \R_+\times \T ,
\label{eq:linear2x}
      \end{align}
with the boundary conditions
\begin{equation}
  \label{eq:bcx}
  \begin{split}
   &E (t,0) = T_- t, \qquad \quad  E(t,1) = \left(T_+ +
     \frac12\bar\tau_+^2\right)t,\qquad t\ge0
  \end{split}
\end{equation}
and
 the initial condition $E(0,u) =0$. Here $r(t,u)$ is the solution of   \eqref{eq:linear}.

Concerning the limit identification for $\{\bar {\cal
    E}^{(n)}\}_{n\ge1}$ we write
\begin{align*}
\int_0^{t}\int_{\T}\bar {\cal E}^{(n)}(s,u) G(s,u)\dd u \dd s= \int_{\T}E_n(t,u)G(t,u) \dd u-\int_0^{t}\int_{\T}E_n(s,u)\partial_s G(s,u)\dd u \dd s,
\end{align*}
and, by passing to the limit $n\to\infty$, we get that the left hand side converges to 
\begin{align*}
 \int_{\T} E(t,u)G(t,u) \dd u
-\int_0^{t}\int_{\T}E(s,u)\partial_s G(s,u)\dd u \dd s
   =\int_0^{t}\int_{\T}\partial_s E(s,u)G(s,u)\dd
  u \dd s.
\end{align*}
Hence, any   $\star-$weak limiting point 
$e\in (L^1([0,t_*];C_{0}^2(\T)))^\star$ of the sequence $\{\bar {\cal
    E}^{(n)}\}_{n\ge1}$
is 
given by $e(t,u)=\partial_t E(t,u)$, which in turn satisfies
\eqref{022308-19} and Theorem \ref{main-result} would then follow. Therefore one is left with proving \eqref{022308-19x}.



 Consider now a smooth test function $G\in C^\infty([0,+\infty)\times
\T)$  such that $G(s,0) = G(s,1) \equiv 0$, $s\ge0$.
Then,  from \eqref{eq:en-evol}, we get 
\begin{align*}
  \int_{\T} G (t,u)\dd u& \bar{\mathcal E}^{(n)}(t,u)   - \int_{\T} G (0,u)
   \bar{\mathcal E}^{(n)} (0,u) \dd u\\
&
=\frac{1}{n}
   \sum_{x\in\T_n} \mathbb E_n\left[  G_x (t) \mathcal E_x(t) \right]- \frac{1}{n}
  \sum_{x\in\T_n} \mathbb E_n\left[  G_x (0) \mathcal E_x(0) \right]
   +o_n(1)\\
&  =    \frac{1}{n} \sum_{x=1}^{n-1}
 \int_0^{t} \mathbb E_n\left[ \partial_s\big(G_x (s){\mathcal
  E}_x(s)\big) \right] \dd s +o_n(1)
=:\mathrm{I}_n+\mathrm{II}_n+o_n(1),
\end{align*}
where $o_n(1)\to0$, as $n\to+\infty$ and 
\begin{align*}
&
\mathrm{I}_n:=\frac{1}{n} \int_0^{t}  
   \sum_{x=0}^{n-1} \E_n\left[  \partial_sG_x (s) \mathcal E_x ^{(n)} (s)
      \right] \dd s=\int_0^{t}\int_{\T} \partial_s G (s,u) \bar{\mathcal E}^{(n)} (s,u) \dd u \dd s +o_n(1),\\
&
  \mathrm{II}_n:=\frac{1}{n} \int_0^{t} 
  \sum_{x=0}^{n-1}   \E_n\left[ G_x (s) \partial_s\mathcal E_x ^{(n)} (s)
      \right] \dd s .
\end{align*}
Thanks to \eqref{eq:en-evol}, and after an integration by parts, we write
 \begin{equation}
\label{eq:cal0}
\mathrm{II}_n = \frac{1}{n} \int_0^{t} \sum_{x=1}^{n-1} \E_n\left[  G_x (s) 
      \big(j_{x-1,x}- j_{x,x+1}\big) (s)\right]  \dd s = \mathrm{II}_{n,1} + \mathrm{II}_{n,2},\end{equation}
      with \begin{align*}
\mathrm{II}_{n,1} &:= \int_0^{t} \sum_{x=1}^{n-2} \mathbb E\big[
       (\nabla_n G)_x(s)\; j_{x,x+1}(s) \big]  \dd s,\\
\mathrm{II}_{n,2} &:=\int_0^{t}  \mathbb E\big[- nG_{n-1}(s) j_{n-1,n}(s)
      + nG_1(s) j_{0,1}(s) \big]\dd s.
              \end{align*}
By Lemma \ref{lem:bound2}--\eqref{eq:11}, we conclude that $\mathrm{II}_{n,2}=o_n(1)$.

By a direct calculation we conclude the following \emph{fluctuation-dissipation relation} for the microscopic currents:
\begin{equation}
  \label{eq:fd1-n}
  j_{x,x+1} = n^{-2} L g_x - (V_{x+1}-V_x), 
\qquad x\in\T_n^{\rm o},
\end{equation}
 with
  \begin{equation} \label{eq:defg}
   g_x:=-\frac14p_x^2 + \frac{1}{4\gamma}p_x(r_x+r_{x+1}),\qquad
 V_x:= \frac 1{4 \gamma}  \big(r_x^2 +p_xp_{x-1}\big) .
\end{equation}
Using the notation $g_x(t):=g_x({\bf r}(t),{\bf p}(t))$ (and similarly for other local functions), this allows us to write $\mathrm{II}_{n,1} =\sum_{j=1}^4\mathrm{II}_{n,1,j}
$, where 
\begin{align*}
&
 \mathrm{II}_{n,1,1} := \int_0^{t}  \frac 1{n}  \sum_{x=2}^{n-2}
                                                 {\mathbb E}_n\big[  (\Delta_n G)_x (s)V_x(s) 
                                                 \big]  \dd s\;  ,\\
&\mathrm{II}_{n,1,2} :=\int_0^{t}  \frac 1{n^2}
  \sum_{x=1}^{n-2}{\mathbb E}_n \big[  (\nabla_n G)_x(s)  L g_x(s)\big] \dd s , \\
&
\mathrm{II}_{n,1,3} := -\int_0^{t} \dd s\; {\mathbb E}_n\big[  (\nabla_n G)_{n-2} (s)V_{n-1}(s) 
                                                 \big]  ,\\
&
\mathrm{II}_{n,1,4} := \int_0^{t} \dd s\; {\mathbb E}_n\big[ 
                                                   (\nabla_n G)_1(s)V_1(s)\big] .
\end{align*}
We have
\begin{align*}
\mathrm{II}_{n,1,2}= & \;  \frac 1{n^2} \sum_{x=1}^{n-2}  (\nabla_n G)_x(t)
  \mathbb E_n\left[ g_x(t)\right]-\frac 1{n^2} \sum_{x=1}^{n-2}
  (\nabla_n G)_x(0) \mathbb E_n\left[ g_x(0)\right]\\
&
-\int_0^{t}\frac 1{n^2} \sum_{x=1}^{n-2} (\nabla_n \partial_sG)_x(s) \mathbb E_n \left[ g_x(s)\right]\dd s,
  \end{align*}
which {tends to zero}, thanks to Proposition \ref{prop022111-19}.
By Lemma \ref{lem:bound2}--\eqref{ex-p0p1} and \eqref{ex-2-l11t}, we have
\begin{align}
  \label{eq:6}
    \lim_{n\to\infty}  \mathrm{II}_{n,1,4}&=\lim_{n\to\infty}   \int_0^t {\mathbb E}_n\big[ (\nabla_n G)_1(s) V_1(s) \big] \dd s
    =\frac{1}{4\ga} \lim_{n\to\infty}  \int_0^t  \partial_u
      G(s,0)\E_n\big[r_1^2(s)\big] \dd s\notag\\
&
=\frac{T_-}{4\ga} \int_0^t  \partial_u G(s,0) \dd s,
\end{align}
which takes care of the left boundary condition. Concerning the right
one:
\begin{align}
  \label{eq:6aa} 
\mathrm{II}_{n,1,3} =-  \int_0^t   {\mathbb E}_n\big[ (\nabla_n G)_{n-2} (s)
     V_{n-1}(s)\big]  \dd s\;  =J_{n,1}+J_{n,2} ,
\end{align}
where 
\begin{align*}
J_{n,1}&:= \int_0^t  {\mathbb E}_n\big[   (\nabla_n G)_{n-2} (s) (V_n-V_{n-1})(s)
 \big] \dd s, \\
J_{n,2}&:= -\int_0^t {\mathbb E}_n\big[  (\nabla_n G)_{n-2} (s)
     V_{n}(s) \big] \dd s.
  \end{align*}
By virtue of Lemma \ref{lem:bound2}--\eqref{ex-2-r11t} we have
\[
\lim_{n\to+\infty}  J_{n,2}= - \lim_{n\to+\infty}  {\mathbb E}_n\left[ \int_0^t  (\nabla_n G)_{n-2} (s)
     V_{n}(s)\dd s \right]
 = -\frac{\left(T_+
     + \bar\tau_+^2\right)}{4\ga} \int_0^t  \partial_u G(s,1) \dd s.
\]
On the other hand, using \eqref{eq:fd1-n} for $x=n-1$, the term $J_{n,1}$
equals
\begin{equation}
  \label{eq:6b}
  \begin{split}  
 J_{n,1}=  -\frac{1}{n^2} {\mathbb E}_n\left[ \int_0^t   (\nabla_n
     G)_{n-2} (s) L g_{n-1}(s) \dd s
 \right]
&+ \int_0^t  (\nabla_n G)_{n-2} (s)
  {\mathbb E}_n  \big[j_{n-1,n}(s)\big] \dd s.
  \end{split}
\end{equation}
From Lemma \ref{lem:bound2}--\eqref{eq:11} we conclude that the second term {tends to zero}, with
$n\to+\infty$. By integration by parts the first term equals
\begin{multline}
  \label{eq:6c}
  \frac{1}{n^2} {\mathbb E}_n\Big[(\nabla_n
  G)_{n-2} (0) g_{n-1}(0)- (\nabla_n
  G)_{n-2} (t) g_{n-1}(t) 
 \Big]  \\  \qquad  +\frac{1}{n^2} {\mathbb E}_n\left[ \int_0^t  (\nabla_n
     \partial_s G)_{n-2} (s) g_{n-1}(s)  \dd s
 \right],
\end{multline}
which {tends to zero}, thanks to Proposition \ref{prop022111-19}. Summarizing, we have shown that
\begin{align}
  \label{eq:6d} 
&  \lim_{n\to+\infty} \mathrm{II}_{n,1,3}   =-\lim_{n\to+\infty} {\mathbb E}_n\left[ \int_0^t (\nabla_n G)_{n-2} (s)
     V_{n-1}(s)   \dd s \right] \notag\\
     &  =
 - \frac{\left(T_+
     + \bar\tau_+^2\right) }{4\ga}\int_0^t  \partial_u G(s,1) \dd s .
  \end{align}
Now, for the bulk expression $\mathrm{II}_{n,1,1}$, we  write
$
\mathrm{II}_{n,1,1} = {\cal J}_{n,1}+{\cal J}_{n,2},
$
where
\begin{align*}
{\cal J}_{n,1}&:=\frac{1}{4\ga}\int_0^{t}  \frac 1{n}  \sum_{x=2}^{n-2} {\mathbb E}_n\big[
                                                  (\Delta_n G)_x (s)r_x^2(s)
                                                 \big] \dd s\;, \\
{\cal J}_{n,2}&:=\frac{1}{4\ga}\int_0^{t}   \frac 1{n}  \sum_{x=2}^{n-2}{\mathbb E}_n\big[
                                                  (\Delta_n G)_x (s)p_x(s)p_{x-1}(s)
                                                 \big] \dd s\;.
\end{align*}

\subsubsection{Estimates of ${\cal J}_{n,2}$}


After a direct calculation, it follows from \eqref{eq:generator} that
\begin{equation}
\label{033108-19}
n^{-2} L h_x = (W_{x+1}-W_x)- p_x p_{x-1},\qquad x=2,\ldots,n-2,
\end{equation}
with
\begin{align*}
&h_x:=\frac{1}{2\ga}
\left(\frac12(r_x +r_{x-1})^2 +p_{x-1}p_{x} -r_x^2\right),\\
&
W_x:=\frac{1}{2\ga}p_{x-2}(r_{x-1}+r_x).
\end{align*}
Substituting into the expression for ${\cal J}_{n,2}$ we conclude that $
  {\cal J}_{n,2}=K_{n,1}+K_{n,2},$
where $K_{n,1}$ and $K_{n,2}$ correspond to $(W_{x+1}-W_x)$ and $n^{-2}Lh_x$,
respectively.
Using the summation by parts to deal with $K_{n,1}$,
performing time integration in the case of  $K_{n,2}$ and subsequently 
invoking the energy bound from Proposition \ref{prop022111-19}, we conclude that
\begin{equation}
\label{050301-20}
\lim_{n\to+\infty}{\cal J}_{n,2}=0.
\end{equation}

\medskip

\subsubsection{Limit of ${\cal J}_{n,1}$}

We write
$
{\cal J}_{n,1}={\cal J}_{n,1,1}+{\cal J}_{n,1,2}+{\cal J}_{n,1,3},$
where
\begin{align*}
&
{\cal J}_{n,1,1}:=\frac{1}{4\ga}\int_0^{t} \frac 1{n}  \sum_{x=2}^{n-2}
                                         {\mathbb E}_n\big[          (\Delta_n G)_x(s)
  {\cal E}_x(s) \big]  \dd s,\\
&
{\cal J}_{n,1,2}:=\frac{1}{8\ga}\int_0^{t} \frac 1{n}  \sum_{x=2}^{n-2}
                                                  (\Delta_n G)_x(s)
  \Big([\bar r_x^{(n)}(s) ]^2-[\bar p_x^{(n)}(s) ]^2\Big)\dd s ,\\                                          
&
{\cal J}_{n,1,3}:=\frac{1}{8\ga}\int_0^{t}  \frac 1{n}  \sum_{x=2}^{n-2}
                                    {\mathbb E}_n\Big[               (\Delta_n G)_x(s)
   \Big([\tilde r_x^{(n)}(s) ]^2-[\tilde p_x^{(n)}(s) ]^2\Big)
                                                 \Big] \dd s .
\end{align*}
To deal with the term ${\cal J}_{n,1,3}$ we use Proposition \ref{equipartition}, which allows us to conclude that
$
\lim_{n\to+\infty}{\cal J}_{n,1,3}=0.
$


Concerning the term ${\cal J}_{n,1,2}$ we first note that
$
\lim_{n\to+\infty}({\cal J}_{n,1,2}-\tilde{\cal J}_{n,1,2})=0,
$
where
\begin{align*}
&
\tilde {\cal J}_{n,1,2}:=\frac{1}{8\ga}\int_0^{t}   \frac 1{n}  \sum_{x=2}^{n-2}
                                                  (\Delta_n G)_x
  (s) \big(\bar r_x^{(n)}(s) \big)^2\dd s.
\end{align*}
This is a consequence of the following result, proved in Section \ref{sec7bis} (and which uses Proposition \ref{prop012812-19}). 

\begin{lemma}
\label{lm010911-19}
Under Assumption \ref{ass1}, for any $t>0$ we have
\begin{equation}
\label{052412-19}
\int_0^t\sup_{x\in\T_n}\big|\bar p^{(n)}_{x}(s)\big|^2\dd s\lesssim \frac{\log^2
  (n+1)}{n^2},\qquad n\ge1.
\end{equation}
\end{lemma}

Next,  using Theorem \ref{hydro2},
we conclude that
\begin{equation}
\label{010401-20}
\lim_{n\to+\infty}\tilde {\cal J}_{n,1,2}=\frac{1}{8\ga}\int_0^{t} \dd s\; \int_{\T} \dd u\;              (\partial^2_u  G)(s,u)r^2(s,u),
\end{equation}
where $r(\cdot)$ is the solution of \eqref{eq:linear}.

Summarizing,  the results announced above, allow us to conclude that 
\begin{align*}
 &\int_{\T} G (t,u) \bar{\mathcal E}^{(n)}(t,u)  \dd u - \int_{\T} G (0,u)
   \bar{\mathcal E}^{(n)} (0,u)  \dd u \\
&
-\frac{1}{4\ga}\int_0^{t}  \int_{\T} (\partial^2_u G)(s,u) \bar{\mathcal E}^{(n)} (s,u)   \dd u \dd s =\frac{1}{8\ga}\int_0^{t} \int_{\T}(\partial^2_u  G)(s,u)r^2(s,u) \dd u \dd s \\
&
+\frac{T_-}{4\ga} \int_0^t \partial_u G(s,0)  \dd s - \frac{\left(T_+
     + \bar\tau_+^2\right)}{4\ga} \int_0^t \partial_u G(s,1) \dd s +o_n(t),
\end{align*}
where $\lim_{n\to\infty}\sup_{s\in[0,t_*]}|o_n(s)|=0$  for a  fixed $t_*>0$. 
Given $t>0$ we can take, as  test function, $G(s,u):=H(t,u)$, for any $s\in[0,t]$, with an arbitrary compactly supported $H\in  C([0,+\infty);C_{0}^2(\T))$. Integrating over $t\in[0,t_*]$ we obtain that $E_n(t)$, cf.~\eqref{seq}, 
satisfies 
\begin{align*}
 &  \int_0^{t_*}  \int_{\T}H(s,u)E_n(s,u) \dd s \dd u \\   = & \; \int_0^{t_*}\dd s\int_{\T}\partial_s H(s,u)E_n(s,u )\dd u +\int_{\T} H(0,u)\bar{\mathcal E}^{(n)}(0,u)\dd u \\
&
+
\frac1{4\ga} \int_0^{t_*} \dd s  \int_{\T} 
    \partial_u^2H(s, u) E_n(s, u)\dd u+ \frac {1}{8\ga}\int_0^{t_*} \dd s \int_{\T} \partial_u^2 H(s,u)
\Big( \int_0^s  r^2(\sigma,u)\dd \sigma\Big) \dd u \notag \\
   & - \frac{1}{4\ga} \int_0^{t_*}s\Big(\partial_uH(s,1) \left( T_+
       + \bar\tau_+^2\right)
     -T_-\partial_uH(s,0) \Big) \dd s+o_n(t_*),
  \end{align*}
  with $o_n(t_*)\to0$, as $n\to+\infty$.
This obviously implies
\eqref{022308-19x} and ends the proof of  Theorem \ref{main-result}.

\section{Dynamics of the averages}

\label{sec66}

This part aims at proving previous results which have been left aside: \begin{itemize}
\item (Sections \ref{sec7} and \ref{sec8}) Proposition \ref{prop012812-19} and
 Proposition \ref{lm010911-19a}, which only deal with the extremity points $x=0$ and $x=n$;
 
 \item (Section \ref{sec6.1})  Proposition \ref{cor013112-19z} which gives an $L^2$ bound on all the averages 
 
 \emph{(its proof uses Proposition \ref{lm010911-19a})};
 
 \item (Section \ref{sec7bis}) Lemma \ref{lm010911-19}, which controls $\sup_{x\in\T_n} |\bar p_x^{(n)}(s)|^2$ 
 
 \emph{(its proof uses Proposition \ref{prop012812-19})}; 
\item (Section \ref{sec6.5}) finally, Lemma \ref{lm010301-20}, which gives a bound on the $H^1$--norm of the stretch averages 

\emph{(its proof uses Proposition \ref{lm010911-19a}, Proposition \ref{cor013112-19z} and Proposition \ref{prop022111-19})}.

\end{itemize}
All their proofs are based on a refined analysis of the system of equations satisfied by the averages of momenta and stretches.

\bigskip

To simplify the notation, in the present section we omit writing the superscript $n$ by  the averages
 $\bar p_x^{(n)}(t)$, $\bar r_x^{(n)}(t)$ defined in  \eqref{En0}.
Their dynamics is given by the following system of ordinary
differential equations
 \begin{align}
 \label{eq:qdynamics-bar}
     \frac{\dd}{\dd t}  \bar  r_x(t) &= n^2 \left(\bar  p_x(t)- \bar  p_{x-1}(t)\right) ,\qquad  x=1,\ldots,n \vphantom{\Big(}\\
      \frac{\dd}{\dd t} \bar  p_x(t) &= n^2 \left(\bar  r_{x+1}(t)-
                                       \bar  r_x(t)\right)  - 2 \gamma
                                       n^2 \bar  p_x(t),\qquad x\in\T_n^{\rm o}
      \label{eq:pdyn-bar}\end{align}
  and at the boundaries: $\bar  r_0(t) \equiv0$, 
     \begin{align}
       \frac{\dd}{\dd t}  \bar   p_0(t) &= n^2 \;  \bar   r_1(t)   
      -  n^2(2\gamma+\tilde \gamma)  \bar   p_0(t) ,\label{eq:rbd-bar}\\
  \frac{\dd}{\dd t}  \bar    p_n(t) &= -n^2 \;  \bar    r_n(t)    +
                                      n^2\; \bar\tau_+(t)  -
                                      n^2(2\gamma +  \tilde \gamma)
                                      \bar    p_n(t) \label{eq:pbd-bar}.
   \end{align}
We have allowed above the forcing $\tau_+(t)$ to depend on $t$. Although in most
cases we shall consider $\tau_+(t)\equiv \tau_+$ constant, yet in some
instances we also admit to be in the form
$\tau_+(t)=\mathbf{1}_{[0,t_*)}(t)\tau_+$ for some $t_*>0$, $\tau_+\in\R$.


The resolution of these equations will allow us to get several crucial estimates. For that purpose, we first rewrite the system in terms of Fourier transforms, and we will then take its Laplace transform. 
{We} define
$$
\widehat{\bar r}(t,k)=\E_n \big[\widehat r(t,k)\big] ,\qquad  \widehat{\bar p}(t,k)=\E_n \big[\widehat p(t,k)\big].
$$
From \eqref{eq:qdynamics-bar} and \eqref{eq:pdyn-bar} we conclude that
\begin{equation}
  \label{eq:rev-sys}
\begin{split} 
& \frac \dd{\dd t}\begin{pmatrix}
\widehat{\bar r}(t,k)\\
\widehat{\bar p}(t,k)
 \end{pmatrix}
   = \; n^2A \begin{pmatrix}
\widehat{\bar r}(t,k)\\
\widehat{\bar p}(t,k)
 \end{pmatrix}+n^2\bar \tau_+(t) \begin{pmatrix}
0\\
e^{2\pi i  k}
 \end{pmatrix}\\
 &
\qquad\qquad\qquad-n^2\bar p_0(t) \begin{pmatrix}
1\\
\tilde\gamma
 \end{pmatrix} + n^2\bar p_n(t)\begin{pmatrix}
1\\
-\tilde\gamma e^{2\pi i  k}
 \end{pmatrix},
 \end{split}\end{equation}
where
$$
A= \begin{pmatrix}
0& &1 - e^{-2i\pi k}\\
 e^{2i\pi k}-1&&-2\gamma
 \end{pmatrix}.
$$
Assuming that $\bar \tau_+(t)\equiv \bar \tau_+$, $t\ge0$ we rewrite equation \eqref{eq:rev-sys} in the mild formulation and we obtain
\begin{align}
  \label{eq:rev-sys-mild}
 \begin{pmatrix}
\widehat{\bar r}(t,k)\\
\widehat{\bar p}(t,k)
 \end{pmatrix}
   = & \; \exp\left\{ n^2A t\right\} \begin{pmatrix}
\widehat{\bar r}(0,k) \notag \\
\widehat{\bar p}(0,k)
 \end{pmatrix}+n^2 \bar \tau_+\int_0^t \exp\left\{ n^2A (t-s)\right\}\begin{pmatrix}
0\\ 
e^{2\pi i  k}
 \end{pmatrix}\dd s\\
&
-n^2 \int_0^t \exp\left\{ n^2A (t-s)\right\} p_0(s) \begin{pmatrix}
1\\
\tilde\gamma
 \end{pmatrix} \dd s \notag\\ & + n^2\int_0^t \exp\left\{ n^2A (t-s)\right\} p_n(s) \begin{pmatrix}
1\\
-\tilde\gamma e^{2\pi i  k}
 \end{pmatrix} \dd s.\end{align}
Denoting by 
$ \lambda_\pm(k)=-\big(\ga\pm\sqrt{\ga^2-4\sin^2(\pi k)}\big)$ the eigenvalues of $A$, 
we obtain the following autonomous integral equation for $\{\bar p_x(t)\}_{x\in\T_n}$
\begin{equation}
  \label{070411-19x}
\bar p_x(t)
   =\hat{\cal T}^{(n)}_x(t)+\int_0^tq_{1,x}^{(n)}(t-s) \bar
    p_{0,n}^{\rm diff}(s) \dd s+\int_0^tq_{2,x}^{(n)}(t-s) \bar
    p_{0,n}^{\rm sum}(s) \dd s,
    \end{equation}
    where
    \begin{align*}
q_{1,x}^{(n)}(t)&:=\underset{{k\in \widehat{\T}_n}}{{\hat\sum}}\frac{n^2 }{\la_-(k)-\la_+(k) } e^{2\pi ik x} (e^{2\pi ik}-1)\big(e^{n^2t\la_+(k)}-
                            e^{n^2t\la_-(k)}\big)\\
q_{2,x}^{(n)}(t)&:=\tilde \ga \underset{{k\in \widehat{\T}_n}}{{\hat\sum}}\frac{n^2 }{\la_-(k)-\la_+(k) } e^{2\pi ik x}  \big(\la_-(k)e^{n^2t\la_-(k)}-
                                             \la_+(k)e^{n^2 t\la_+(k)}\big),\\
\bar p_{0,n}^{\rm diff}(s)&:= \bar p_0(s) -\bar p_n(s), \vphantom{\bigg(}\\
\bar p_{0,n}^{\rm sum}(s)&:= \bar p_0(s) +e^{2\pi i  k}\bar p_n(s),\vphantom{\bigg(}
 \\
\hat{\cal T}^{(n)}_x(t)&:=\underset{{k\in \widehat{\T}_n}}{{\hat\sum}}e^{2\pi ik x} \frac{\big(-(1-e^{2\pi ik})\hat{ \bar r}(0,k)+\bar \tau_+  e^{2\pi i  k}\big)\big(e^{n^2 t\la_-(k)}- e^{n^2  t\la_+(k)}\big)}{\la_-(k)-\la_+(k)},\\
& \quad 
+\underset{{k\in \widehat{\T}_n}}{{\hat\sum}}e^{2\pi ik x} \frac{\hat{ \bar p}(0,k) \big(\la_-(k)e^{t\la_-(k)}-
                                             \la_+(k)e^{t\la_+(k)}\big)}{\la_-(k)-\la_+(k)}.
\end{align*}

%
\subsection{Proof of Proposition \ref{prop012812-19}}
\label{sec7}
We define, for any $x \in \T_n$,  and any $\lambda \in \mathbb{C}$ such that $\mathrm{Re}(\lambda) >0$, the Laplace transforms:
$$
\tilde{\bar p}_x(\la):=\int_0^{+\infty}e^{-\la t}{\bar p}_x(t)\dd
t,\qquad \tilde{\bar \tau}_+(\la):=\int_0^{+\infty}e^{-\la t}{\bar \tau}_+(t)\dd t
$$
and
$$
\tilde{{\bar r}}(\la,k):=\int_0^{+\infty}e^{-\la t}\widehat{\bar r}(t,k)\dd t,\qquad \tilde{{\bar p}}(\la,k):=\int_0^{+\infty}e^{-\la t}\widehat{\bar p}(t,k)\dd t.
$$
Performing the Laplace transform on both sides of \eqref{eq:rev-sys}
we obtain the following system
\begin{align}
  \label{eq:rev-sys-laplace1}
 \begin{pmatrix}
\tilde{{\bar r}}(\la,k)\\
\tilde{{\bar p}}(\la,k)
 \end{pmatrix}
   = &\; (\la-n^2A)^{-1}    
   \begin{pmatrix}
\widehat{\bar r}(0,k)\\
\widehat{\bar p}(0,k)
 \end{pmatrix}+n^2\tilde{\bar \tau}_+(\la) (\la-n^2A)^{-1} \begin{pmatrix}
0\\
e^{2\pi i  k}
 \end{pmatrix}\\
&
-n^2\tilde{\bar p}_0(\la) (\la-n^2A)^{-1}\begin{pmatrix}
1\\
\tilde\gamma
 \end{pmatrix} + n^2\tilde{\bar p}_n(\la) (\la-n^2A)^{-1}\begin{pmatrix}
1\\
-\tilde\gamma e^{2\pi i  k}
 \end{pmatrix}. \notag \end{align}
Here
$$
(\la-n^2A)^{-1}= \frac{1}{n^2\Delta(\la/n^2,k)}
\begin{pmatrix}
\dfrac{\la}{n^2}+2\ga & &1 - e^{-2i\pi k}\\
e^{2i\pi k}-1&&\dfrac{\la}{n^2}
\end{pmatrix}
$$
and
$
\Delta(\la,k):=\la^2+2\ga \la+4\sin^2(\pi k).
$
Let
$$
\tilde{\bar p}^{\rm diff}_{0,n}(\la):=\tilde{\bar p}_{0}(\la)-\tilde{\bar
  p}_{n}(\la)\quad\mbox{and}\quad \tilde{\bar p}_{0,n}^{(+)}(\la):=\tilde{\bar p}_{0}(\la)+\tilde{\bar
  p}_{n}(\la).
$$
Since
${\hat\sum}_{k\in \widehat{\T}_n}\widehat{\bar
    r}(0,k)=r_0=0$ we have
\begin{align}
\label{010402-20}
\underset{{k\in \widehat{\T}_n}}{{\hat\sum}}\frac{-4\sin^2(\pi k)\widehat{\bar
    r}(0,k)}{\Delta(\la/n^2,k)}& =\underset{{k\in \widehat{\T}_n}}{{\hat\sum}}\frac{\big(-4\sin^2(\pi k)+\Delta(\la/n^2,k)\big)\;\widehat{\bar
    r}(0,k)}{\Delta(\la/n^2,k)}\notag\\
&
=\underset{{k\in \widehat{\T}_n}}{{\hat\sum}}\frac{\la/n^2 (\la/n^2+2\ga)\;\widehat{\bar
    r}(0,k)}{\Delta(\la/n^2,k)}.
\end{align} 
Using \eqref{eq:rev-sys-laplace1} and
\eqref{010402-20}  we conclude the expressions
for $\tilde{\bar p}^{\rm diff}_{0,n}(\la)$ and $\tilde{\bar p}_{0,n}^{(+)}(\la)$
\begin{align}
  \label{eq:diff}
  \tilde{\bar p}^{\rm diff}_{0,n}(\la)
   = & \; \frac{1}{n^2  \mathbf{e}_{\rm d,n}(\la/n^2)}\underset{{k\in \widehat{\T}_n}}{{\hat\sum}}\frac{ (\la/n^2+2\ga)\widehat{\bar
    r}(0,k)}{\Delta(\la/n^2,k)}\\&+\frac{1}{ n^2  \mathbf{e}_{{\rm d},n}(\la/n^2)}\underset{{k\in \widehat{\T}_n}}{{\hat\sum}}\frac{(1-e^{-2\pi
    i k})\widehat{\bar p}(0,k)}{\Delta(\la/n^2,k)}\notag
- \frac{2 \tilde{\bar
    \tau}_+(\la)}{   \mathbf{e}_{{\rm d},n}(\la/n^2)}\underset{{k\in \widehat{\T}_n}}{{\hat\sum}}\frac{\sin^2(\pi k)}{\Delta(\la/n^2,k)}
\end{align}
and
\begin{align}
  \label{eq:sum}
 \tilde{\bar p}^{(+)}_{0,n}(\la)
   =& \;\frac{2i}{n^2 \mathbf{e}_{{\rm s},n}(\la/n^2)}\underset{{k\in \widehat{\T}_n}}{{\hat\sum}}\frac{\sin(2\pi k)\widehat{\bar
    r}(0,k)}{\Delta(\la/n^2,k)} \\&
+\frac{\la}{n^4 \mathbf{e}_{{\rm s},n}(\la/n^2)}\underset{{k\in \widehat{\T}_n}}{{\hat\sum}}\frac{(1+e^{-2\pi i k}) \widehat{\bar p}(0,k)}{\Delta(\la/n^2,k)}
 +\frac{2\la \tilde{\bar
    \tau}_+(\la) }{n^2 \mathbf{e}_{{\rm s},n}(\la/n^2)}\underset{{k\in \widehat{\T}_n}}{{\hat\sum}}\frac{\cos^2(\pi k)}{\Delta(\la/n^2,k)}, \notag
  \end{align}
where
\begin{align*}
 \mathbf{e}_{{\rm d},n}(\la)&:=\underset{{k\in \widehat{\T}_n}}{{\hat\sum}}\frac{\la+2\ga+2 \tilde\ga\sin^2(\pi
  k)}{\la^2 +2\ga\la +4\sin^2(\pi
  k)},\\
 \mathbf{e}_{{\rm s},n}(\la) &:
=\underset{{k\in \widehat{\T}_n}}{{\hat\sum}}\frac{\la^2+2\ga
 \la+2 \tilde\ga \la\cos^2(\pi k)+4\sin^2(\pi
  k)}{\la^2+2\ga\la +4\sin^2(\pi k)}.
\end{align*}

\subsubsection{First part of Proposition \ref{prop012812-19}: estimates of $\|\bar p_0-\bar p_n\|_{L^2(\R_+)}$}.

We fix $t_*>0$ and consider 
$\bar\tau_+(t):=\bar\tau_+ \mathbf 1_{[0,t_*]}(t)$. Then, for $\lambda \in \mathbb{C}, \mathrm{Re}(\lambda)>0$,
$$
\tilde{\bar
    \tau}_+(\la) =\bar\tau_+\int_0^{t_*}e^{-\la t}\dd
  t=\frac{\bar\tau_+}{\la}\left(1-e^{-\la t_*}\right).
$$
By the Plancherel Theorem we have
\begin{equation}
\label{plancherel1}
\|\bar p_0-\bar p_n\|_{L^2(\R_+)}^2=\frac{1}{2\pi}\int_{\R}\big|\tilde{\bar
  p}_{0,n}^{\rm diff}(i\eta)\big|^2\dd\eta,
\end{equation}
therefore, from \eqref{eq:diff}, we estimate
\begin{equation}
\label{012712-19}
\|\bar p_0-\bar
  p_n\|_{L^2(\R_+)}^2\lesssim P^{\rm d}_{n,1}+P^{\rm d}_{n,2}+P^{\rm d}_{n,3},
\end{equation}
where
\begin{align}
&P^{\rm d}_{n,1}:=\frac{1}{n^2}\int_{\R}\big|\rho_{{\rm d},n}(\eta) \big|^2
 \big|e_{{\rm d},n}(\eta)\big|^{-2}\dd\eta,
\label{022712-19a}\\
&P^{\rm d}_{n,2}:=\frac{1}{n^2}\int_{\R}\big|\pi_{{\rm d},n}(\eta) \big|^2
 \big|e_{{\rm d},n} (\eta)\big|^{-2}\dd\eta,
\label{022712-19b}\\
&P^{\rm d}_{n,3}:=\frac{1}{n^2}\int_{\R}\left|\frac{\sin(n^2\eta
  t_*/2)}{\eta}\right|^2\big|a_{n}(\eta) \big|^2
 \big|e_{{\rm d},n}(\eta)\big|^{-2}\dd\eta,
\label{022712-19c}
\end{align}
and 
\begin{align}
e_{{\rm d},n}(\eta)&:=\mathbf{e}_{{\rm d},n}(i\eta), \label{e-d}\\
 \rho_{{\rm d},n}(\eta)&:=\underset{{k\in \widehat{\T}_n}}{{\hat\sum}}\frac{(i\eta+2\ga)\widehat{\bar
    r}(0,k)}{4\sin^2(\pi  k)-\eta^2+2i\ga \eta},\label{pi-d}\\
 \pi_{{\rm d},n}(\eta)&:=\underset{{k\in \widehat{\T}_n}}{{\hat\sum}}\frac{(1-e^{-2\pi
    i k})\widehat{\bar p}(0,k)}{4\sin^2(\pi  k)-\eta^2+2i\ga \eta},
\label{rho-d}\\
a_{n}(\eta)&:=\underset{{k\in \widehat{\T}_n}}{{\hat\sum}}\frac{2\sin^2(\pi
  k)}{4\sin^2(\pi  k)-\eta^2+2i\ga \eta } .
\label{a-n}
\end{align}
After elementary, but somewhat tedious calculations, see
the Appendix sections \ref{sec:app1}, \ref{sec:app2}, \ref{sec:app4} and \ref{sec:app5} for details, we conclude that: for any $\eta \in \R$,
\begin{align}
\label{021612-19}
\frac{1}{|\eta|}&\lesssim |e_{{\rm d},n}(\eta)|, \\
\label{011412-19}
  |a_{n}(\eta)|&\lesssim \frac{1}{1+\eta^2},\\
\label{042712-19cc}
 |\pi_{{\rm d},n}(\eta)|&\lesssim 
  \frac{1}{\eta^2+1} \log\big(1+|\eta|^{-1}\big) ,\\
\label{062712-19}
 \bigg|\frac{\rho_{{\rm d},n}(\eta)}{e_{{\rm d},n}(\eta)}\bigg|&\lesssim 
  \frac{1}{\eta^2+1}, \qquad n \geqslant 1.
\end{align}
{We} emphasize here that these bounds are obtained thanks to the assumption that we made on the spectrum of the averages at initial time, recall Assumption \ref{ass1}--\eqref{spec-bound}.  From \eqref{022712-19c},     \eqref{021612-19} and \eqref{011412-19} we get
\begin{align*}
&P^{\rm d}_{n,3}\lesssim \frac{1}{n^2}\int_{0}^1\sin^2(n^2\eta
  t_*/2)\dd\eta+\frac{1}{n^2}\int_{1}^{+\infty}\frac{\dd\eta}{\eta^2}\lesssim \frac{1}{n^2}.
\end{align*}
A similar argument, using \eqref{042712-19cc} and \eqref{062712-19}, shows that $P^{\rm d}_{n,j}\lesssim n^{-2}$, for $j=1,2.$ 
As a result we conclude \eqref{022812-19a}.


\subsubsection{Second part of Proposition \ref{prop012812-19}: estimates of $\|\bar p_0+\bar p_n\|_{L^2(\R_+)}$}. 

Recall that $\bar\tau_+(t)=\bar\tau_+ \mathbf{1}_{[0,t_*)}(t)$. The strategy to estimate $\|\bar p_0+\bar p_n\|_{L^2(\R_+)}$ is completely similar. First, we write
\begin{equation}
\label{plancherel2}
 \|\bar p_0+\bar p_n\|_{L^2(\R_+)}^2=\frac{1}{2\pi}\int_{\R}\big|\tilde{\bar
  p}^{(+)}_{0,n}(i\eta)\big|^2 \dd\eta,
\end{equation}
with $\tilde{\bar
  p}^{(+)}_{0,n} $ given by \eqref{eq:sum}.
Substituting from \eqref{eq:sum} we get
\begin{align}
  \label{061812-19}
\|\bar p_0+\bar p_n\|_{L^2(\R_+)}^2\lesssim P^{\rm s}_{n,1}+P^{\rm s}_{n,2}+P^{\rm s}_{n,3},
 \end{align}
where
\begin{align}
  \label{061812-19b}
&P^{\rm s}_{n,1}:
=\frac{1}{n^2}\int_{\R}\big|\rho_{{\rm s},n}(\eta) \big|^2
  \big|e_{{\rm s},n}(\eta)\big|^{-2}\dd\eta,
 \\
 & \label{061812-19c}
P^{\rm s}_{n,2}:=\frac{1}{n^2}\int_{\R}\eta^2\big|\pi_{{\rm s},n}(\eta) \big|^2
  \big|e_{{\rm s},n}(\eta)\big|^{-2}\dd\eta
 ,\\
 & \label{061812-19a}
P^{\rm s}_{n,3}:
=\frac{1}{n^2}\int_{\R}\sin^2(n^2\eta
  t_*/2)\big|c_n(\eta) \big|^2
  \big|e_{{\rm s},n}(\eta)\big|^{-2}\dd\eta
 \end{align} 
and 
\begin{align}
e_{{\rm s},n}(\eta)&:=\mathbf{e}_{{\rm s},n}(i\eta),\\
\label{062712-19a}
\rho_{{\rm s},n}(\eta) &:=\underset{{k\in \widehat{\T}_n}}{{\hat\sum}}\frac{\sin(2\pi k)\widehat{\bar
    r}(0,k)}{-\eta^2+2i\ga\eta+4\sin^2(\pi k)},\\
\label{062712-19d} \pi_{{\rm s},n}(\eta) &:=\underset{{k\in \widehat{\T}_n}}{{\hat\sum}}\frac{(1+e^{-2\pi
  i k}) \widehat{\bar p}(0,k)}{-\eta^2+2i\ga\eta+4\sin^2(\pi k)},\\
\label{c-eta}
c_n(\eta)&:=\underset{{k\in \widehat{\T}_n}}{{\hat\sum}}\frac{1+\cos(2\pi k)}{-\eta^2+4\sin^2(\pi k)+2i\ga\eta }.
\end{align}
We have, see the Appendix Sections \ref{sec:app3}, \ref{sec:app4}, \ref{sec:app5}, \ref{sec:app6},  for any $\eta\in\R$
\begin{align}
\label{022012-19}
1&\lesssim |e_{{\rm s},n}(\eta)|,\\
\label{052312-19x}
|c_n(\eta)|&\lesssim \frac{1}{\sqrt{|\eta|}(1+|\eta|^{3/2})},\\
\label{042712-19c}
 |\rho_{{\rm s},n}(\eta)|&\lesssim
  \frac{1}{\eta^2+1} \log\big(1+|\eta|^{-1}\big),\\
\label{092712-19}
 |\pi_{{\rm s},n}(\eta)|&\lesssim
  \frac{1}{|\eta|+\eta^2},\qquad n\ge 1.
\end{align}
Using \eqref{022012-19} and \eqref{052312-19x} we conclude that
\begin{align}
  \label{061812-19abis}
P^{\rm s}_{n,3}
& \lesssim \frac{1}{n^2}\int_{0}^{+\infty}\frac{\sin^2(n^2\eta
  t_*/2)}{\eta} \frac{\dd\eta}{1+\eta^{3}}\lesssim  \frac{1}{n^2}\int_{0}^{1}\frac{\sin^2(n^2\eta
  t_*/2)}{\eta}\dd\eta+\frac{1}{n^2}\\
&
  \lesssim \frac{1}{n^2}\int_{0}^{n^2t_*}\frac{\sin^2(\eta
  )}{\eta}\dd\eta +\frac{1}{n^2}\lesssim \frac{\log(n^2 t_*+1)}{n^2},\qquad n\ge1.\notag
 \end{align} 
From  \eqref{022012-19}, \eqref{042712-19c} and \eqref{092712-19} we easily obtain
$
P^{\rm s}_{n,j}
\lesssim n^{-2}$, $n\ge1$, for $j=1,2$.
 By virtue of \eqref{061812-19} we  conclude \eqref{022812-19b}. Finally, Proposition \ref{prop012812-19} is proved. \qed


\subsection{Proof of Proposition \ref{lm010911-19a}}

\label{sec8}

To prove Proposition \ref{lm010911-19a}, we use once again the autonomous system of equations for the averages of the momenta \eqref{eq:qdynamics-bar}--\eqref{eq:pbd-bar}, and we will prove two estimates, on 
\[\int_0^t\big(\bar p_0(s)+\bar p_n(s)\big)\dd s \qquad \text{and}\qquad \int_0^t\big(\bar p_0(s)-\bar p_n(s)\big)\dd s.\]
Recall that $\bar\tau_+(t)=\bar \tau_+\mathbf{1}_{[0,t_*)}(t)$. Note that for
any $0\le a<b\le t_*$ we write
\begin{equation}
\label{012412-19z}
\int_a^b\bar p_x(s)\dd s=\int_{\R}\frac{e^{i\eta
    (b-a)}-1}{2\pi i\eta}e^{i\eta a}\tilde{\bar p}_x(i\eta)\dd \eta.
\end{equation}
Using \eqref{eq:diff} and \eqref{eq:sum} we conclude that
\begin{align*}
  \frac{e^{i\eta
    (b-a)}-1}{2\pi i\eta}e^{i\eta a}\;\tilde{\bar p}^{\rm diff}_{0,n}(i\eta) & =\tilde
    P_{n,1}^{\rm d}+\tilde P_{n,2}^{\rm d}+\tilde
    P_{n,3}^{\rm d}, \\
    \frac{e^{i\eta
    (b-a)}-1}{2\pi i\eta}e^{i\eta a}\;\tilde{\bar p}^{(+)}_{0,n}(i\eta) & =\tilde
    P_{n,1}^{\rm s}+\tilde P_{n,2}^{\rm s}+\tilde
    P_{n,3}^{\rm d},
\end{align*}
where, for  $\iota\in\{{\rm d},{\rm s}\}$, we denote
\begin{align*}
     \tilde P_{n,1}^{\iota}&:= \frac{1-e^{i\eta
    (b-a)}}{2\pi i n^2 \eta}e^{i\eta a}\rho_{\iota,n}\left(\frac{\eta}{n^2}\right) e_{\iota,n}^{-1}\left(\frac{\eta}{n^2}\right),\\
 \tilde P_{n,2}^{\iota}&:=\frac{e^{i\eta
    (b-a)}-1}{2\pi i n^2 \eta}e^{i\eta a}\pi_{\iota,n}\left(\frac{\eta}{n^2}\right) e_{\iota,n}^{-1}\left(\frac{\eta}{n^2}\right), \\
 \tilde P_{n,3}^{\rm d}&:=-\frac{\tilde{\bar
    \tau}_+ }{2\pi \eta^2}\big(e^{i\eta
    (b-a)}-1\big)\big(e^{-i\eta
    t_*}-1\big)e^{i\eta a} a_n\left(\frac{\eta}{n^2}\right)
  e_{{\rm d},n}^{-1}\left(\frac{\eta}{n^2}\right),\\
 \tilde P_{n,3}^{\rm s}&:= \frac{\tilde{\bar
    \tau}_+ }{2\pi i n^2\eta} \big(e^{i\eta
    (b-a)}-1\big)\big(e^{-i\eta
    t_*}-1\big)e^{i\eta a}c_n \left(\frac{\eta}{n^2}\right) e_{{\rm s},n}^{-1}\left(\frac{\eta}{n^2}\right).
\end{align*}
Hence,
\begin{align}
\label{022412-19z}
\left|\int_a^b\big(\bar p_0(s)-\bar p_n(s)\big)\dd s\right|& \le\left\{\int_{\R}\frac{
   \sin^2(\eta (b-a)/2)}{4\pi^2\eta^2}\dd
  \eta\right\}^{1/2}\left\{\int_{\R}\left|\tilde{\bar
  p}_{0,n}^{\rm diff}(i\eta)\right|^2 \dd \eta\right\}^{1/2} \notag \\ & \lesssim \frac{ (b-a)^{1/2}}{n},
\end{align}
where we have used Proposition \ref{prop012812-19}--\eqref{022812-19a}.
On the other hand
\begin{align}
\label{032412-19z}
\left|\int_0^t\big(\bar p_0(s)+\bar p_n(s)\big)\dd s\right|\lesssim \mathrm{I}_{n,1}+\mathrm{I}_{n,2}+\mathrm{I}_{n,3},
\end{align}
where
\begin{align*}
\mathrm{I}_{n,1}&:= \frac{1}{n^2}\int_{\R}\big|\sin(n^2\eta (b-a)/2)\big|\; \big|\rho_{{\rm s},n}(\eta)
  \big|\; \left|e_{{\rm s},n}(\eta)\right|^{-1}\frac{\dd\eta }{|\eta|},\\
\mathrm{I}_{n,2}&:=\frac{1}{n^2}\int_{\R}\big|\sin(n^2\eta (b-a)/2)\big|\; \big|\pi_{{\rm s},n}(\eta) \big|\; \left|e_{{\rm s},n} (\eta)\right|^{-1}\dd\eta,\\
\mathrm{I}_{n,3}&:=\frac{1}{n^2}\int_{\R}\big|\sin(n^2\eta (b-a)/2)\sin(n^2\eta t_*/2)\big|\;\big|c_n(\eta)
  \big|\big|\left|e_{{\rm s},n} (\eta)\right|^{-1}\frac{\dd\eta }{|\eta|}.
\end{align*}
Suppose that $p\in(0,\frac12]$ and $0<b-a<1$. Thanks to the estimates \eqref{022012-19}--\eqref{092712-19} we conclude that
\begin{align}
\label{032412-19az}
\mathrm{I}_{n,3}&\lesssim 
 \frac{1}{n^2}\int_{(b-a)^{-p}}^{+\infty}\frac{\dd\eta }{\eta^{2}}
+\frac{1}{n^2}\int_{0}^{(b-a)^{-p}}\frac{|\sin(n^2\eta (b-a)/2)\sin(n^2\eta
  t_*/2)|\dd\eta}{\sqrt{\eta}(1+\eta^{3/2})}\notag\\
&
\lesssim \frac{ (b-a)^{p}}{n^2}+\mathrm{I}_{n,3}^1+\mathrm{I}_{n,3}^2,
\end{align}
where $\mathrm{I}_{n,3}^j$, $j=1,2$ correspond to splitting the domain of integration
in the last integral into $[0,1]$ and $[1, (b-a)^{-p}]$, respectively.
We have
$$
|\sin x|\lesssim x^p,\quad x>0.
$$
Using this estimate to bound $|\sin(n^2\eta (b-a)/2)|$,
since $|\sin(n^2\eta
  t_*/2)|\le 1$, we bound  (recall that $p\in(0,\frac12]$):
\begin{align}
\label{032412-19azz1}
\mathrm{I}_{n,3}^1\lesssim \frac{(b-a)^p}{n^{2(1-p)}}\int_{0}^{1}\frac{\dd\eta
  }{\eta^{1/2-p}}
\lesssim 
  \frac{(b-a)^p}{n}.
\end{align}
In addition,
\begin{align}
\label{032412-19azz2}
\mathrm{I}_{n,3}^2\lesssim \frac{(b-a)^{p}}{n^{2(1-p)}}\int_{1}^{(b-a)^{-p}}\frac{\dd\eta
  }{\eta^{2-p}}
\lesssim 
  \frac{(b-a)^p}{n}.
\end{align}
As a result, we obtain
$
\mathrm{I}_{n,3}
\lesssim 
  (b-a)^p/n.
$
Estimates for $\mathrm{I}_{n,j}$, $j=1,2$ are similar. 
Hence, for any $t>0$ we can find $p>0$ such that
\begin{align}
\label{012812-19}
&\left|\int_a^b \big(\bar p_0(s)+\bar p_n(s)\big)\dd s\right|\lesssim
  \frac{(b-a)^p}{n},\qquad n\ge1,
\end{align}
and this together with \eqref{022412-19z} implies  \eqref{052412-19a}.
\qed


\subsection{Proof of Proposition \ref{cor013112-19z}}

\label{sec6.1}

Let
 $$
 \Xi_n(t) :=\sum_{ x\in\T_n} \bar  p_{x}^2(t)+\sum_{ x=1}^{n}  \bar
 r_{x}^2(t),\qquad \mathfrak{P}_n(t):=\sum_{ x\in \T_n} \bar  p_{x}^2(t).
 $$
 Multiplying the  equations \eqref{eq:qdynamics-bar},
 \eqref{eq:pdyn-bar} by $\bar  r_x(t)$
 and $\bar p_x(t)$, respectively and
 \eqref{eq:rbd-bar}, \eqref{eq:pbd-bar} by $ \bar   p_0(t)   $ and
 $\bar    p_n(t)$ and summing up we get
\begin{align}
     \frac12\frac{\dd}{\dd t}  \Xi_n(t)
 =n^2\left\{-\gamma \mathfrak{P}_n(t)-\frac12\tilde \gamma
       \bar p_0^2(t) -\frac12 
      \tilde \gamma\bar p_n^2(t)+ \; \bar p_n(t)\bar\tau_+\right\}.
      \label{eq:pdyn-bar2a}
 \end{align}
Hence, by virtue of \eqref{052412-19a}, we conclude that for any $t_*>0$
\begin{align}
     \frac12\Big( \bar\Xi_n(t) -\bar\Xi_n(0)\Big) \le \;
  n^2|\bar\tau_+|\;\bigg|\int_0^t\bar p_n(s)\dd s\bigg|\lesssim n,\qquad t\in[0,t_*],
      \label{eq:pdyn-bar2aa}
 \end{align}
and \eqref{053112-19z} follows.
From \eqref{eq:pdyn-bar2a} we obtain further
\begin{align}
     \frac12  \mathfrak{P}_n(t) \le \; \frac12 \Xi_n(0)
  -n^2\int_0^t  \mathfrak{P}_n(s)\dd s+n^2|\bar\tau_+|\;\bigg|\int_0^t\bar
  p_n(s)\dd s \bigg|.
      \label{010301-20a}
 \end{align}
Using Assumption \ref{ass1}--\eqref{E0} and Proposition \ref{lm010911-19a} we conclude that given
$t_*>0$ we have
\begin{equation}
    \mathfrak{P}_n(t) \lesssim \;n
  -n^2\int_0^t \bar{\mathfrak{P}}_n(s)\dd s
      \label{010301-20c},\quad t\in[0,t_*],\,n\ge1.
 \end{equation}
Therefore
\eqref{010301-20} follows upon an application of the Gronwall inequality.
\qed

\subsection{Proof of Lemma \ref{lm010911-19}} 


\label{sec7bis}

To prove Lemma \ref{lm010911-19}, we need an estimate for 
\[ \sup_{x\in\T_n}\big|\bar p^{(n)}_{x}(s)\big|^2,\] which can be done using the explicit formula \eqref{070411-19x}.
Note first that we have the following inequalities for $\lambda_\pm(k)$: 
\[-2\gamma \leqslant \lambda_+(k) \leqslant - \gamma  \qquad \text{and} \qquad \frac{-4\sin^2(\pi k)}{\gamma} \leqslant \lambda_-(k) \leqslant \frac{-2 \sin^2(\pi k)}{\gamma}.\]
Therefore, in order to estimate the members which appear in the right hand side of \eqref{070411-19x}, we introduce, for $\ell=0,1,2$,
\begin{align*}
&
\bar Q_\ell^{(n)}(t):= \underset{{k\in \widehat{\T}_n}}{{\hat\sum}}\frac{|\sin(\pi
  k)|^{\ell}}{|\la_-(k)-\la_+(k)| }  \Big(\exp\Big\{-\frac{2t\sin^2(\pi
  k)}{\ga}\Big\}+\exp\left\{-\ga t\right\}\Big).
\end{align*}
\begin{lemma}
\label{lm010112-19}
We have
\begin{equation}
\label{010112-19}
\bar Q_\ell^{(n)}(t)\lesssim
\frac{1}{(1+t)^{(1+\ell)/2}},\qquad t\ge0,\,n\ge1,\,\ell=0,1,2.
\end{equation}
\end{lemma}
\begin{proof} We prove the result for $\ell=2$. The argument in the remaining
cases is similar.
It suffices only to show that 
\begin{equation}
\label{010112-19a}
\underset{{k\in \widehat{\T}_n}}{{\hat\sum}}\sin^2(\pi
  k) \exp\Big\{-\frac{2t\sin^2(\pi
  k)}{\ga}\Big\}\lesssim
\frac{1}{(1+t)^{3/2}},\qquad t\ge0,\,n\ge 1.
\end{equation}
Since the left hand side of \eqref{010112-19a} is obviously bounded it
suffices to show that 
\begin{equation}
\label{010112-19b}
\underset{{k\in \widehat{\T}_n}}{{\hat\sum}}\sin^2(\pi
  k) \exp\Big\{-\frac{2t\sin^2(\pi
  k)}{\ga}\Big\}\lesssim
\frac{1}{t^{3/2}}, \qquad t\ge0,\,n\ge 1,
\end{equation}
which follows easily from the fact that  $\sin^2(\pi
  k)\sim k^2$, as $|k|\ll 1$.
\end{proof}

From \eqref{070411-19x} we get that for any $t_*>0$
\begin{align}
  \label{070411-19xs}
  & 
\sup_{x\in\T_n}|\bar p_x(t)|\mathbf{1}_{[0,t_*]}(t)
   \le \mathrm{I}_{n}(t)+ \mathrm{II}_{n}(t)+ \mathrm{III}_{n}(t),
\end{align}
where
\begin{align*}
\mathrm{I}_{n}(t)&:=\bar Q_{0}^{(n)}\big(n^2t\big) \mathbf{1}_{[0,t_*]}(t), \vphantom{\bigg(} \\
\mathrm{II}_{n}(t)&:=\int_0^tq_{1}^{(n)}\big(t-s\big)\big|\bar
    p_{0,n}^{\rm diff}(s) \big| \mathbf{1}_{[0,t_*]}(s)\dd s,\notag\\
 \mathrm{III}_{n}(t)&:=\int_0^tq_{2}^{(n)}\big(t-s\big)
    \big|\bar p_{0,n}^{\rm sum}(s)\big|\mathbf{1}_{[0,t_*]}(s) \dd s
\end{align*}
and 
 $q_{\ell}^{(n)}(t):= n^2\bar Q_\ell^{(n)}(n^2t)\mathbf{1}_{[0,t_*]}(t)$, for
 $\ell=1,2$. 
Therefore, from Lemma \ref{lm010112-19} we get
$$
\|\mathrm{I}_{n}\|_{L^2[0,+\infty)}^2=\int_0^{t_*}\big|\bar Q_{0}^{(n)}\big(n^2t\big) \big|^2\dd
t \lesssim \frac{1}{n^2}\int_0^{n^2t_*}\frac{\dd t}{1+t}\lesssim  \frac{\log (n+1)}{n^2}.
$$
We also have
\begin{align*}
\|q_{1}^{(n)}\|_{L^1[0,+\infty)}&= n^2\int_0^{t_*}\bar Q_1^{(n)}(n^2t)\dd t
\lesssim \int_0^{n^2t_*}\frac{\dd t}{1+t}\lesssim \log (n+1), \\
\|q_{2}^{(n)}\|_{L^1[0,+\infty)}	&=n^2\int_0^{t_*}\bar Q_2^{(n)}(n^2t)\dd t
\lesssim  \int_0^{n^2t_*}\frac{\dd t}{(1+t)^{3/2}}\lesssim 1, \qquad n\ge1.
\end{align*}
In order to estimate $|\bar
    p_{0,n}^{\rm diff}|$, recall Proposition \ref{prop012812-19}: using \eqref{022812-19a} and  the Young inequality for convolution we obtain
$$
\|\mathrm{II}_{n}\|_{L^2[0,+\infty)}^2\le \big\|q_{1}^{(n)}\big\|_{L^1[0,+\infty)}^2 \big\|\bar
    p_{0,n}^{\rm diff} \mathbf{1}_{[0,t_*]}\big\|_{L^2[0,+\infty)}^2 \lesssim \frac{\log^2 (n+1)}{n^2}, 
$$
and
$$
\|\mathrm{III}_{n}\|_{L^2[0,+\infty)}^2\le \big\|q_{2}^{(n)}\big\|_{L^1[0,+\infty)}^2 \big\|\bar
    p_{0,n}^{\rm sum} \mathbf{1}_{[0,t_*]}\big\|_{L^2[0,+\infty)}^2 \lesssim \frac{\log (n+1)}{n^2}.
$$
Thus, the conclusion of Lemma \ref{lm010911-19} follows. \qed

\subsection{Proof of Lemma \ref{lm010301-20}}

\label{sec6.5}

To prove Lemma \ref{lm010301-20} we need to get an expression for 
\[\sum_{x=0}^{n-1}\big(\bar r_{x+1}^{(n)}(t)-\bar
r_x^{(n)}(t)\big)^2. \]
Using \eqref{eq:qdynamics-bar} and then summing by parts we get
\begin{multline*}
\frac{\dd}{\dd t}\left(\frac{1}{n+1}\sum_{x\in\T_n}\bar
  r_x^2(t)\right)=\frac{n^2}{n+1}\sum_{x=1}^{n}\bar r_x(t)\big(\bar p_x(t)-\bar
  p_{x-1}(t)\big)\\
=\frac{n^2}{n+1}\sum_{x=1}^{n-1}\bar p_x(t)\big(\bar r_x(t)-\bar
  r_{x+1}(t)\big)+\frac{n^2}{n+1}\bar p_n(t) \bar
  r_n(t)-\frac{n^2}{n+1}\bar p_0(t) \bar r_1(t).
\end{multline*}
Computing  $\bar p_x(t)$ from \eqref{eq:pdyn-bar} we rewrite the
utmost right hand side as
\begin{align*}
&
-\frac{n^2}{2\ga(n+1)}\sum_{x=1}^{n-1}\big(\bar r_x(t)-\bar
  r_{x+1}(t)\big)^2
  -\frac{1}{2\ga(n+1)}\sum_{x=1}^{n-1}\frac{\dd\bar p_x(t)}{\dd t}\big(\bar r_x(t)-\bar
  r_{x+1}(t)\big)\\
&
+\frac{n^2}{n+1}\bar p_n(t) \bar
  r_n(t)-\frac{n^2}{n+1}\bar p_0(t) \bar r_1(t).
\end{align*}
Therefore, after integration by parts in the temporal variable, we get
\begin{align*}
&
\frac{n^2}{2\ga(n+1)}\int_0^t\sum_{x=1}^{n-1}\big(\bar r_x(s)-\bar
  r_{x+1}(s)\big)^2\dd s
=\frac{1}{n+1}\sum_{x\in\T_n}
  \big(\bar r_x^2(t)-\bar r_x^2(0)\big)
 \\
& \quad
 -\frac{n^2}{2\ga(n+1)}\int_0^t\sum_{x=1}^{n-1}\bar p_x(s)\big(\bar  p_{x+1}(s)+\bar p_{x-1}(s) - 2 \bar p_x(s)\big)\dd s\\
&\quad
 -\frac{1}{2\ga(n+1)}\bigg(\sum_{x=1}^{n-1}\bar p_x(t)\big(\bar r_x(t)-\bar
  r_{x+1}(t)\big)-\sum_{x=1}^{n-1}\bar p_x(0)\big(\bar r_x(0)-\bar
  r_{x+1}(0)\big)\bigg)\\
&\quad
+\frac{n^2}{n+1}\int_0^t\bar p_n(s) \bar
  r_n(s)\dd s-\frac{n^2}{n+1}\int_0^t\bar p_0(s) \bar r_1(s)\dd s.
\end{align*}
Using \eqref{053112-19z} from Proposition \ref{cor013112-19z} we conclude that the first and third expressions in the right hand side stay
bounded, as $n\ge1$.
Summing by parts we conclude that 
the second expression equals
\begin{align*}
&
\frac{n^2}{2\ga(n+1)}\int_0^t\sum_{x=1}^{n-2}\big(\bar
  p_{x+1}(s)- \bar  p_{x}(s)\big)^2\dd s\\
&
+\frac{n^2}{2\ga(n+1)}\int_0^t\Big(\bar p_1(s)\big(\bar
  p_1(s)- \bar  p_0(s)\big)-\bar p_{n-1}(s)\big(\bar
  p_n(s)- \bar  p_{n-1}(s)\big)\Big)\dd s.
\end{align*}
The expression stays bounded, due to  Proposition \ref{cor013112-19z}--\eqref{010301-20}.

Multiplying \eqref{eq:rbd-bar} by $\bar p_0(t)$ and integrating we get
\begin{equation}\label{eq:this}
 \frac{1}{2n}\big(\bar   p^2_0(t)-\bar   p^2_0(0)\big) =   n\int_0^t\bar r_1(s)  \bar   p_0(s)\dd s                                            
      -  n(2\gamma+\tilde \gamma) \int_0^t \bar   p_0^2(s)\dd s.
\end{equation}
From here, thanks to Proposition \ref{cor013112-19z}--\eqref{010301-20} (which controls $\int \bar p_0^2$) and thanks to Proposition \ref{prop022111-19} (which controls the left hand side of \eqref{eq:this} since $\bar p_0^2(t) \leqslant 2\bar{\mathcal{E}}_0(t)$ by convexity),  we conclude that
$$
n\left|\int_0^t\bar r_1(s)  \bar   p_0(s)\dd s  \right|\lesssim 1.
$$
Multiplying \eqref{eq:pbd-bar} by $\bar p_n(t)$ and integrating we get
\begin{align*}
 \frac{1}{2n} \big( \bar   p^2_n(t)-\bar   p^2_n(0)\big) &= n\bar \tau_+\int_0^t  \bar   p_n(s)\dd s      -\;  n\int_0^t\bar r_n(s)  \bar   p_n(s)\dd s                                            
      -  n(2\gamma+\tilde \gamma) \int_0^t \bar   p_n^2(s)\dd s.
\end{align*}
From here, thanks to Proposition \ref{cor013112-19z}--\eqref{010301-20}, Proposition \ref{lm010911-19a}  and Proposition \ref{prop022111-19} we conclude that
$$
n\left|\int_0^t\bar r_n(s)  \bar   p_n(s)\dd s  \right|\lesssim 1.
$$
 This ends the proof of \eqref{060301-20}.
\qed

\section{Entropy production}

\label{sec6}

This section is mainly devoted to proving Proposition \ref{prop022111-19}, which will be concluded in Section \ref{sec:endproof-1}. For that purpose, we will obtain an \emph{entropy production bound} stated in Proposition \ref{cor021211-19} below, in Section \ref{sec7entropy}. Its proof uses Proposition \ref{lm010911-19a}. This new result, together with the estimates given in Proposition \ref{lm010911-19a} and Proposition \ref{cor013112-19z}, will also allow us to conclude the proof of Lemma  \ref{lem:bound2} (which give all boundary estimates), in Section \ref{sec:conseque}.

\subsection{Entropy production bound}

\label{sec7entropy}

To prove Proposition \ref{prop022111-19}, we need to show that
$$
 \sup_{s\in[0,t]}\sum_{x \in \T_n} \E_n\big[ \mathcal{E}_x(s) \big]
$$
 grows at most
linearly in $n$.
We first relate this quantity to the \emph{entropy production}, as
follows: recall that $f_n(t)$ is the density of
the distribution $\mu_n(t)$ of $({\bf r}(t),{\bf p}(t))$ with respect to $\nu_T$, see \eqref{eq:nuT}.
We denote the expectation with respect to $\nu_T$ by  $\llangle \cdot \rrangle_T$.
Given a density $F\in L^2(\nu_T)$ we define the relative entropy
$H_{n,T}[F]$ of
the measure $\dd \mu:=F\dd \nu_T$, with respect to  $\nu_T$ by 
\begin{equation}
\label{rel-ent}
H_{n,T}[F]:= \llangle F\log F \rrangle_T= \int_{\Om_n} F\log F \dd \nu_{T}.
\end{equation}
We interpret $F\log  F=0$, whenever $F=0$. Finally, we denote \begin{equation}
  \label{eq:7a}
\mathbf{H}_{n,T}(t) := H_{n,T}[f_n(t)].
\end{equation}
Then, by virtue of the entropy inequality, see e.g.~
  \cite[p.~338]{KL} (and also \eqref{entropy-in-t} below), and from our assumption \eqref{eq:ass0} on the initial condition, 
we conclude that: for any $\al>0$ we can find $C_\al>0$ such that
\begin{equation}\label{eq:boundent}
 \E_n\left[\sum_{x\in\T_n}{\cal E}_x(t)\right]
  \le \frac{1}{\al}\big(C_\al n+ \mathbf{H}_{n,T}(t)\big),\qquad t\ge0.
 \end{equation}
This reduces the problem to showing a linear bound on $\mathbf{H}_{n,T}(t)$, which is the main result of this section.
\begin{proposition}
\label{cor021211-19}
Under Assumptions \ref{ass1}, 
for any $t,\,T>0$  we have
\begin{equation}
  \label{eq:10bb}
 \sup_{s\in[0,t]} \mathbf{H}_{n,T}(s) \lesssim  n,\qquad n\ge1.
\end{equation}
\end{proposition}

In order to prove Proposition  \ref{cor021211-19} (which will be
achieved in Section \ref{sec:endproof}), we first introduce another
relative entropy which takes into account the boundary temperatures
fixed at $T_-$ and $T_+$ and explains how relate them to each other.  

\subsubsection{Relative entropy of an inhomogeneous product measure}

Recall the definition of the non-homogeneous product measure $\tilde\nu$ given in \eqref{tilde-nu} and of the density $\tilde f_n(t)$ given in \eqref{eq:tildef}. 
The relative entropies $ \tilde{\mathbf{H}}_n(t)$ (defined in \eqref{eq:7}) and $\mathbf{H}_{n,T}(t)$ (defined in \eqref{rel-ent}) are related by the following formula.
\begin{proposition}
\label{prop-rel-ent}
For any $T>0$ and $n\ge1$ we have 
\begin{align}
  \mathbf{H}_{n,T}(t) = \tilde{\mathbf{H}}_n(t) & - \int_{\Om_n}  \sum_{x\in\T_n} \left(\left(\beta_x - T^{-1}\right) \mathcal E_x
    - \beta_x\bar\tau_+ r_x\right) \tilde f_n(t)  \dd \tilde{ \nu} \notag\\ 
   & - \sum_{x=1}^n \left(\mathcal G(\beta_x,\bar\tau_+) - \mathcal
      G(T^{-1},0)\right) -\frac12 \log \left(T_- T^{-1}\right),\quad t\ge0.\label{eq:8-0} 
 \end{align}
In addition, for any $t_*>0$
\begin{equation}
  \label{eq:9-0}
 \mathbf{H}_{n,T}(t) \lesssim \tilde{\mathbf{H}}_n(t) +  n,\qquad n\ge1,\,t\in[0,t_*].
\end{equation}
\end{proposition}
\begin{proof}
Formula \eqref{eq:8-0} can be obtained by a direct calculation.
To prove the bound \eqref{eq:9-0} note first that
one can choose a sufficiently small $\al>0$  so that
\begin{equation}
\label{030611-19}
\sup_{n\ge1}\frac1n\log\left\{\int_{\Om_n}
     \exp\Big\{-\al \sum_{x\in\T_n} \left(\left(\beta_x - T^{-1}\right) \mathcal E_x
    - \beta_x\bar\tau_+ r_x\right) 
  \Big\}\tilde{\nu} (\dd\mathbf r, \dd\mathbf p)\right\}=:C_\al<+\infty.
\end{equation}
From the entropy inequality  we write 
\begin{align*}
&-\sum_{x\in\T_n}\int_{\Om_n}  \left(\left(\beta_x - T^{-1}\right) \mathcal E_x
    - \beta_x\bar\tau_+ r_x\right) \tilde f_n(t) \dd \tilde{\nu}
\le \frac{1}{\al}\left( C_\al n+\tilde{\mathbf{H}}_n(t)\right).
\end{align*}
Thus \eqref{eq:9-0}
follows from \eqref{eq:8-0}.\end{proof}

\subsubsection{Estimate of $\tilde{\mathbf{H}}_{n}(t)$}
Next step consists in estimating $\tilde{\mathbf{H}}_{n}(t)$ by computing its derivative.
Using the regularity theory for solutions of stochastic differential
equations and Duhamel formula, see e.g.~Section 8 of \cite{bo2}, we
can argue that $\tilde f_n(t, {\bf r},{\bf p})$ is
twice continuously differentiable in $({\bf r},{\bf p})$ and once in
$t$,
provided that $\tilde f_n(0)\in C^2(\Om_n)$, which is the case, due to  \eqref{eq:ass0}. 
Using the dynamics \eqref{eq:qdynamicsbulk}--\eqref{eq:qdynamicsbound} we   therefore obtain:
\begin{align}
  \label{eq:10}
 \tilde{\mathbf{H}}_n'(t) 
=-\left(T_+^{-1} - T_-^{-1}\right) n\sum_{x=0}^{n-1}  \bbE_n\Big[ j_{x,x+1}(t)\Big]
+n^2 T_+^{-1} \bar\tau_+\bar p_n^{(n)}(t) - n^2 \mathbf D(\tilde f_n(t)),
\end{align}
where $j_{x,x+1} (t):= j_{x,x+1} ({\bf r}(t), {\bf p}(t))$, with $j_{x,x+1}$ given in \eqref{jx}, and
the operator $\mathbf{D}$ is defined for any $F\ge0$ such that
    $F\log F\in L^1(\tilde \nu)$ and $(\nabla_{\bf
    p}F)^{1/2}\in L^2(\tilde \nu)$, by
\begin{equation}
  \label{eq:11Dtilde}
  \mathbf D(F) := \gamma \sum_{x\in\T_n} {\mc D}_{x,\beta} (F) +\tilde\gamma  \sum_{x=0,n}
 T_x \; \int_{\Om_n}  \frac{\big(\partial_{p_x} F\big)^2}{F} \dd \tilde\nu,
\end{equation}
with
\begin{equation}  
    \mc D_{x,\beta}(F) := - \int_{\Om_n}F({\mathbf r},{\mathbf p}) \log\bigg(\frac{F({\mathbf r},{\mathbf p}^x
          )}{F({\mathbf r},{\mathbf p})} \bigg)\dd\tilde \nu.
    \end{equation} 
It is standard to show, using the inequality $a \log(b/a) \leqslant 2 \sqrt{a}(\sqrt b -\sqrt a)$ for any $a,b>0$, that: for any positive, measurable function $F$  on $\Omega_n$ satisfying $\int_{\Omega_n} F(\bf r,\bf p)\dd\tilde\nu=1$, and any $x \in \T_n$, 
\begin{equation} \mc D_{x,\beta}(F)  \geqslant \int_{\Om_n} \big(\sqrt{F({\mathbf r},{\mathbf
      p}^x)} - \sqrt{F({\mathbf r},{\mathbf p})}\big)^2 \dd\tilde \nu \geqslant 0. \label{eq:dirich-pos}
\end{equation}
The main result of this section is the following:
\begin{proposition}[Entropy production]
\label{thm-entropy-production} 
For any $t>0$ we have
\begin{equation}
 \int_0^t \mathbf D(\tilde f_n(s))\dd s \lesssim \frac{1}{n^2}\left(\tilde{\mathbf{H}}_n(0) +n\right)\label{eq:88t}
\end{equation}
and
\begin{equation}
\sup_{s\in[0,t]}\tilde{\mathbf{H}}_n(s) \lesssim \tilde{\mathbf{H}}_n(0) + n, \qquad n\ge1.\label{eq:88tt}
\end{equation}
\end{proposition}
\begin{proof}
From \eqref{eq:10} we get
\begin{align}
  \label{eq:10a}
\tilde{\mathbf{H}}_n(t)=\tilde{\mathbf{H}}_n(0)+ \mathrm{I}_n+\mathrm{II}_n+\mathrm{III}_n,
\end{align}
where
\begin{align*}
\mathrm{I}_n & := -\left(T_+^{-1} - T_-^{-1}\right) n\sum_{x=0}^{n-1}\int_0^t
  \bbE_n\Big[ j_{x,x+1}(s)\Big]\dd s ,\\
\mathrm{II}_n & :=n^2 T_+^{-1} \bar\tau_+\int_0^t\bar p_n^{(n)}(s)\dd s,
\\
\mathrm{III}_n& :=  - n^2 \int_0^t\mathbf D(\tilde f_n(s))\dd s \leqslant 0, 
\end{align*}
where the last inequality follows from \eqref{eq:11Dtilde} and \eqref{eq:dirich-pos}. We now estimate $\mathrm{I}_n$, $\mathrm{II}_n$.

\subsubsection*{(i) Estimates of $\,\mathrm{I}_n$}

Recall the fluctuation-dissipation relation \eqref{eq:fd1-n}
and recall also the notation $g_x(t):=g_x({\bf r}(t),{\bf p}(t))$ (and similarly for other local functions). We  write 
\begin{align*}
|\mathrm{I}_n|\le \mathrm{I}_{n,1}+\mathrm{I}_{n,2}+\mathrm{I}_{n,3}, \\
\end{align*}
where
\begin{align}
\mathrm{I}_{n,1}&:= \frac{1}{n}\left|T_+^{-1} - T_-^{-1}\right| \left|\sum_{x=1}^{n-1}
  \bbE_n\Big[ g_x(t)-g_x(0)\Big]\right|,\label{In1}\\
\mathrm{I}_{n,2}&:= \left|T_+^{-1} - T_-^{-1}\right| n\left|\int_0^t
  \bbE_n\Big[ V_n(s)-V_1(s)\Big]\dd s\right|,\label{In2}\\
\mathrm{I}_{n,3}&:=\left|T_+^{-1} - T_-^{-1}\right| n\left|\int_0^t
  \bbE_n\Big[ j_{0,1}(s)\Big]\dd s\right|.\label{In3}
\end{align}
 To deal with $\mathrm{I}_{n,1}$ we invoke   the entropy inequality: 
for any $\al>0$ we have
\begin{equation}
\label{entropy-in-t}
\int_{\Om_n}  \Big(\sum_{x\in\T_n}{\cal E}_x\Big) \tilde f_n(t) \dd
 \tilde\nu
\le \frac{1}{\al}\left\{\log\bigg(\int
     \exp\left\{ \frac\al 2 \sum_{x\in\T_n}(p_x^2+r_x^2)
     \right\}\dd\tilde \nu \bigg)+\tilde{\mathbf{H}}_{n}(t)\right\}
  \end{equation}
for any $t\ge0$. Recalling the definition of $g_x$ given in \eqref{eq:defg} and choosing $\al>0$ sufficiently small
it allows us to estimate
\begin{equation*}
\mathrm{I}_{n,1}  \lesssim \frac{1}{n}\sum_{x\in \T_n}\bar{\cal E}_x(0)+
  \frac{1}{n}\sum_{x\in \T_n}\bar{\cal E}_x(t) \lesssim \frac{1}{n}\sum_{x\in \T_n}\bar{\cal
  E}_x(0)+1+\frac{1}{n}\tilde{\mathbf{H}}_n(t),\qquad n\ge 1,\,t\ge0.
\end{equation*}
To deal with $\mathrm{I}_{n,2}$, which involves boundary terms, we shall need some auxiliary 
estimates.
\begin{lemma}
\label{lm011401-20}
For any $t_*\ge0$ we have: for $ n\ge 1,\,t\in[0,t_*]$,
  \begin{equation}\label{eq:rn-21xx}
  \mathbb{E}_n\left[ \int_0^t\big(p_1^2(s)+ p_n^2(s)\big)\dd s\right]  \lesssim
1+\int_0^t \mathbf D(\tilde
f_n(s))\dd s
\end{equation} 
and
\begin{align}\label{eq:rn-21x}
  \mathbb{E}_n\left[ \int_0^t \big(r_1^2(s)  +r_n^2(s)\big) \dd s\right]    \lesssim & \;
 1+ \frac{1}{n^2}\tilde{\mathbf{H}}_n(t)+  
  \int_0^t \mathbf D(\tilde
f_n(s))\dd s \notag\\
&
+\bigg\{1+\int_0^t\tilde{\mathbf{H}}_n(s)\dd s\bigg\}^{1/2}\;\bigg\{\int_0^t \mathbf D(\tilde f_n(s))\dd s\bigg\}^{1/2}.
\end{align} 
\end{lemma}
\proof
Note that
$$
\E_n \big[p_n^2(s)\big]=\frac{1}{T_+}\int_{\Om_n}\tilde f_n(s)\dd\tilde
  \nu+\frac{1}{T_+}\int_{\Om_n}p_n\partial_{p_n}\tilde f_n(s)\dd\tilde
  \nu.
$$
Therefore, by the Cauchy-Schwarz inequality
\begin{align*}
\int_0^t \E_n\big[ p_n^2(s)\big]\dd s &\lesssim 1+\bigg|\int_0^t\int_{\Om_n}p_{n}\; (\tilde f_n(s))^{1/2}\frac{\partial_{p_n}\tilde
   f_n(s)}{(\tilde f_n(s))^{1/2}}\; \dd\tilde\nu \dd s  \bigg|\\
&
\le 1+\bigg\{\int_0^t\int_{\Om_n}p_{n}^2\tilde f_n(s)
  \dd\tilde\nu \dd s \bigg\}^{1/2}\;\bigg\{\int_0^t\int_{\Om_n}\frac{(\partial_{p_n}\tilde
   f_n(s))^2}{\tilde f_n(s)}\dd\tilde\nu \dd s \bigg\}^{1/2}\\
&
\le 1+\bigg\{\int_0^t \E_n \big[p_n^2(s)\big]\dd s\bigg\}^{1/2}\;\bigg\{\int_0^t \mathbf D(\tilde f_n(s))\dd s\bigg\}^{1/2}.
\end{align*}
A similar calculation can be made at the left boundary point. This yields
\eqref{eq:rn-21xx}.

To prove \eqref{eq:rn-21x} note that, by \eqref{eq:qdynamicsbulk}--\eqref{eq:qdynamicsbound}, 
\begin{align}
&     n^{-2} \bbE_n\big[p_n(t) r_n(t)-p_n(0) r_n(0)\big]\label{eq:b111x} \\
&= \int_0^t \bbE_n\big[p_n(s)(p_{n}(s) - p_{n-1}(s)) + (\bar\tau_+-r_n(s)) r_n(s) - (\tilde\gamma +2\gamma)
    p_n(s) r_n(s) \big]\dd s .\notag
  \end{align}
Thanks to the assumption on the initial energy bound \eqref{E0} we have
\begin{align}
  \E_n \left[\int_0^t r_n^2(s) \dd s \right]\lesssim  & \;  \int_0^t \E_n\big[ p_n^2(s)\big]\dd s
  +\left|\int_0^t \E_n \big[p_{n}(s) p_{n-1}(s)\big]\dd s\right| \notag \\ 
  & + \left\{\int_0^t \E\big[r_n^2(s)\big]\dd s\right\}^{1/2} \notag + (\tilde\gamma +2\gamma)
   \left|\int_0^t \bbE_n\big[ p_n(s) r_n(s)\big]\dd s\right| \notag \\
   & +n^{-2}\E_n\big[ p_n^2(t)+ r_n^2(t)\big]+\frac{1}{n}, \qquad n\ge1, t\ge0.\label{eq:b111ax}
  \end{align}
Similarly we get
\begin{equation}
\label{041401-20}
\left|\int_0^t \E_n\big[ p_n(s)r_n(s)\big]\dd s\right|\lesssim\bigg\{\int_0^t \E_n \big[r_n^2(s)\big]\dd s\bigg\}^{1/2}\;\bigg\{\int_0^t \mathbf D(\tilde f_n(s))\dd s\bigg\}^{1/2}
\end{equation}
and
$$
\left|\int_0^t \E_n\big[ p_n(s)p_{n-1}(s)\big]\dd s\right|\lesssim\bigg\{\int_0^t \E_n \big[p_{n-1}^2(s)\big]\dd s\bigg\}^{1/2}\;\bigg\{\int_0^t \mathbf D(\tilde f_n(s))\dd s\bigg\}^{1/2}
$$
for $n\ge1$, $t\ge0$.
Using the entropy inequality the last bound leads to 
\begin{equation}
\label{051401-20}
\left|\int_0^t \E_n\big[ p_n(s)p_{n-1}(s)\big]\dd
  s\right|\lesssim\bigg\{1+\int_0^t\tilde{\mathbf{H}}_n(s)\dd
s\bigg\}^{1/2}\;\bigg\{\int_0^t \mathbf D(\tilde f_n(s))\dd
s\bigg\}^{1/2},
\end{equation}
for $n\ge1, \,t\ge0.$
By the entropy inequality we  also get
\begin{equation}
\label{061401-20} 
n^{-2}\E_n\Big[ p_n^2(t)+ r_n^2(t)\Big]\lesssim \frac{1}{n^2}\Big(1+\tilde{\mathbf{H}}_n(t)\Big).
\end{equation}
Substituting these bounds into \eqref{eq:b111ax} we conclude
\eqref{eq:rn-21x}. The proofs for the case of the left boundary point are analogous.
\qed

\bigskip

Returning to the proof of Proposition \ref{thm-entropy-production},
with the help of Lemma \ref{lm011401-20}, we get 
\begin{align*}
\mathrm{I}_{n,2}\lesssim   n 
& +n\int_0^t \mathbf D(\tilde
f_n(s))\dd s+ \frac{1}{n}\tilde{\mathbf{H}}_n(t)\notag\\
&
+n\bigg\{1+\int_0^t\tilde{\mathbf{H}}_n(s)\dd s\bigg\}^{1/2}\; \bigg\{\int_0^t \mathbf D(\tilde f_n(s))\dd s\bigg\}^{1/2},\qquad n\ge 1,\,t\ge0.
\end{align*}
By an application of the Young inequality we conclude  that for any
$\al>0$ we choose $C>0$ so that
\begin{equation}
\label{In2a}
\mathrm{I}_{n,2}\lesssim  C n
+\al n^2\int_0^t \mathbf D(\tilde
f_n(s))\dd s+ \frac{C}{n}\tilde{\mathbf{H}}_n(t)
+C\int_0^t\tilde{\mathbf{H}}_n(s)\dd s,\qquad n\ge 1,\,t\ge0.
\end{equation}
An analogous bound holds for $\mathrm{I}_{n,3}$. Therefore we conclude that for
any $\al>0$,
\begin{align}
\label{In}
\mathrm{I}_{n}\lesssim n +\al n^2\int_0^t \mathbf D(\tilde
f_n(s))\dd s +\frac{1}{n}\tilde{\mathbf{H}}_n(t) +\int_0^t\tilde{\mathbf{H}}_n(s)\dd s,\qquad n\ge 1,\,t\ge0.
\end{align}

\subsubsection*{(ii) Estimates of $\mathrm{II}_n$}

To estimate $\mathrm{II}_n$  we need Proposition \ref{lm010911-19a}. Thanks to \eqref{052412-19a} we conclude that 
\begin{equation}
\label{061111-19}
|\mathrm{II}_n| =n^2 T_+^{-1} |\bar\tau_+|\left|\int_0^t\bar p_n ^{(n)} (s)\dd
  s\right|\lesssim n,\quad n\ge1.
\end{equation}
Choosing $\al$ sufficiently small in \eqref{In} and substituting
from \eqref{In} and \eqref{061111-19}  into \eqref{eq:10}
we conclude that there exists $c>0$, for which
\begin{align}
  \label{eq:10aa}
 & \tilde{\mathbf{H}}_n(t)\lesssim \tilde{\mathbf{H}}_n(0) +n+ \int_0^t
\tilde{\mathbf{H}}_n(s)\dd s
  - cn^2\int_0^t \mathbf D(\tilde f_n(s))\dd s. 
\end{align}
This, by an application of the Gronwall inequality, in particular implies that 
\begin{align}
  \label{eq:10b}
  \tilde{\mathbf{H}}_n(t)\lesssim \tilde{\mathbf{H}}_n(0) +n.
\end{align}
Thus \eqref{eq:88tt} follows. Estimate \eqref{eq:88t} is an easy
consequence of \eqref{eq:10aa} and  \eqref{eq:88tt}. 
\end{proof}

\subsubsection{Proof of Proposition {\ref{cor021211-19}}} \label{sec:endproof} From the assumption \eqref{eq:ass0}, one has
\[  \tilde{\mathbf{H}}_n(0) \lesssim n, \qquad n \geqslant 1.\]
Therefore, Proposition \ref{cor021211-19} directly follows from \eqref{eq:9-0} and \eqref{eq:88tt}.

\subsection{The end  of the proof of Proposition \ref{prop022111-19}} \label{sec:endproof-1} 

Proposition \ref{prop022111-19} directly follows from the entropy
inequality \eqref{entropy-in-t} and bound \eqref{eq:9-0}.
 
\subsection{Boundary estimates: proof of Lemma \ref{lem:bound2}}
\label{sec:conseque}
The entropy production bound from Proposition \ref{thm-entropy-production} is also crucial  in order to get information on the behavior of boundary quantities. We prove here all the estimates of Lemma \ref{lem:bound2}. 

\begin{proof}[Proof of Lemma \ref{lem:bound2}, estimate \eqref{ex-p0p1}]
We start with the right boundary point $x=n$. The proof  for $x=1$ is similar. Using the definition \eqref{tilde-nu} we  write
\begin{align}
\label{023112-19a}
  \bigg|\int_0^t& \E_n  \big[p_n(s)p_{n-1}(s)
   \big]\dd s\bigg| =\bigg|\int_0^t \dd s\int_{\Om_n}p_np_{n-1}\tilde f_n(s) \dd\tilde\nu
   \bigg| \notag \\ 
&
=T_+^{-1}\bigg|\int_0^t\dd s\int_{\Om_n}p_{n-1}(\tilde f_n(s))^{1/2}\frac{\partial_{p_n}\tilde
   f_n(s)}{(\tilde f_n(s))^{1/2}}\dd\tilde\nu \bigg|\notag\\
&
\le T_+^{-1}\bigg\{\int_0^t\dd s\int_{\Om_n}p_{n-1}^2\tilde f_n(s)
  \dd\tilde\nu\bigg\}^{1/2}\;\bigg\{\int_0^t\dd s\int_{\Om_n}\frac{(\partial_{p_n}\tilde
   f_n(s))^2}{\tilde f_n(s)}\dd\tilde\nu\bigg\}^{1/2}.
  \end{align}
Invoking the entropy production bound \eqref{eq:88t} we conclude that 
\begin{align}
\label{023112-19b}
\bigg|\int_0^t\E_n\big[p_n(s)p_{n-1}(s)
   \big]\dd s\bigg|
\lesssim\frac{1}{n}\bigg\{\E_n\left[\int_0^tp_{n-1}^2(s) \dd
  s\right]\bigg\}^{1/2}\lesssim \frac{1}{\sqrt{n}},\quad n\ge1,
  \end{align}
in light of the energy estimate \eqref{022111-19} of Proposition \ref{prop022111-19}, which is now proved.
\end{proof}

\begin{proof}[Proof of Lemma \ref{lem:bound2}, estimate  \eqref{ex-2-r}]
To show \eqref{ex-2-r} note that
  \begin{align}
\label{013112-19}
    \left|\mathbb{E}_n\left[ \int_0^t r_n(s) p_n(s) \dd s \right]\right|
      & = \left|\int_0^t \dd s \int r_n \partial_{p_n}\tilde  f_n(s)
        \dd\tilde\nu\right| \\
 & \le  \bigg\{\int_0^t \dd s \int r_n^2\tilde  f_n(s) \dd\tilde\nu\bigg\}^{1/2}  \bigg\{ \int_0^t  \dd s \int \frac{\left( \partial_{p_n} \tilde f_n(s) \right)^2}{\tilde f_n(s)} \dd\tilde\nu \bigg\}^{1/2}\notag\\
     & \lesssim  
     \frac{1}{ \sqrt{n}} \left\{\mathbb{E}_n\left[ \int_0^t r_n^2(s)
       \dd s\right] \right\}^{1/2} \lesssim \frac{1}{\sqrt{n}}, \notag
  \end{align}
in light of Proposition \ref{prop022111-19}--\eqref{022111-19}.
  The proof for the left boundary is similar.
\end{proof}

\begin{proof}[Proof of Lemma \ref{lem:bound2}, estimate \eqref{ex-2-l1}]
From the time evolution of the dynamics \eqref{eq:qdynamicsbound} (see also \eqref{eq:pbd-bar}) we obtain for $x=n$
 \begin{equation}
\frac{1}{n^2} \Big(\bar    p_n^{(n)}(t)- \bar    p_n^{(n)}(0)\Big)=  \int_0^t\big(\bar\tau_+  -  \bar r_n^{(n)}(s)\big)\dd s -(2\gamma +  \tilde \gamma) \int_0^t\bar    p_n^{(n)}(s) \dd s \label{eq:pbd-bar2a}.
   \end{equation}
Using the energy bound \eqref{022111-19} we conclude that the right hand side is of
order of magnitude $n^{-3/2}$
as $n\to+\infty$. Thanks to Proposition \ref{lm010911-19a} we conclude
\eqref{ex-2-l1}. The proof for $x=1$ is analogous.
\end{proof}

\begin{proof}[Proof of Lemma \ref{lem:bound2}, estimate \eqref{eq:ex-1}]
We show the proof for $x=n$ only, as the argument for $x=0$ is analogous.   Note that, 
  \begin{equation*}
     \mathbb{E}_n\bigg[   \int_0^t  \left(p_{n}^2(s) - T_+\right)\dd s\bigg] 
=
      T_+  \int_0^t  \dd s \int_{\Om_n} p_{n} \partial_{p_n}\tilde  f_n(s) \dd\tilde \nu.
  \end{equation*}
Thus, by the Cauchy-Schwarz inequality
\begin{multline}
  \left|\mathbb{E}_n\bigg[   \int_0^t  \left(p_{n}^2(s) - T_+\right)\dd s\bigg]\right|
 \\  \le T_+ \left\{ \int_0^t  \dd s \int_{\Om_n} p_{n}^2 \tilde f_n(s) \dd\tilde \nu \right\}^{1/2}
      \bigg\{ \int_0^t  \dd s \int_{\Om_n}  \frac{\left( \partial_{p_n}\tilde  f_n(s) \right)^2}{\tilde f_n(s)} \dd\tilde \nu \bigg\}^{1/2}.
    \label{eq:7x}
\end{multline}
From \eqref{eq:b111x}, Lemma \ref{lem:bound2}--\eqref{ex-p0p1} and \eqref{ex-2-r}, which have been proved above, we 
 get
\begin{align}   
 \left|\mathbb{E}_n\left[ \int_0^t \Big(r_n^2(s)-\bar\tau_+^2-p_n^2(s)\Big)\dd s\right] \right|\lesssim
    \frac{1}{\sqrt{n}}, & \quad n\ge 1. \label{ex-2-r11}
\end{align}
Then, by Lemma \ref{lem:bound2}--\eqref{eq:13} and \eqref{ex-2-r11}, for any $t>0$ 
  \begin{equation}\label{eq:p2b}
  \int_0^t \dd s\ \int_{\Om_n}  p_{n}^2 \tilde f_n(s) \dd\tilde \nu \lesssim 1,\quad n\ge1.
  \end{equation}
 Using this estimate together with  Proposition \ref{prop022111-19}--\eqref{022111-19}
  we
 conclude \eqref{eq:ex-1}.
\end{proof}

\begin{proof}[Proof of Lemma \ref{lem:bound2}, estimate \eqref{eq:13}]
From \eqref{eq:b111x} and Lemma \ref{lem:bound2}--\eqref{ex-p0p1} we get
\begin{align}
\label{eq:b111d}
  \int_0^t \E_n \big[ r_n^2(s)\big] \dd s  \le  & \; \int_0^t \E_n
 \big[ p_n^2(s)\big]\dd s+|\bar\tau_+|\; \Big|\int_0^t \E_n  \big[r_n(s)\big] \dd s \Big| \notag\\
& + (\tilde\gamma +2\gamma)
 \bigg| \int_0^t \E_n  \big[ p_n(s) r_n(s)\big] \dd s\bigg| +o_n(1),
  \end{align}
where $o_n(1)\rightarrow0$, as $n\to+\infty$. Using the Young inequality we 
   conclude that 
  \begin{equation}\label{eq:rn-21}
  \mathbb{E}_n\left[ \int_0^t r_n^2(s)\dd s\right]  \lesssim
  \mathbb{E}_n\left[ \int_0^t p_n^2(s)\dd s\right] +1,\qquad
  n\ge 1.
\end{equation} 
We use Lemma \ref{lem:bound2}--\eqref{eq:ex-1} to conclude that
$$
\mathbb{E}_n\left[ \int_0^t r_n^2(s) \;
      \dd s\right] \lesssim 1,\qquad n\ge1.
$$
An analogous  estimate on $\mathbb{E}_n\left[ \int_0^t r_1^2(s) \;
      \dd s\right] $ follows from the same argument, using the relation
\begin{align}
     n^{-2}L(p_0 r_1) &= (p_1-p_0)p_0+  r_1^2 - (\tilde\gamma +2\gamma)
    p_0r_1\label{eq:a}
  \end{align}
and the entropy production bound at $x=0$.
\end{proof}

\begin{proof}[Proof of Lemma \ref{lem:bound2}, estimate \eqref{eq:11}]
For the right boundary current $j_{n-1,n}$, the equality follows from the definition
\eqref{eq:en-evol+}, thanks to: Proposition \ref{lm010911-19a}, Lemma \ref{lem:bound2}--\eqref{eq:ex-1}, and the energy estimate \eqref{022111-19}.  An analogous argument, using \eqref{eq:en-evol-} instead, works for
left boundary current.  
\end{proof}

\begin{proof}[Proof of Lemma \ref{lem:bound2}, estimate \eqref{ex-2-l11t}--\eqref{ex-2-r11t}] 
From \eqref{eq:b111x}, combined with the above \eqref{ex-2-r11}, and Lemma \ref{lem:bound2}--\eqref{eq:ex-1}, the result follows.
\end{proof}

\begin{proof}[Proof of Lemma \ref{lem:bound2}, estimate \eqref{eq:ex-1bis}]

 We prove that for each $t_*>0$ there exists $C>0$ such that
   \begin{align}
   \label{eq:7b}
\sup_{t\in[0,t_*]}\sup_{x\in \T_n}\bigg|\int_0^t\bbE_n j_{x-1,x}(n^2s)\dd s\bigg|\le \frac{C}{n},\quad n=1,2,\ldots.
 \end{align}
To show it we use the fluctuation-dissipation relation
\eqref{eq:fd1-n}, \eqref{eq:defg}. We complement it with the relations
at $x=0$ and $n$. Then
\begin{equation*}
  \begin{split}
   & n^{-2} L g_0 - (V_{1}-V_0)= j_{0,1},\quad\mbox{where}\\
   &V_0:=\frac 1{4 \gamma}
   p_0^2-
   \frac{\tilde\ga}{2}(p_0^2 -T_0)+\frac{\tilde\ga}{4\ga}p_0r_1,\quad g_0:=-\frac14p_0^2 + \frac{1}{4\gamma}p_0r_{1} .
\end{split}
\end{equation*}
Likewise
\begin{equation*}
  \begin{split}
   & n^{-2} L g_n - (V_{n+1}-V_n)= j_{n,n+1},\quad\mbox{where}\\
   &V_{n+1}:=\frac 1{4 \gamma}
   p_n^2+\frac12\tau_+ p_n+\frac1{4\ga}\tau_+ r_n
   -\frac{\tilde\ga}{2}(p_n^2 -T_n)-\frac{\tilde\ga}{4\ga} p_nr_n,\\
   &
   j_{n,n+1}=\tilde\ga(p_n^2-T_n)-\tau_+p_n,\quad g_n:=-\frac14p_n^2 + \frac{1}{4\gamma}p_nr_{n} .
\end{split}
\end{equation*}
Having the above relations we can write
\begin{align}
  \label{eq:7b}
\int_0^t\sum_{x=0}^{n}j_{x,x+1}(n^2 s)\dd
                  s&=\frac{1}{n^2}\int_0^t\sum_{x=0}^{n+1}Lg_{x}(n^2
                  s)\dd s+\int_0^t\sum_{x=0}^{n}\big(V_{x}(n^2
                  s)-V_{x+1}(n^2 s)\big)\dd s\notag\\
  &
    =\frac{1}{n^2}\sum_{x=0}^{n+1}\big[g_{x}(n^2
                  t)-g_{x}(0
                  )\big]+\int_0^t \big[V_{0}(n^2
                  s)-V_{n+1}(n^2 s)\big] \dd s.
\end{align}
From the already established Proposition \ref{prop022111-19} and
estimates \eqref{ex-p0p1}--\eqref{ex-2-r11t} we conclude that
for any $t_*>0$ there exists $C>0$ such that
\begin{align}
\label{031211-19}
\sup_{t\in[0,t_*]}\left|\int_0^t\sum_{x=0}^{n}\bbE_nj_{x,x+1}(n^2 s)\dd
                  s\right|\le C,\quad n=1,2,\ldots.
\end{align}
Hence, denoting $j_{-1,0}:=\tilde\ga(T_0-p_0^2)$, we can write, using \eqref{eq:en-evol}--\eqref{eq:en-evol+},
\begin{align*}
 & \left|(n+1)\int_0^t\bbE_n j_{-1,0}(n^2 s)\dd s\right|\le
  \left|\int_0^t\sum_{x=0}^{n}\bbE_nj_{x,x+1}(n^2 s)\dd
                  s\right|+ \left|\int_0^t\sum_{x=0}^{n}\bbE_n\big[j_{x,x+1}(n^2 s)-j_{-1,0}(n^2 s)\big]\dd
   s\right|
  \\
  &
    \le C+\left|\int_0^t\sum_{x=0}^{n}\sum_{y=0}^{x}\bbE_n\big[j_{y,y+1}(n^2 s)-j_{y-1,y}(n^2 s)\big]\dd
    s\right|\\
  &
    \le
    C+\frac{1}{n^2}\left|\int_0^t\sum_{x=0}^{n}\sum_{y=0}^{x}\bbE_nL{\cal E}_y(n^2 s)\dd
    s\right|=C+\frac{1}{n^2}\left|\sum_{x=0}^{n}\sum_{y=0}^{x}\big[\bbE_n{\cal E}_y(n^2 t) -\bbE_n{\cal E}_y(0)\big]\dd
    s\right|
    \le C',
\end{align*}
by virtue of \eqref{022111-19},
for some constant $C'>0$ independent of $n$. This proves that
 \begin{equation}
\label{031211-19a}
\left|\int_0^t\left(T_0-
  \E_n\big[p_0^2(s)\big] \right)\dd s \right|\lesssim\frac1{n+1}.
\end{equation}
In a similar fashion we show that
\begin{equation*}
 \left|\int_0^t\bbE_n j_{n,n+1}(n^2s)\dd s\right|\lesssim\frac1{n+1}.
\end{equation*}
This combined with Lemma \ref{lm010911-19} allows us to conclude that
\begin{equation}
\label{031211-19b}
\left|\int_0^t\left(T_n -
  \E_n\big[p_n^2(s)\big] \right)\dd s\right|\lesssim\frac1{n+1},
\end{equation}
which ends the proof of \eqref{eq:ex-1bis}.
  \end{proof}

\section{Energy balance identity and equipartition}
\label{sec:wigner}

The main result which is left to be proved is Proposition \ref{equipartition}, which describes an equipartition phenomenon between the mechanical and thermal energies. 
To prove that result, we will use the \emph{Fourier-Wigner
distributions} which permit to control the energy profiles over various frequency modes, and have been successfully used in previous works. The major difficulty here is the presence of boundary terms, which all need to be controlled. 
In Section \ref{sec:wave-function-1} we introduce definitions and write down the evolution equation satisfied by wave functions. In Section \ref{sec:balance} we obtain an \emph{energy balance identity} (Proposition \ref{thm020411-19}). The proof of Proposition \ref{equipartition} is achieved in Section \ref{sec11}.

\subsection{The wave and Wigner functions}
\label{sec:wave-function-1}

In the present section we restore the superscript $n$ when referring
to the mean and fluctuation of the stretch and momentum.
We define the \emph{fluctuating wave function} as
\begin{equation}
  \label{eq:wfx}
  \tilde \psi_x^{(n)}(t) =\tilde  r_x^{(n)}(t)  + i \tilde  p_x^{(n)}(t) , \qquad  x \in\T_n, t\geqslant 0,
\end{equation}
and its  Fourier transform,
\begin{equation}
  \label{eq:wf}
  \widehat{ \tilde\psi}^{(n)} (t,k) = {\widehat{ \tilde  r}}^{(n)} (t,k) + i {\widehat{ \tilde  p}}^{(n)} (t,k) , \qquad k\in \hat \T_n, t\ge 0.
\end{equation}
The wave function extends to a periodic function on $\frac{1}{n+1}\Z $, by
letting 
$
\widehat{\tilde\psi}^{(n)}(k +1) = \widehat{\tilde \psi}^{(n)}(k)$  for any 
$k\in \widehat \T_n.
$ 
In particular $\widehat{\tilde\psi}^{(n)}\left(k+\frac{\eta}{n+1}\right)$ is well defined for any $\eta\in \Z$.
Then for $k\in \widehat \T_n, \eta\in\Z, t \ge 0$ we define the Fourier-Wigner functions: 
\begin{align*}
\tilde   W^+_n(t,\eta,k)& := \frac 1{2(n+1)} \mathbb E_n\left[  \widehat{\tilde\psi}^{(n)}\left(t,k+\tfrac{\eta}{n+1}\right) \big[\widehat{\tilde\psi}^{(n)}\big]^{\star}(t,k)\right],\\
 \tilde    W^-_n(t, \eta,k)& := \frac 1{2(n+1)} \mathbb E_n\left[  \big[\widehat{\tilde\psi}^{(n)}\big]^\star\left(t,-k-\tfrac{\eta}{n+1}\right) \widehat{\tilde\psi}^{(n)}(-k)\right]=(\tilde W^+_n)^\star(t,-\eta,-k),
\\
 \tilde    Y^+_n(t, \eta,k) &:= \frac 1{2(n+1)} \mathbb E_n\left[
    \widehat{\tilde\psi}^{(n)}\left(t,k+\tfrac{\eta}{n+1}\right)
    \widehat{\tilde\psi}^{(n)}(t,-k)\right],\\
\tilde    Y^-_n(t, \eta,k) &:= \frac 1{2(n+1)} \mathbb E_n\left[
    \big[\widehat{\tilde\psi}^{(n)}\big]^\star\left(t,-k-\tfrac{\eta}{n+1}\right)
   \big[\widehat{\tilde\psi}^{(n)}\big]^\star(k)\right]=(\tilde Y^+_n)^\star(t,-\eta,-k).
\end{align*}
As a direct corollary from  Proposition \ref{prop022111-19} we
conclude the following bound: for any $t>0$ we have
\begin{align}
\label{042111-19}
\sum_{\iota=\pm}\bigg(\sup_{s\in[0,t]}\sup_{\eta\in\Z}\underset{{k\in \widehat{\T}_n}}{{\hat\sum}} | \tilde  W^\iota_n(s, \eta,k)|+\sup_{s\in[0,t]}\sup_{\eta\in\Z}\underset{{k\in \widehat{\T}_n}}{{\hat\sum}}| \tilde  Y^\iota_n(s,\eta,k)| \bigg)<+\infty.
\end{align}
%
%
%
%
%
%
%
%
%
%
%
%
%
%
Note that, similarly to \eqref{eq:wfx}--\eqref{eq:wf}, we can also define the (full) wave function $\psi_x^{(n)}(t)=r_x^{(n)}(t)+ip_x^{(n)}(t)$ and its Fourier transform $\hat\psi^{(n)}(t,k)$. Using \eqref{eq:qdynamicsbulk}--\eqref{eq:qdynamicsbound} we conclude that
the fluctuating  wave function satisfies
\begin{align}
\label{010502-20psi}
 &
 \dd\hat{\tilde\psi}^{(n)}(t,k)= -{n^2} \left(2i \sin^2(\pi k) \hat{\tilde\psi}^{(n)}(t,k) + \sin(2\pi k)
  [\hat{\tilde\psi}^{(n)}]^-(t,k) \right) \dd t \\
&
- \ga n^2\left\{\sum_{\iota=\pm} \iota[\hat{\tilde\psi}^{(n)}]^\iota (t,k) \right\}
   \dd t
-  \underset{{k'\in \widehat{\T}_n}}{{\hat\sum}} \left\{\sum_{\iota=\pm} \iota[\hat\psi^{(n)}]^\iota(t^-,k-k') 
   \right\} \dd\widehat{\tilde{\mathcal
    N}}(t,k')+\dd\hat{\tilde{\cal R}}_n(t,k),\notag
\end{align}
where  $
\widehat{\tilde{\mathcal
    N}}(t,k):=\widehat{\mathcal{N}}(  t,  k)-\ga n^2 t(n+1)\delta_{k,0}
$ is a martingale, with
$$
\widehat{{\mathcal
    N}}(t,k):=\sum_{x\in \T_n} \mathcal N_x(\gamma n^2 t) e^{-2i\pi x k},
$$
where $[\hat\psi^{(n)}]^\pm(t,k)$ are defined by \eqref{iota} 
and finally
  \[
\dd\hat{\tilde{\cal R}}_n(t,k):=n^2(\tilde p_n(t)-\tilde p_0(t))
 +  i\sum_{x=0,n}e^{-2\pi i x k} \big( -\tilde\gamma n^2 \tilde p_{x}(t) \dd t+ n\sqrt{2\tilde\gamma T_x}
  \dd w_{x}(t)\big).
 \]
Here $\delta_{x,y}$ is the usual Kronecker delta function, which equals 1 if $x=y$ and 0 otherwise. 
The process $\widehat{{\mathcal
    N}}(t,k)$ is a semi-martingale whose mean and covariation can be
computed from the relations
$
\langle \dd\widehat{\mathcal{N}}(  t,  k)\rangle=\gamma n^2(n+1)  \delta_{k,0} \dd t,$ and $
\langle \dd\widehat{\mathcal{N}}(  t,  k),\dd\widehat{\mathcal{N}}(t,k')\rangle=\gamma n^2(n+1) t \delta_{k,-k'} \dd t.
$

\subsection{Energy balance for the fluctuating Wigner functions} \label{sec:balance}

After straightforward computations one gets, for $\iota=\pm$:
\begin{align}
\partial_t \tilde W^\iota_n = & \; \vphantom{\int} \iota\left(-in  (\delta_ns)\;  \tilde W_n^\iota -n^2\sin(2\pi k)\; \tilde Y_n^\iota-n^2\sin\big(2\pi\big(k+\tfrac{\eta}{n+1}\big)\big)\; \tilde Y_n^{-\iota}\right)\label{eq:closed-syst1f}  \\
&
 + \gamma n^2\; \bb
 L\big(\tilde W_n^++\tilde W_n^--{\tilde Y^+_n-\tilde Y_n^-}\big)+\iota\gamma
n^2\big(\tilde W_n^--\tilde W_n^+\big)\vphantom{\int}\notag\\
&
+\gamma n^2 \bar U_n(t,\eta)
 +\frac{\tilde \ga n^2}{n+1}\sum_{x=0,n} e^{-2\pi i x\frac\eta{n+1}} T_x \notag\\
&
+\frac{n^2}{2(n+1)}\;\bb E_n\left[
    \hat{\tilde Z}_n^{-\iota}(t, k)[\hat{\tilde \psi}^{(n)}]^{\iota}\left(t,k+\tfrac{\eta}{n+1}\right)+
 \hat{\tilde Z}_n^{\iota}\left(t,k+\tfrac{\eta}{n+1}\right)[\hat{\tilde \psi}^{(n)}]^{-\iota}(t,k)\right], \notag
\\ ~\notag \\ 
\partial_t \tilde Y^\iota_n  = & \; \iota\left(n^2\sin(2\pi k)\; \tilde W_n^\iota-i n^2 (\sigma_n s) \; \tilde Y^\iota_n -n^2\sin\big(2\pi\big(k+\tfrac{\eta}{n+1}\big)\big)\; \tilde W_n^{-\iota} \vphantom{\int} \right)\label{eq:closed-syst2f} \\
&
+ \gamma n^2\;\bb L\big(\tilde Y^+_n+\tilde Y^-_n - {\tilde W_n^+-\tilde W^-_n}\big)+\iota \gamma
n^2\big(\tilde Y_n^--\tilde Y_n^+\big)\vphantom{\int} \notag \\
&
-\gamma n^2\; \bar U_n(t,\eta)
-\frac{\tilde \ga n^2}{n+1}\sum_{x=0,n} e^{-2\pi i x\frac\eta{n+1}} T_x\vphantom{\int}\notag\\
&\vphantom{\int}
+\frac{n^2}{2(n+1)}\; \bb E_n\left[
    \hat{\tilde Z}_n^\iota(t,-k)\hat{\tilde\psi}^\iota\left(t,k+\tfrac{\eta}{n+1}\right)+
\hat{\tilde Z}_n^\iota\left(t,k+\tfrac{\eta}{n+1}\right)\hat{\tilde\psi}^{\iota}(t,-k)\right], \notag
\end{align} 
where $\bb L$  is defined by $
(\bb L f)(k):=\underset{{k'\in \widehat{\T}_n}}{{\hat\sum}}f(k') -
f(k)$ for any $f:\widehat{\T}_n \to \mathbb{C}$ 
and we let
\begin{equation}
\label{012308a} \begin{split}
(\delta_n s)(\eta,k)&:=2n\Big(\sin^2\left(\pi\big(k+\tfrac{\eta}{n+1}\big)\right)-\sin^2(\pi k)\Big), \vphantom{\Bigg\{}
\\
(\sigma_ns)(\eta,k)&:=2\Big(\sin^2\left(\pi\big(k+\tfrac{\eta}{n+1}\big)\right)+\sin^2(\pi
                     k)\Big),\\
\hat{\tilde Z}_n(t,k)&:=  \tilde p_n(t)-\tilde p_0(t)
   -\tilde\gamma i\sum_{x=0,n}e^{-2\pi i x k}\tilde p_{x}(t),\vphantom{\Bigg\{}\\
\bar U_n(t,\eta)&:=\frac 2{n+1}\underset{{k\in \widehat{\T}_n}}{{\hat\sum}}\mathbb E_n\left[  \widehat{\bar  p}^{(n)}\left(t,k+\tfrac{\eta}{n+1}\right)  [\widehat{\bar  p}^{(n)}]^\star\left(t,k\right)\right] .
 \end{split}\end{equation}
We are interested in the time evolution of the following quantity: we denote 
\begin{equation}
\label{eq:energy1f}
\tilde{\mf E}_n(t):=\sum_{\eta\in\T_n}\underset{{k\in \widehat{\T}_n}}{{\hat\sum}}
\big(|\tilde W_n^+|^2+|\tilde W_n^-|^2 + |\tilde Y_n^+|^2+|\tilde Y_n^+|^2 \big)(t,\eta,k).\end{equation}
One can check that
\begin{align*} \tilde{\mf E}_n(t) 
&=\frac 1{2(n+1)} 
\sum_{x,x'}\bigg\{\left(\mathbb E_n\left[ \tilde p_x^{(n)}(t)
  \tilde p_{x'}^{(n)}(t)\right] \right)^2+\left(\mathbb E_n\left[ \tilde r_x^{(n)}(t)
  \tilde r_{x'}^{(n)}(t)\right] \right)^2 \notag \\ & \quad \qquad \qquad \qquad +2 \left (\mathbb E_n\left[ \tilde p_x^{(n)}(t)
  \tilde r_{x'}^{(n)}(t)\right] \right)^2\bigg\}.\notag
\end{align*}
After a tedious but direct calculation, using \eqref{eq:closed-syst1f}--\eqref{eq:closed-syst2f}, we get the following identity
\begin{align}
\label{energy-evol3ff}
&\frac12 \partial_t \tilde{\mf E}_n(t) 
=   \frac{2\tilde \ga n^2}{n+1} \sum_{x=0,n}\left(T_x-\E_n\left[\tilde p_x^{(n)}(t)\right]^2\right)\E_n\left[\tilde p_x^{(n)}(t)\right]^2 \notag\\ & +\frac{4\ga n^2}{n+1}\sum_{x\in\T_n}\left(\bar p_x^{(n)}(t)\right)^2\E_n\left[ \big(\tilde p_x^{(n)}(t)\big)^2\right]\nonumber
\\
&
-\frac{2n^2 \tilde\gamma}{n+1}\sum_{x=0,n}\sum_{x'\in\T_n}\left\{
  \left(\E_n\Big[\tilde p_{x}^{(n)}(t)\tilde  r_{x'}^{(n)}(t)\Big]\right)^2
+\left(\E_n\Big[\tilde p_{x}^{(n)}(t) \tilde p_{x'}^{(n)}(t)\Big]\right)^2 \right\}\notag\\
&
-\frac{4\ga
  n^2}{n+1}\left\{\sum_{x,x'\in\T_n,x\not=x'}\left(\E_n\left[\tilde p_x^{(n)}(t)\tilde p_{x'}^{(n)}(t)\right]\right)^2
 +
\sum_{x,x'\in\T_n}\left(\mathbb E_n\left[ \tilde  r_{x}^{(n)}(t) \tilde p_{x'}^{(n)}(t)\right]\right)^2\right\}.
\end{align}
The main result of this section is the following
\begin{proposition}[Energy balance identity]
\label{thm020411-19}
For any $t_*>0$ there exists $C>0$ such that 
\begin{multline}
\label{020511-19}
\tilde{\mf E}_n(t)+\ga n^2\sum_{\eta\in\T_n}\underset{{k\in \widehat{\T}_n}}{{\hat\sum}}\int_0^t\left(\left|\tilde W^+_n -\tilde W^-_n \right|^2+\left|\tilde Y^+_n -\tilde Y^-_n
  \right|^2\right)(s,\eta,k)\dd s \\
\le \tilde{\mf E}_n(0) +Ct\log^2(n+1),\qquad t\in[0,t_*],\,n\ge 1.
\end{multline}
\end{proposition}

\begin{proof}
Thanks to Lemma \ref{lem:bound2}--\eqref{eq:ex-1bis} we  write
\begin{equation*}
\sum_{x=0,n}\left(T_x-\E_n\left[\tilde p_x^{(n)}(t)\right]^2\right)\E_n\left[\tilde p_x^{(n)}(t)\right]^2 \le
\sum_{x=0,n}T_x\left(T_x-\E_n\left[\tilde p_x^{(n)}(t)\right]^2\right)\lesssim
\frac{1}{n},\quad n\ge1.
\end{equation*}
By Proposition \ref{prop022111-19} and  Lemma \ref{lm010911-19} for any
$t_*>0$ we also
have
\begin{multline}
\label{020602-20}
\sum_{x\in\T_n}\int_0^t\left(\bar p_x^{(n)}(s)\right)^2\E_n\left[\big(\tilde
  p_x^{(n)}(s)\big)^2\right]\dd s\\
\le
\left(\sup_{s\in[0,t_*]}\sum_{x\in\T_n}\E_n\left[\big(\tilde p_x^{(n)}(s)\big)^2\right]\right)
\left(\int_0^t \sup_{x\in\T_n} \left(\bar p_x^{(n)}(s)\right)^2\dd s\right)\lesssim\frac{\log^2(n+1)}{n}
\end{multline}
for $n\ge1$, $t\in[0,t_*]$.
Since 
\begin{equation*}
\sum_{\eta\in\T_n}\underset{{k\in \widehat{\T}_n}}{{\hat\sum}}\left(\left|\tilde W^+_n -\tilde W^-_n \right|^2+\left|\tilde Y^+_n -\tilde Y^-_n
  \right|^2\right) (t,\eta,k)=\frac {4}{n+1} 
  \sum_{x,x'\in \T_n}
\left\{\mathbb E_n\left[  \tilde r_{x}^{(n)}(t) \tilde p_{x'}^{(n)}(t)\right]\right\}^2
\end{equation*}
we conclude   \eqref{020511-19}.
\end{proof}

\subsection{Equipartition of energy: proof
of Proposition \ref{equipartition}}

\label{sec11}

In this section we prove Proposition \ref{equipartition}. We recall here its statement: under the assumptions of Theorem \ref{main-result} (namely Assumptions  \ref{ass1},    and  \ref{ass3}), for any complex valued test function $
G\in C_0^\infty([0,+\infty)\times \T\times\bbT)$ we have
 \begin{equation}
\label{042701-20bis}
\lim_{n\to+\infty}\int_0^{t} \frac 1{n}  \sum_{x\in\T_n}
                                                 G_x(s)
   {\mathbb E}_n\Big[\big(\tilde r_x^{(n)}(s) \big)^2-\big(\tilde p_x^{(n)}(s) \big)^2\Big]\dd s=0.
\end{equation}
We introduce
\begin{align}
\label{Vn}
\tilde  V_n(t,\eta,k)&:=
 \tilde   Y_n^+(t,\eta,k)+\tilde   Y_n^-(t,\eta,k),\\
\tilde  R_n(t,\eta,k)&:=
 \tilde   Y_n^+(t,\eta,k)-\tilde   Y_n^-(t,\eta,k).\notag
\end{align}
By the Parseval identity,
\begin{equation}
\label{032701-20}
\sum_{\eta\in\T_n}\underset{{k\in \widehat{\T}_n}}{{\hat\sum}}\big|\tilde  V_n(t,\eta,k)\big|^2= \frac{1}{2(n+1)}\sum_{x,x'\in\T_n}\left(\bbE_n\big[\tilde
  r_x^{(n)}(t)\tilde r_{x'}^{(n)}(t)\big]-\bbE_n\big[\tilde p_x^{(n)}(t)\tilde
  p_{x'}^{(n)}(t)\big]\right)^2
\end{equation}
and for any  $G\in C^\infty(\T)$, cf.~\eqref{disc-approx},
\begin{align*}
&\sum_{\eta\in\T_n}\underset{{k\in \widehat{\T}_n}}{{\hat\sum}}
 \tilde  V_n(t,\eta,k)\hat
  G^\star(\eta)=\frac{1}{n+1}\sum_{x\in\T_n}\mathbb E_n\left[
  \tilde  r_x^2-\tilde  p_x^2\right]
  G^\star_x.
\end{align*}
To prove Proposition \ref{equipartition}  we need to show that
 \begin{equation}
\label{040511-19}
\lim_{n\to+\infty}\sum_{\eta\in\T_n}\underset{{k\in \widehat{\T}_n}}{{\hat\sum}}\int_0^{+\infty}\tilde  V_n(t,\eta,k)\hat
G^\star(t,\eta,k)
\dd t=0
\end{equation}
 for any  $
G\in C_0^\infty([0,+\infty)\times \T\times\bbT)$.
From \eqref{eq:closed-syst2f} 
we obtain,
\begin{align}\label{010511-19a} 
&\partial_t \tilde R_n =\; n^2 (\delta_n \hat s)(\eta,k)\;
\left(\tilde W_n^+-\tilde W_n^-\right)-i n^2 (\sigma_n s) \; \tilde V_n \vphantom{\int} 
+ 2\gamma
n^2\tilde R_n\vphantom{\int} \\
&
\vphantom{\int}
+\frac{in^2}{n+1}\; \Big(\bb E_n\left[
    \left(\tilde p_n(t)-\tilde p_0(t)\right)\hat{\tilde{p}}\left(t,k+\tfrac{\eta}{n+1}\right)\right]+
\bb E_n\left[\left(\tilde p_n(t)-\tilde
  p_0(t)\right)\hat{\tilde{p}}^\star\left(t,k\right)\right]\Big)\notag\\
&
-\frac{i\tilde \ga n^2}{n+1}\; \sum_{x=0,n}\left\{e^{2\pi i x k}\bb E_n\left[
   \tilde  p_{x}(t)\hat{\tilde r}\left(t,k+\tfrac{\eta}{n+1}\right)
  \right]
+
e^{-2\pi i x \left(k+\frac{\eta}{n+1}\right)}\bb E_n\left[\tilde  p_{x}(t) \hat{\tilde r}^\star(t,k)\right]\right\},\notag
\end{align} 
where $\sigma_n s$ is given by \eqref{012308a} and 
$$
(\delta_n \hat s)(\eta,k):=\sin(2\pi
  k)-\sin\big(2\pi\big(k+\tfrac{\eta}{n+1}\big)\big).
$$
Given $s\in(0,1)$ we let
$$
\hat\T_{n,s}:=\big\{k\in\hat\T_n :\,0\le k\le (n+1)^{-s}\big\}\qquad\mbox{and}\qquad
\hat\T_n^s:=\hat \T_n\setminus \hat\T_{n,s}.
$$
 We  write
\begin{equation}
\label{050511-19}
\sum_{\eta\in\T_n}\underset{{k\in \widehat{\T}_n}}{{\hat\sum}}\int_0^t\tilde V_n(s,\eta,k) \hat G^\star(s,\eta,k)\dd
  s={\cal O}_{n,s}+{\cal O}_n^s,
\end{equation}
where terms ${\cal O}_{n,s}$ and ${\cal O}_n^s $ correspond to the summation in $k$
over $\hat\T_{n,s}$ and $\hat\T_n^s$, respectively, and $s\in(0,1)$ is to be determined later on.
Denoting $\hat G_1:=-\hat G/(\sigma_n s)$ and $\hat G_2:=\delta_n \hat s\; \hat G/(\sigma_n s)$, and using \eqref{010511-19a} we  write
\begin{equation}\label{eq:decompO} {\cal O}_{n}^s=\mathrm{I}_n+\mathrm{II}_n+\mathrm{III}_n+\mathrm{IV}_{n}+\mathrm{V}_{n},\end{equation}
where
\begin{align*}
\mathrm{I}_n&:=\frac{i}{n^2}\sum_{\eta\in\T_n}\underset{{k\in \widehat{\T}_n^s}}{{\hat\sum}}\int_0^t\partial_s \tilde R_n(s,\eta,k) \hat G^\star_1(s,\eta,k)\dd s,\notag\\
\mathrm{II}_n&:=i\sum_{\eta\in\T_n}\underset{{k\in \widehat{\T}_n^s}}{{\hat\sum}}\int_0^t\left(\tilde W_n^+-\tilde W_n^-\right)(s,\eta,k) \hat
  G^\star_2(s,\eta,k)\dd s,\\
\mathrm{III}_n&:=-2\gamma i\sum_{\eta\in\T_n}\underset{{k\in \widehat{\T}_n^s}}{{\hat\sum}}\int_0^t
\tilde R_n (s,\eta,k) \hat
  G^\star_1(s,\eta,k)\dd s,\notag\\
\mathrm{IV}_{n}&:=-\frac{1}{n+1}\sum_{\eta\in\T_n}\underset{{k\in \widehat{\T}_n^s}}{{\hat\sum}}\int_0^t
\vphantom{\int}
\Big\{\bb E_n\left[
    \left(\tilde p_n(s)-\tilde
  p_0(s)\right)\hat{\tilde{p}}\left(s,k+\tfrac{\eta}{n+1}\right)\right]\\
&
\qquad \qquad\qquad\qquad\qquad +\vphantom{\int}
\bb E_n\left[\left(\tilde p_n(s)-\tilde
  p_0(s)\right)\hat{\tilde{p}}^\star\left(s,k\right)\right]\Big\}\hat
  G_1^\star(s,\eta,k)\dd s,\\
\mathrm{V}_{n} & :=-\frac{1}{n+1}\sum_{x=0,n}\sum_{\eta\in\T_n}\underset{{k\in \widehat{\T}_n^s}}{{\hat\sum}}\int_0^t\; \Big\{e^{2\pi i x k}\bb E_n\left[
   \tilde  p_{x}(s)\hat{\tilde p}\left(s,k+\tfrac{\eta}{n+1}\right)
  \right]
\\
&
\qquad \qquad\qquad\qquad\qquad  +
e^{-2\pi i x \left(k+\frac{\eta}{n+1}\right)}\bb E_n\left[\tilde  p_{x}(s) \hat{\tilde p}^\star(s,k)\right]\Big\}\hat
  G_1^\star(s,\eta,k)\dd s.
\end{align*}

\subsubsection*{Estimates of ${\cal O}_{n,s}$}
We show that
for any $s\in(0,1)$ 
\begin{equation}
\label{012111-19}
\lim_{n\to+\infty}{\cal O}_{n,s}=0.
\end{equation}
Choose an arbitrary $\varepsilon>0$. Since $G\in
C_0^\infty([0,+\infty)\times[0,1]\times\T)$ we can find a sufficiently
large $M>0$ such that
\begin{equation}
\label{012111-19a}
\sum_{|\eta|\ge M}\sup_{k\in\T,s\in[0,t]}|\hat G(s,\eta,k)|<\eps.
\end{equation}
We  write
$
{\cal O}_{n,s}={\cal O}_{n,s,M}+{\cal O}_{n,s}^M,
$
where the terms ${\cal O}_{n,s,M}$ and ${\cal O}_{n,s}^M$ correspond in \eqref{050511-19}
to the summation over $|\eta|\le M$ and $|\eta|> M$, respectively.
Thanks to \eqref{012111-19} and \eqref{042111-19} we conclude that
\begin{equation}
\label{020602-20OM}
|{\cal O}_{n,s}^M|\le t\bigg\{\sup_{s\in[0,t],\eta\in\T_n}\underset{{k\in \widehat{\T}_n}}{{\hat\sum}}\big|\tilde V_n(s,\eta,k)\big|\bigg\}\bigg\{ \sum_{|\eta|> M}\sup_{k\in\T,s\in[0,t]}\big|\hat G(s,\eta,k)\big|\bigg\}<\frac{\varepsilon}{2},
\end{equation}
provided $M>0$ is sufficiently large. On the other hand, using the
Cauchy-Schwarz inequality and \eqref{020511-19} we get
\begin{align*}
|{\cal O}_{n,s,M}|&\le \|\hat G\|_\infty\int_0^t \sum_{|\eta|\le M}\underset{{k\in \widehat{\T}_{n,s}}}{{\hat\sum}}|\tilde V_n(s,\eta,k)| \dd
  s\\
&
\le \|\hat G\|_\infty t \sup_{s\in[0,t]}\bigg\{\sum_{|\eta|\le
  M}\underset{{k\in \widehat{\T}_{n,s}}}{{\hat\sum}}|\tilde
  V_n(s,\eta,k)|^2\bigg\}^{1/2} \frac{M^{1/2}}{n^{s/2}}\lesssim
  \frac{(M\log^2(n+1))^{1/2}}{n^{s/2}}.
\end{align*}
Therefore
$
\lim_{n\to+\infty}|{\cal O}_{n,s,M}|=0.
$
Combining with  \eqref{020602-20OM} we obtain \eqref{012111-19}.

\subsubsection*{Estimates of ${\cal O}_{n}^s$}

To estimate ${\cal O}_{n}^s$ we use the decomposition \eqref{eq:decompO}.
By integration by parts formula we get
$
\mathrm{I}_n=\mathrm{I}_{n,1}+\mathrm{I}_{n,2},
$
where
\begin{align*}
\mathrm{I}_{n,1}&:=-\frac{i}{n^2}\sum_{\eta\in\T_n}\underset{{k\in \widehat{\T}_n^s}}{{\hat\sum}}\int_0^{+\infty} \tilde R_n(t)
\partial_t \hat G^\star_1(t)
\dd t,\\
\mathrm{I}_{n,2}&:=\frac{i}{n^2}\sum_{\eta\in\T_n}\underset{{k\in \widehat{\T}_n^s}}{{\hat\sum}}\tilde R_n(0)\hat
G^\star_1(0).
\end{align*}
Using the Cauchy-Schwarz inequality and $\hat
G_1(t,\eta,k)\equiv 0$, for $t\ge t_*$,  we obtain
\begin{multline*}\bigg|\sum_{\eta\in\T_n}\underset{{k\in \widehat{\T}_n^s}}{{\hat\sum}}\int_0^{+\infty} \tilde R_n(t)
\partial_t \hat G^\star_1(t)\dd t\bigg|
\\
\le \bigg\{\sum_{\eta\in\T_n}\underset{{k\in \widehat{\T}_n}}{{\hat\sum}}\int_0^{t_*}\left|\tilde R_n(t)
  \right|^2\dd t\bigg\}^{1/2}\bigg\{\sum_{\eta\in\T_n}\underset{{k\in \widehat{\T}_n^s}}{{\hat\sum}}\int_0^{t_*}\left|\partial_s \hat G^\star_1(t)
  \right|^2\dd t\bigg\}^{1/2}.
\end{multline*}
Let
$
\phi(\eta):=1/(1+\eta^2).
$
We have
\begin{equation}
\label{012211-19}
\sum_{\eta\in\T_n}\underset{{k\in \widehat{\T}_n^s}}{{\hat\sum}}\int_0^t\left|\partial_s \hat G^\star_1(s)
  \right|^2
\le \sum_{\eta\in\T_n}\underset{{k\in \widehat{\T}_n^s}}{{\hat\sum}}\frac{\phi(\eta)}{\big(k^2+(k+\eta/n)^2\big)^2}\lesssim \underset{{k\in \widehat{\T}_n^s}}{{\hat\sum}}\frac{1}{k^4}
\lesssim  n^{3s}.
\end{equation}
Thanks to \eqref{020511-19} we conclude that
\begin{align*}
\frac{1}{n^2}\bigg|\sum_{\eta\in\T_n}\underset{{k\in \widehat{\T}_n^s}}{{\hat\sum}}\int_0^{+\infty} \tilde R_n(t)
\partial_t \hat G^\star_1(t)\bigg|
\lesssim n^{3s/2-3}\log( n+1),\qquad n\ge1.
\end{align*}
Thus, for any $s\in(0,2)$ we get 
$
\lim_{n\to+\infty}\mathrm{I}_{n,1}=0.
$
The argument to prove that $\lim_{n\to+\infty}\mathrm{I}_{n,2}=0$ is analogous.

Concerning $\mathrm{II}_n$,
by the Cauchy-Schwarz inequality we get
\begin{equation*}
|\mathrm{II}_n|\le \bigg\{\sum_{\eta\in\T_n}\underset{{k\in \widehat{\T}_n}}{{\hat\sum}}\int_0^t\left|\left(\tilde W_n^+-\tilde
  W_n^-\right)(s,\eta,k) \right|^2\dd s\bigg\}^{1/2} 
\bigg\{\sum_{\eta\in\T_n}\underset{{k\in \widehat{\T}_n^s}}{{\hat\sum}}\int_0^t |\hat G_2(s,\eta,k)|^2\dd s\bigg\}^{1/2} .
\end{equation*}
Using \eqref{020511-19} we  estimate the right hand side by an
expression of the form
\begin{multline*}
 \frac{\log (n+1)}{n}
\bigg\{\sum_{\eta\in\T_n}\underset{{k\in \widehat{\T}_n^s}}{{\hat\sum}}\int_0^{t_*} \left(\frac{\delta_n \hat s}{\sigma_n s}\right)^2|\hat G(s,\eta,k)|^2\dd s\bigg\}^{1/2} \\
\lesssim \frac{\log (n+1)}{n}
\bigg\{\sum_{\eta\in\T_n}\underset{{k\in \widehat{\T}_n^s}}{{\hat\sum}}\frac{1}{(1+\eta^2)k^2}\bigg\}^{1/2} \lesssim \frac{\log (n+1)}{n^{1-s/2}}\xrightarrow[n\to\infty]{} 0+, \quad \text{provided }s\in(0,2).
\end{multline*}

\subsubsection*{Estimates of $\mathrm{III}_n$}

By the Cauchy-Schwarz inequality we  write
\begin{align*}
|\mathrm{III}_n|&\le 2\gamma\sum_{\eta\in\T_n}\underset{{k\in \widehat{\T}_n^s}}{{\hat\sum}}\int_0^{t_*} \left|
\tilde R_n (s,\eta,k) \hat
  G^\star_1(s,\eta,k)\right|\dd s\\
&
\le  2\gamma\bigg\{\sum_{\eta\in\T_n}\underset{{k\in \widehat{\T}_n}}{{\hat\sum}}\int_0^{t_*} \left|
\tilde R_n (s,\eta,k)\right|^2\dd s\bigg\}^{1/2}\bigg\{\sum_{\eta\in\T_n}\underset{{k\in \widehat{\T}_n^s}}{{\hat\sum}}\int_0^{t_*} \frac{\left|
\hat
  G(s,\eta,k)\right|^2}{(\si_n s)^2}\dd s\bigg\}^{1/2}.
\end{align*}
Using \eqref{020511-19}, together with \eqref{012211-19}, we  estimate
the right hand side by
\begin{align*}
  \frac{Cn^{3s/2}\log (n+1)}{n}\xrightarrow[n\to\infty]{} 0+, \quad \text{provided }s\in(0,\tfrac23).
\end{align*}

\subsubsection*{Estimates of $\mathrm{IV}_{n}$ and $\mathrm{V}_n$}

The argument in both cases is the same, so we only consider $\mathrm{IV}_{n}$.
We  write $\mathrm{IV}_{n}=\mathrm{IV}_{n,1}+\mathrm{IV}_{n,2} $, where
\begin{align*}
\mathrm{IV}_{n,1}&:=-\frac{1}{n+1}\sum_{\eta\in\T_n}\underset{{k\in \widehat{\T}_n^s}}{{\hat\sum}}\int_0^t
\vphantom{\int}
\bb E_n\left[
    \left(\tilde p_n(s)-\tilde
  p_0(s)\right)\hat{\tilde{p}}\left(s,k+\tfrac{\eta}{n+1}\right)\right] \hat G_1^\star(s,\eta,k)\dd s,\\
\mathrm{IV}_{n,2}&:=-\frac{1}{n+1}\sum_{\eta\in\T_n}\underset{{k\in \widehat{\T}_n^s}}{{\hat\sum}}\int_0^t\vphantom{\int}
\bb E_n\left[\left(\tilde p_n(s)-\tilde
  p_0(s)\right)\hat{\tilde{p}}^\star\left(s,k\right)\right]\hat
  G_1^\star(s,\eta,k)\dd s.
\end{align*}
By the Plancherel identity we  write
\begin{equation*}
\mathrm{IV}_{n,2}=-\frac{1}{n+1}\sum_{x \in 
  \T_n}\int_0^t\vphantom{\int}
\bb E_n\Big[\left(\tilde p_n(s)-\tilde
  p_0(s)\right)\tilde{p}_x\left(s\right)\Big]\tilde
  G_{1,x}^\star(s)\dd s,
\end{equation*}
where
$$
\tilde
  G_{1,x}^\star(s):=\sum_{\eta\in\T_n}\underset{{k\in \widehat{\T}_n^s}}{{\hat\sum}}e^{2\pi i k x}\hat
  G_1^\star(s,\eta,k).
$$
As a result, invoking \eqref{022111-19}, the Cauchy-Schwarz inequality and
Plancherel identity,
we can find a constant $C>0$, independent of $n$, and such that
\begin{align*}
&
|\mathrm{IV}_{n,2}|\le \frac{1}{n+1}\sum_{x \in 
  \T_n}\int_0^t\vphantom{\int}
\left\{\bb E_n\big[|\tilde p_n(s)|+|\tilde
  p_0(s)|\big]^2 \right\}^{1/2}\left\{\bb E_n\left[\tilde{p}_x^2\left(s\right)\right]\right\}^{1/2}|\tilde
  G_{1,x}(s)|\dd s\\
&
\le \frac{1}{n+1}\vphantom{\int}
\left\{\int_0^t \bb E_n\big[|\tilde p_n(s)|+|\tilde
  p_0(s)|\big]^2\sum_{x \in 
  \T_n}\bb E_n\left[\tilde{p}_x^2\left(s\right)\right]\dd s\right\}^{1/2}\left\{\int_0^t\sum_{x \in 
  \T_n}|\tilde
  G_{1,x}(s)|^2 \dd s \right\}^{1/2}
\\
&
\le  \frac{C}{(n+1)^{1/2}}\vphantom{\int}
\left\{\int_0^t\bb E_n\left[\tilde p^2_n(s)+\tilde
  p_0^2(s)\right] \dd s\right\}^{1/2}\bigg\{\int_0^t \underset{{k\in \widehat{\T}_n^s}}{{\hat\sum}}\Big|\sum_{\eta\in\T_n}\hat
  G_1^\star(s,\eta,k)\Big|^2 \dd s \bigg\}^{1/2}.
\end{align*}
Invoking Lemma \ref{lem:bound2}, see \eqref{eq:ex-1}, we conclude that
\begin{equation*}
|\mathrm{IV}_{n,2}|
\lesssim   \frac{1}{(n+1)^{1/2}}\vphantom{\int}
\left\{\left( \underset{{k\in \widehat{\T}_n^s}}{{\hat\sum}}\frac{1}{k^4}\right)\left(\sum_{\eta\in\T_n}\frac{1}{1+\eta^2}\right)
  \right\}^{1/2}\lesssim \frac{n^{3s/2}}{(n+1)^{1/2}}\xrightarrow[n\to\infty]{} 0+, 
\end{equation*}
provided $s\in(0,\frac13)$. The proof of the fact that $\lim_{n\to+\infty}\mathrm{IV}_{n,1}$ follows the same lines as the argument presented
above. 

\subsubsection{Conclusion} Therefore, for $s \in (0,\frac13)$ we have proved that both ${\cal O}_{n,s}$ and ${\cal O}_n^s$
{tend to zero} as $n\to\infty$, and we conclude \eqref{040511-19}.

 This ends the proof of Proposition \ref{equipartition}.

\appendix

\section{Proof of (\ref{011412-19})}
\label{sec:app1}

Clearly
\begin{equation}
\label{021412-19}
|a_n(\eta)|\le \frac{2}{\eta^2-4},\quad \eta^2>8.
\end{equation}
Suppose that $\eta^2\in(0,8)$. 
Let $\Phi(u,k):=[4\sin^2(\pi  k)-u]^2+4\ga^2 u$. From \eqref{a-n} we see
that 
$$
|a_n(\eta)|\lesssim \underset{{k\in \widehat{\T}_n^s}}{{\hat\sum}}\frac{\sin^2(\pi k)}{\Phi(\eta^2,k)}.
$$ 
After a simple calculation one concludes that 
$
\ga^2\sin^4(\pi
  k)\lesssim \Phi(u,k)$ for $u\in(0,8)$, $k\in\T_n$.
Hence $|a_n(\eta)|\lesssim 1$, $\eta^2\in(0,8)$. This together with \eqref{021412-19} yield \eqref{011412-19}.

\section{Proof of (\ref{021612-19})}

\label{sec:app2}

Recall that
\begin{align}
\label{041612-19}
e_{{\rm d},n}(\eta)=\underset{{k\in \widehat{\T}_n}}{{\hat\sum}}\frac{i\eta +2\ga+2 \tilde\ga\sin^2(\pi
  k)}{-\eta^2+4\sin^2(\pi
  k)+2i\ga\eta}=\frac{\Xi(\eta,k)}{|\Theta(\eta,k)|^2},
\end{align}
where
\begin{align*}
&\Xi(\eta,k):
=\sin^2(\pi
  k)\big[-2 \eta^2\tilde\ga+8\ga+8 \tilde\ga\sin^2(\pi
  k)\big]+i\eta[-\eta^2+4\sin^2(\pi
  k) (1-\tilde\ga\ga)-4\ga^2],\\
&
\Theta(\eta,k):=-\eta^2+4\sin^2(\pi
  k)+2i\ga\eta.
\end{align*}
We have
\begin{equation}
\label{011812-19}
\frac{1}{|\Theta(\eta,k)|^2}\ge \frac{1}{4\ga^2\eta^2}.
\end{equation}
Let $\rho:=\sin^2(\pi
  k)$.
Let $\Gamma$ be the parabola in the $(x,y)$ plane described by the
system  of equations
\begin{align*}
&
f(\rho)=\rho\big[-2 \eta^2\tilde\ga+8\ga+8 \tilde\ga\rho\big],\\
&
g(\rho)=\eta[-\eta^2+4\rho (1-\tilde\ga\ga)-4\ga^2],\quad \rho\in\R.
\end{align*}
By a direct calculation one can  check that 
there exist two tangent  lines  to $\Gamma$ passing through $(0,0)$.
Hence  $(0,0)\not \in{\rm Conv}(\Gamma) $ - the closed region bounded
by  $\Gamma$. Denote by ${\rm d}_*>0$ the distance between 
$(0,0)$ and ${\rm Conv}(\Gamma) $ and $P_*$ the respective nearest neighbor
projection of $(0,0)$. 
Note that
\begin{equation*}
D^2(\rho):=f^2(\rho)+g^2(\rho)=\rho^2[-2 \eta^2\tilde\ga+8\ga+8
  \tilde\ga\rho]^2
+\eta^2[\eta^2 +4\ga^2-4\rho (1-\tilde\ga\ga)]^2.
\end{equation*}
When $\tilde\ga\ga\ge 1$ then 
$D^2(\rho)\ge 4\ga^2\eta^2.$
If, on the one hand, $\tilde\ga\ga <1$, then for
$
\rho \le \ga^2+\eta^2/4,
$ 
we have
$$
D(\rho)\ge \tilde\ga\ga |\eta|\left(\eta^2 +4\ga^2\right)^{1/2}\ge 2\tilde\ga\ga^2 |\eta|.
$$
If, on the other hand,
$
\rho > \ga^2+\eta^2/4,
$ 
then,
$$
-2 \eta^2\tilde\ga+8\ga+8
  \tilde\ga\rho\ge -2 \eta^2\tilde\ga+8\ga+2
  \tilde\ga (4\ga^2+\eta^2)\ge 8\ga+8
  \tilde\ga \ga^2>0.
$$
Therefore,
\begin{align*}
&D(\rho)\ge \rho\big|-2 \eta^2\tilde\ga+8\ga+8
  \tilde\ga\rho\big|\ge 2(\ga+
  \tilde\ga \ga^2)(4\ga^2+\eta^2)\ge 8\ga^3(1+
  \tilde\ga \ga).
\end{align*}
We conclude that
$$
{\rm Re}\,\Big(\Xi(\eta,k)\cdot  \frac{P^\star_*}{|P_*|}\Big)\ge {\rm d}_*\ge 2\ga^2\min\big\{4\ga(1+
  \tilde\ga \ga), \tilde\ga |\eta|\big\}.
$$
Therefore, by virtue of \eqref{041612-19} and \eqref{011812-19} we
conclude
 \begin{align}
\label{011712-19}
&|e_{{\rm d},n} (\eta)|
\ge 2\ga^2\; \frac{\min\{4\ga(1+
  \tilde\ga \ga), \tilde\ga |\eta|\}}{4\ga^2\eta^2}
\end{align}
and
\eqref{021612-19} follows.

\section{Proof of (\ref{022012-19})}

\label{sec:app3}

Note that
\begin{align*}
e_{{\rm s},n}(\eta) 
=1+4 \tilde\ga \ga\eta^2\underset{{k\in \widehat{\T}_n}}{{\hat\sum}}\frac{\cos^2(\pi k)}{|4\sin^2(\pi
  k)-\eta^2+2i\ga\eta|^2}
+2 \tilde\ga i\eta\underset{{k\in \widehat{\T}_n}}{{\hat\sum}}\frac{\cos^2(\pi k) [4\sin^2(\pi k)-\eta^2]}{|4\sin^2(\pi k)-\eta^2+2i\ga\eta|^2}
\end{align*}
and \eqref{022012-19} follows.
\qed

\section{Proof of (\ref{042712-19cc}) and  (\ref{042712-19c})}

\label{sec:app4}

We only prove  (\ref{042712-19cc}), as the argument for
(\ref{042712-19c}) follows the same lines.
It is clear from \eqref{rho-d} that 
\begin{align}
\label{042712-19}
 |\pi_{{\rm d},n}(\eta)|\lesssim \frac{1}{\eta^2+1},\quad \eta^2>8.
\end{align}
By an elementary calculation one gets
$$
\left\{\frac{1}{2}\left[(4\sin^2(\pi  k)-\eta^2)^2+4\ga^2 \eta^2\right]\right\}^{1/2} \ge
 \big(8\ga^2\sin^4(\pi k)+4\ga^2 \eta^2\big)^{1/2} 
$$
for $\eta^2<8$.
Hence,
$$
 |\pi_{{\rm d},n}(\eta)|\lesssim  \underset{{k\in \widehat{\T}_n}}{{\hat\sum}}\frac{|\sin(\pi
  k)|}{|\eta|+\sin^2(\pi k)}\lesssim \log\Big(1+\frac{1}{|\eta|}\Big),\qquad \eta^2<8
$$
and combining with \eqref{042712-19} we conclude (\ref{042712-19cc}).

\section{Proofs of (\ref{062712-19}) and (\ref{092712-19})}

\label{sec:app5}

Estimate
(\ref{092712-19}) follows straightforwardly from the definition
\eqref{062712-19d}
and assumption \eqref{spec-bound} on $\hat {\bar p}(0,k)$.

Concerning  (\ref{062712-19}) 
we estimate first in the case  $\eta^2\le 8$. 
 Hypothesis
\eqref{spec-bound}  allows us then to estimate
$|\rho_{{\rm d},n}(\eta)|\lesssim 1/|\eta|$ for $n\ge1$, cf.~\eqref{rho-d}. Combing with \eqref{022012-19}
 we conclude $|\rho_{{\rm d},n}(\eta)/e_{{\rm d},n}(\eta)|\lesssim1$. 

In the case  $\eta^2> 8$, using the fact that
$\underset{{k\in \widehat{\T}_n}}{{\hat\sum}}\widehat{\bar
    r}(0,k)=r_0=0$, we  write
\begin{align}
\label{052712-19a}
 \Big|\rho_{{\rm d},n}(\eta)\big|&=\bigg|\underset{{k\in \widehat{\T}_n}}{{\hat\sum}}\frac{(i\eta+2\ga)\widehat{\bar
    r}(0,k)}{4\sin^2(\pi k)-\eta^2+2i\ga \eta}-\underset{{k\in \widehat{\T}_n}}{{\hat\sum}}\frac{(i\eta+2\ga)\widehat{\bar
    r}(0,k)}{-\eta^2+2i\ga \eta}\bigg|\\
&
=\bigg|\underset{{k\in \widehat{\T}_n}}{{\hat\sum}}\frac{4\sin^2(\pi k) (i\eta+2\ga)\widehat{\bar
    r}(0,k)}{[4\sin^2(\pi k)-\eta^2+2i\ga \eta][-\eta^2+2i\ga
  \eta]}\bigg|\lesssim \frac{1}{1+|\eta|^3},\notag
\end{align}
and  (\ref{062712-19}) follows.

\section{Proof of (\ref{052312-19x})}


\label{sec:app6}

From the definition of $c_n$ we conclude that
\begin{equation}
\label{022312-19}
|c_n(\eta)|\lesssim \frac{1}{|\eta|+\eta^2},\quad \eta\in\R,\,n\ge1.
\end{equation}
By an elementary calculation one gets, for any $w\in \mathbb{C}\setminus[-1,1]$
\begin{equation}
\label{031412-19}
\underset{{k\in \widehat{\T}_n}}{{\hat\sum}}\frac{e^{2\pi i
    k}}{e^{4\pi i
    k}-2(1+w) e^{2\pi i
    k}+1}=\frac{(\Phi ^{n+1}_+(1+w)+1) \Phi _+(1+w)}{(\Phi
  ^2_+(1+w)-1)(1-\Phi ^{n+1}_+(1+w))},
\end{equation}
where
$\Phi_+:  \mathbb{C}\setminus[-1,1]\to \mathbb D\setminus\{0\}$ is the inverse to
the function
$$
J(z):=\frac{1}{2}\left(z+\frac{1}{z}\right) ,\quad z\in \mathbb D\setminus\{0\}.
$$
Here $\mathbb D:=\{z\in\mathbb{C}:|z|<1\}$.
Using the branch of the square root that maps $\mathbb{C}\setminus(-\infty,0]$
into $\mathbb{C}_+:=\{w:\,{\rm Re}\,w>0\}$ we  write
\begin{equation}
\label{phi+}
\Phi_+(w)=w-\sqrt{w^2-1},\qquad w\in \mathbb{C}_+\cap (\mathbb{C}\setminus[-1,1]).
\end{equation}
Using  \eqref{c-eta}  and \eqref{031412-19}
we
 write
\begin{align}
\label{c-eta1}
c_n(\eta)&=-\frac12-\frac{4-\eta^2+2i\ga \eta}{2}\underset{{k\in \widehat{\T}_n}}{{\hat\sum}}\frac{e^{2\pi ik}}{e^{4\pi i k}-2(1+w)e^{2\pi ik}+1}\notag\\
&
=-\frac12-\left(2-\frac{\eta^2}{2}+i\ga \eta\right)b_n(\eta),
\end{align}
where
$w(\eta):=-\eta^2/2+i\ga\eta$
and
$$
b_n(\eta):=\frac{(\Phi ^{n+1}_+(1+w)+1) \Phi _+(1+w)}{(\Phi ^2_+(1+w)-1)(1-\Phi ^{n+1}_+(1+w))}.
$$
Suppose now that $\eta^2\le \delta$ for a sufficiently small
$\delta>0$ so that
\begin{align*}
&\Big|\Phi_+\Big(1-\frac{\eta^2}{2}+i\ga\eta
  \Big)\Big|<1-\frac{1}{2}(\ga|\eta|)^{1/2},\quad \eta^2\le \delta.
\end{align*}
Let 
$
A_n:=\left\{\eta:\,\eta^2>2^{-1}\sin^2(\pi/n)\right\}.
$
 Then, for any $\eta\in A_n$ we have
$
\eta^2\ge 2/(\pi n)^2,
$
therefore, thanks to \eqref{phi+},
\begin{align*}
&\Big|\Phi_+\Big(1-\frac{\eta^2}{2}+i\ga\eta
  \Big)\Big|\le 1-\frac{1}{2}\Big(\frac{2^{1/2}\ga}{\pi n}\Big)^{1/2},\qquad |\eta|<\delta
\end{align*}
and
\begin{align*}
&\limsup_{n\to+\infty}\sup_{\eta\in A_n,|\eta|<\delta}\Big|\Phi_+\left(1-\frac{\eta^2}{2}+i\ga\eta
  \right)\Big|^{n+1}\le \limsup_{n\to+\infty}\bigg[1-\frac{1}{2}\Big(\frac{2^{1/2}\ga}{\pi
  n}\Big)^{1/2}\bigg]^{n+1}=0.
\end{align*}
Therefore, 
$$
|b_n(\eta)|\lesssim 
\frac{1}{|\eta|^{1/2}},\quad \mbox{for }\eta\in A_n,\quad n\ge1.
$$
Combining with \eqref{022312-19} we conclude that 
\eqref{052312-19x} holds for
$\eta^2\ge 2^{-1}\sin^2(\pi/n)$.
When $\eta^2< 2^{-1}\sin^2(\pi/n)$ we conclude directly from \eqref{c-eta} that
$$
|c_n(\eta)|\lesssim \underset{{k\in \widehat{\T}_n}}{{\hat\sum}}\frac{1}{4\sin^2(\pi k)+|\eta|
}\lesssim \int_0^1\frac{dk}{k^2+|\eta|}\lesssim \frac{1}{|\eta|^{1/2}}
$$
and \eqref{052312-19x} is also in force.

\bigskip

\bibliographystyle{amsalpha}

\begin{thebibliography}{A}

%
%
\bibitem{stefano} C. Bernardin, S. Olla, \emph{Non-equilibrium macroscopic
  dynamics of chains of anharmonic oscillators}, manuscript in
  preparation, available at \texttt{http://www.ceremade.dauphine.fr/olla/}.  

\bibitem{bo1} C. Bernardin, S. Olla, \emph{Fourier's Law for a Microscopic Model of Heat Conduction}, 
J. Stat. Phys. \textbf{121}, 271--289 (2005).   


\bibitem{bo2} C. Bernardin, S. Olla, \emph{Transport
Properties of a Chain of Anharmonic Oscillators with Random Flip of
Velocities}, J. Stat. Phys. \textbf{145}, 1224--1255 (2011).   


\bibitem{ffl} J. Fritz, T.  Funaki, J.L. Lebowitz,  \emph{ Stationary
  states of random Hamiltonian systems}, Probab. Theory Related
      Fields \textbf{99}(2),211--236 (1994).  

%

\bibitem{gant} F. R.  Gantmakher {\em The theory of matrices} vol. 1, AMS Chelsea, 2000.

\bibitem{iaco} A. Iacobucci, F.  Legoll, S. Olla, G. Stoltz, \emph{Negative thermal conductivity of chains of rotors with mechanical forcing}, Phys. Rev. E \textbf{84}, 061108 (2011). 

 \bibitem{jko} M. Jara, T. Komorowski, S. Olla, \emph{Superdiffusion
    of Energy in a system of harmonic oscillators with noise}, 
  Commun. Math. Phys. \textbf{339}, 407--453 (2015).


\bibitem{kelley} J. L. Kelley, {\em General topology}, Springer-Verlag, ISBN 978-0-387-90125-1, 1991.


\bibitem{KL} C. Kipnis and C. Landim,
{\it Scaling Limits of Interacting Particle Systems},
Springer-Verlag: Berlin, 1999.

\bibitem{kos1} T. Komorowski, S. Olla, M. Simon, \emph{ Macroscopic evolution of mechanical and thermal energy in a harmonic chain with random flip of velocities}, Kinetic and Related Models \textbf{11}(3), 615--645 (2018)

\bibitem{kos2}  T. Komorowski, S. Olla, M. Simon, \emph{An open microscopic model of heat conduction: evolution and non-equilibrium stationary states} (2019), to appear in Communications in Mathematical Sciences.


%
%

\bibitem{ls} J. Lukkarinen, H. Spohn, \emph{Kinetic limit for wave propagation in a random medium}, Arch. Ration. Mech. Anal. \textbf{183}(1), 93--162 (2006).


%
%
%
%

\bibitem{olla19} S. Olla, \emph{Role of conserved quantities in Fourier
    law for diffusive mechanical systems},
 Comptes Rendus Physique \textbf{20}, vol.5, 429--441 (2019).
 
%

\bibitem{spohn} H. Spohn, \emph{Large scale dynamics of interacting particles}.
  Springer, 1991.

\bibitem{strauss} W. Strauss, \emph{Partial Differential Equations}. John Wiley and Sons, 1992.

%
%

\end{thebibliography}

\end{document}